\documentclass[a4paper,12pt]{book}
\usepackage{amssymb,amsmath,amscd,dbnsymbnew,epsfig}
\usepackage{float}
\usepackage[all]{xy}

\restylefloat{figure}
\makeatletter

\renewcommand{\@makechapterhead}[1]{%
\vspace*{50\p@}%
{\parindent \z@\raggedright \normalfont\bfseries
\LARGE\thechapter\hspace{0.6cm} #1\par
\nobreak
\vskip 20\p@
}}
\renewcommand{\@makeschapterhead}[1]{%
\vspace*{50\p@}%
{\parindent \z@\raggedright \normalfont\bfseries
\LARGE #1\par
\nobreak
\vskip 20\p@
}}

\def\section{\@startsection{section}{1}{0mm}{1.5em}{1em}%
{\large\normalfont\bfseries}}
\def\subsection{\@startsection{subsection}{2}{0mm}{1em}{-2.7mm}%
{\normalfont\bfseries}}
\makeatother

\newtheorem{lemma}{Lemma}
\newtheorem{theorem}{Theorem}

\newtheorem{corollary}{Corollary}
\newtheorem{definition}{Definition}

\newenvironment{claim}{\vskip 0.08in\noindent{\bf Claim.\/}\begin{itshape}}
{\end{itshape}\vskip 0.08in}
\newenvironment{proof}{\noindent{\bf Proof.\/}}{$\Box$\vskip 0.08in}
\newenvironment{proofc}{\noindent{\bf Proof of the claim.\/}}
{$\Box$\vskip 0.08in}
\newenvironment{proofsk}{\noindent{\bf Sketch of the proof.\/}}
{$\Box$\vskip 0.08in}
\newenvironment{proofl}{\noindent{\bf Proof of the lemma.\/}}
{$\Box$\vskip 0.08in}
\newenvironment{proofle}[1]
{\noindent{\bf Proof of Lemma~\ref{#1}.\/}}{$\Box$\vskip 0.08in}
\newenvironment{proofth}[1]
{\noindent{\bf Proof of Theorem~\ref{#1}.\/}}{$\Box$\vskip 0.08in}
\newenvironment{remark}{\vskip 0.08in
\noindent{\bf Remark.\/}}{\vskip 0.08in}
\newenvironment{example}{\vskip 0.08in
\noindent{\bf Example.\/}}{\vskip 0.08in}
\newenvironment{notations}{\vskip 0.08in
\noindent{\bf Notations.\/}} {\vskip 0.08in}

\newenvironment{conjecture}{\vskip 0.08in
\noindent{\bf Conjecture.\/}} {\vskip 0.08in}

\newcommand{\hsp}{\hspace{0.09in}}
\newcommand{\vsp}{\vspace{0.05in}}

\newcommand{\al}{\alpha}
\newcommand{\be}{\beta}
\newcommand{\Ad}{A^{\bullet}}
\newcommand{\e}{\epsilon}
\newcommand{\D}{\Delta}
\newcommand{\g}{\gamma}
\newcommand{\F}{\mathcal{F}}
\newcommand{\G}{\Gamma}
\newcommand{\C}{\mathcal{C}}
\newcommand{\Co}{\overline{\C}}
\newcommand{\HC}{\mathcal{H}}
\newcommand{\HCo}{\overline{\HC}}
\newcommand{\OC}{\mathcal{O}}
\newcommand{\T}{\mathcal{T}}
\newcommand{\N}{\mathcal{N}}
\newcommand{\Q}{\mathbb{Q}}
\newcommand{\R}{\mathbb{R}}
\newcommand{\Z}{\mathbb{Z}}
\newcommand{\lk}{\operatorname{lk}}

\newcommand{\fk}{\mathbf{k}}
\newcommand{\fn}{\mathbf{n}}
\newcommand{\fs}{\mathbf{s}}
\newcommand{\s}{{\mathfrak s}}
\newcommand{\sltwo}{\mathfrak{sl}(2)}
\newcommand{\sln}{\mathfrak{sl}(n)}

\newcommand{\kernel}{\operatorname{ker}}
\newcommand{\image}{\operatorname{im}}
\newcommand{\MOD}{\operatorname{mod}}

\newcommand{\Ob}{\operatorname{Ob}}
\newcommand{\Mor}{\operatorname{Mor}}

\newcommand{\Cob}{{\mathcal Cob}}
\newcommand{\Cobd}{{\mathcal Cob}_\bullet}
\newcommand{\Cobi}{{\mathcal Cob}_{/i}}
\newcommand{\Cobfi}{{\mathcal Cob}_{f/i}}
\newcommand{\Cobdi}{{\mathcal Cob}_{\bullet/i}}

\newcommand{\Cobdl}{{\mathcal Cob}_{\bullet/l}}

\newcommand{\gCobdl}{\text{{\normalfont g}}{\mathcal Cob}_{\bullet/l}}
\newcommand{\Cobf}{{\mathcal Cob}_f}

\newcommand{\Mat}{\operatorname{Mat}}
\newcommand{\Kom}{\operatorname{Kom}}
\newcommand{\Komh}{\Kom_{/h}}
\newcommand{\Kobd}{\operatorname{Kob}}
\newcommand{\Kobdh}{\operatorname{Kob}_{/h}}
\newcommand{\Kobdpr}{{\operatorname{Kob}_{/\pm h}}}
\newcommand{\gKobd}{\operatorname{gKob}}
\newcommand{\gKobdh}{\operatorname{gKob}_{/h}}

\newcommand{\Kh}{\operatorname{Kh}}
\newcommand{\Lee}{\operatorname{Lee}}
\newcommand{\Id}{\operatorname{Id}}

\newcommand{\ls}{\left[\!\!\left[}
\newcommand{\rs}{\right]\!\!\right]}
\newcommand{\la}{\langle}
\newcommand{\ra}{\rangle}
\newcommand{\K}{{\mathcal K}}

\newcommand{\epsh}[2]{{\hspace{-3pt}\begin{array}{c}%
  \raisebox{-2.5pt}{\includegraphics[height=#1]{figs/#2.eps}}%
\end{array}\hspace{-3pt}}}

\title{Contributions to Khovanov Homology}
\author{Stephan M. Wehrli}
\date{Z\"urich, 2007}

\begin{document}
\newcounter{footnotebuffer}

\maketitle

\begin{center}
{\large\bf Abstract}
\end{center}

Khovanov homology ist a new link invariant, discovered
by M.~Khovanov~\cite{kh:first}, and used by
J.~Rasmussen~\cite{ra} to give a combinatorial proof
of the Milnor conjecture.
In this thesis, we give examples of mutant links
with different Khovanov homology.
We prove that Khovanov's chain complex
retracts to a subcomplex, whose generators are
related to spanning trees of the Tait graph,
and we exploit this result to investigate
the structure of Khovanov homology for alternating knots.
Further, we extend Rasmussen's invariant to links.
Finally, we generalize Khovanov's \cite{kh:colored}
categorifications of the
colored Jones polynomial, and study conditions
under which our categorifications are functorial
with respect to
colored framed
link cobordisms. In this context,
we develop a theory of Carter--Saito movie moves
for framed link cobordisms.

\tableofcontents
\pagestyle{headings}
\chapter*{Introduction}
\addcontentsline{toc}{chapter}{Introduction}
In his seminal paper \cite{kh:first}, M.~Khovanov introduced
a new invariant for oriented knots and links,
which
can be viewed as a ``categorification''
of the Jones polynomial \cite{jo}.
To a diagram $D$ of an oriented link $L\subset\R^3$,
Khovanov assigned a bigraded chain complex $\C^{i,j}(D)$
whose differential
is graded of bidegree $(1,0)$,
and whose homotopy type
depends only on the isotopy class of the oriented link $L$.
The graded Euler characteristic
$$
\chi_q(\C(D)):=\sum_{i,j}(-1)^iq^j\dim_\Q(\C^{i,j}(D)\otimes\Q)\,\in\Z[q,q^{-1}]
$$
is a suitably normalized version of the Jones polynomial of $L$:
$$
V(L)_{\sqrt{t}=-q}=\frac{\chi_q(\C(D))}{q+q^{-1}}
$$
The bigraded homology group $\HC^{i,j}(D)$ of the chain complex
$\C^{i,j}(D)$ provides
an invariant of oriented links, now known as Khovanov homology.
Because
Khovanov's construction is manifestly combinatorial,
Khovanov homology is algorithmically computable.

One of the remarkable properties of Khovanov homology is that
it fits into a topological quantum field theory of
$2$--knots in $4$--space.
Indeed, any smooth link cobordism
$S\subset\R^3\times [0,1]$ between two oriented
links $L_0\times\{0\}$ and $L_1\times\{1\}$
induces a chain transformation $\C(S):\C(D_0)\rightarrow\C(D_1)$,
which is a relative isotopy invariant of the cobordism $S$
when considered up to sign and homotopy.
Moreover, $\C(S)$ is graded of bidegree $(0,\chi(S))$,
where $\chi(S)$ denotes the Euler characteristic of
the surface $S$.

In \cite{le2}, E.~S.~Lee
modified Khovanov's construction by adding
additional terms to the differential.
On the basis of Lee's results,
J.~Rasmussen \cite{ra} defined a new
knot invariant $s(K)\in\Z$ and used it to give a
purely combinatorial proof of Milnor's conjecture on the
slice genus of torus knots. Previously, this conjecture had
been accessible only via Donaldson invariants, Seiberg--Witten
theory and knot Floer homology, and was considered
as a main application of these theories.
In many ways, Khovanov homology appears to
be an algebro--combinatorial replacement for
gauge theory and Heegaard Floer homology.
An explicit relation between reduced Khovanov
homology with coefficients in $\Z/2\Z$
and Heegaard Floer homology of
branched double--covers of the $3$--sphere,
in the form of a spectral sequence, was
discovered by P.~Ozsv\'ath and Z.~Szab\'o \cite{os:sequence}.

In the past few years,
several new link homology theories have emerged.
Among these
are the Khovanov--Rozansky theories
for the $\sln$ polynomials
and the HOMFLY--PT polynomial \cite{kr1,kr2},
and two categorifications of the
colored Jones polynomial,
proposed by Khovanov \cite{kh:colored}.
Moreover,
D.~Bar--Natan \cite{ba:tangles}
discovered a ``formal Khovanov bracket'',
which generalizes both Khovanov homology
and Lee homology, and which extends naturally
to tangles.

This thesis is devoted to the study of structural properties
of Khovanov homology, as well as to the generalization
of Rasmussen's invariant and its applications, and
contains contributions towards a $4$--dimensional
lift of Khovanov's theory for the colored Jones polynomial.

In Chapter~1 we review the definition of the
formal Khovanov bracket and discuss its
relation with Khovanov homology and Lee homology.

Chapter~2 deals with Rasmussen's invariant.
We give a new proof of a theorem of E.~S.~Lee \cite{le2},
which states that the Lee homology
of an $n$--component link has dimension $2^n$.
Then we extend Rasmussen's knot invariant to links,
and give examples where this invariant is
a stronger obstruction to sliceness than
the multivariable Levine--Tristram signature.

In Chapter~3,
we study the behavior of Khovanov homology under
Conway mutation. Conway mutation
is a procedure for modifying links,
which was invented by J.~Conway \cite{co}.
We present infinitely many examples
of mutant links with different Khovanov homology.
The existence of such examples is remarkable
since many classical invariants, such
as the HOMFLY--PT polynomial,
the knot signature and the
hyperbolic volume of the knot complement,
are unable to detect Conway mutation.
In particular, our examples show that
Khovanov homology is strictly stronger
than the Jones polynomial.

In \cite{ba:first}, Bar--Natan computed the ranks
of the Khovanov homology groups for all prime
knots with up to 11 crossings.
One of his surprising experimental results is
that the ranks of the Khovanov homology groups
tend to be much smaller than the ranks of the chain
groups. In Chapter~4 we give an explanation for
this phenomenon:
we prove that the complex $\C^{i,j}(D)$ retracts to
a subcomplex, whose generators are in $2:1$ correspondence
with the spanning trees of the Tait graph of $D$.
Using this result, we
give a new proof of a theorem of Lee \cite{le1},
which states that the non--trivial homology groups $\HC^{i,j}(K)$
of an alternating knot $K$ are concentrated on
two straight lines in the $ij$--plane.
Our spanning tree model has applications to Legendrian
knots (cf. \cite{wu}),
and it is of theoretical interest because spanning
trees also appear in the context of knot Floer homology \cite{os}.

Chapter~5 is purely topological.
We investigate link cobordisms equipped with
a framing, i.e. with a relative homotopy class
of non--singular normal vector fields.
The most important part of Chapter~5 is the last
section, where we give a list of movie moves for
movie presentations of framed link cobordisms.
Framed movie moves are needed if one wishes to
establish functoriality of colored Khovanov invariants
\cite{kh:colored} with respect to framed link cobordisms.

In Chapter~6, we focus on Khovanov's \cite{kh:colored}
categorification
of the non--reduced colored Jones polynomial.
By reformulating Khovanov's construction
in Bar--Natan's setting, we obtain a
``colored Khovanov bracket''.
We prove that the colored Khovanov bracket is well--defined over
integer coefficients. Moreover, we introduce a family of modified
colored Khovanov brackets,
and study conditions under which our modified
theories are functorial with respect to colored
framed link cobordisms.
Lifting the colored Jones polynomial to a functor
can be seen as a first step into the direction of
categorification of the $\sltwo$ quantum invariant
for $3$--manifolds, and might ultimately
lead to an intrinsically $3$-- or $4$--dimensional
understanding of Khovanov homology.

The material of Chapter~1 is taken from
\cite{ba:tangles}, \cite{ba:computations},
\cite{kh:first}, \cite{kh:frobenius}, \cite{le2}
and \cite{we:trees}.
Chapters 2, 5 and 6 contain the results
of my joint paper with A.~Beliakova \cite{bw},
and Chapters 3 and 4 are taken from
\cite{we:mutation} and \cite{we:trees}.

\section*{Acknowledgements}
First and foremost, I would like to thank my supervisors Anna
Beliakova and Norbert A'Campo for their constant support
and encouragement, and for their readiness to share their
advice and expertise with me. I would also like to thank
Sebastian Baader, Mikhail Khovanov and Alexander Shumakovitch for
their interest in this work and for many valuable discussions.
%
%for their constant encouragement.
%for supervising and reviewing this
%thesis and for many valuable discussions.
%I would especially like to thank
%Anna Beliakova
%for sharing her expertise and advice,
%for her excellent support,
%for our joint research work,
%and for her constant encouragement.
%Norbert A'Campo
%has profoundly influenced my mathematical
%way of thinking and kept broadening my knowledge
%in many insightful conversations.
%A very special thank goes to
%Mikhail Khovanov, for all his support, for many helpful
%hints, and for taking the time to read and review this
%thesis.
%Among the people who taught me mathematics and physics,
%I would like to single out Alexander Shumakovitch
%who initiated my studies in knot theory.
%I would also like to thank Alexander Shumakovitch, whose
%computer program 
%
%and for making his symbol font \texttt{dbnsymb} and the
%perl script \texttt{makefont} freely available.
Dror Bar--Natan's symbol font \texttt{dbnsymb}
was used throughout this thesis. The material covered
in Chapters 2, 5 and 6 is
taken from my joint work with Anna Beliakova.
During the work on this thesis, I was partially supported
by the Swiss National Science Foundation.

\newpage\thispagestyle{empty}
%%%%%%%%%%%%%%%%% Bar--Natan's Khovanov bracket %%%%%%%%%%%%%%%%%%%%%%%

\chapter{Khovanov homology}\label{cKhovanovhomology}
In this chapter, we first recall
basic concepts of knot theory.
Then we give the definitions of the
Jones polynomial and the formal Khovanov bracket,
and discuss how Khovanov homology
and Lee homology can be recovered from
the Khovanov bracket by applying
a TQFT.

\section{Links and link cobordisms}\label{slinks}
\noindent
A {\it link} in $\R^3$
is a finite collection of disjoint
circles which are smoothly embedded into $\R^3$. These circles
are called the {\it components}
of the link. If an orientation of the components is specified,
we say that the link is {\it oriented}.
For an oriented link $L$, we denote by $-L$ the same link
but with reversed orientations. A link consisting of
only one component is called a {\it knot}.
\begin{figure}[H]
\centerline{\psfig{figure=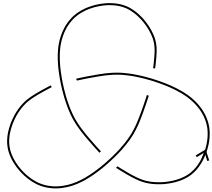,height=2.5cm}}
\caption{An oriented link diagram.}\label{flinkdiagram}
\end{figure}
To present links, one uses pictures
such as the one in Figure~\ref{flinkdiagram},
called {\it link diagrams}.
Given an oriented link diagram $D$,
we denote by $c_+(D)$ and $c_-(D)$ the numbers
its {\it positive} ($\overcrossing$) and {\it negative}
($\undercrossing$) crossings,
and by $w(D):=c_+(D)-c_-(D)$
the {\it writhe} of $D$.
(E.g. in the above figure we have $w(D)=-c_-(D)=-3$
and $c_+(D)=0$).

It is known that two link diagrams represent isotopic links
if and only if they are related
by a finite sequence of
the following local modifications, called
{\it Reidemeister moves}.

\begin{figure}[H]
\centerline{
\begin{tabular}{c@{\qquad\qquad}c@{\qquad\qquad}c}
\psfig{figure=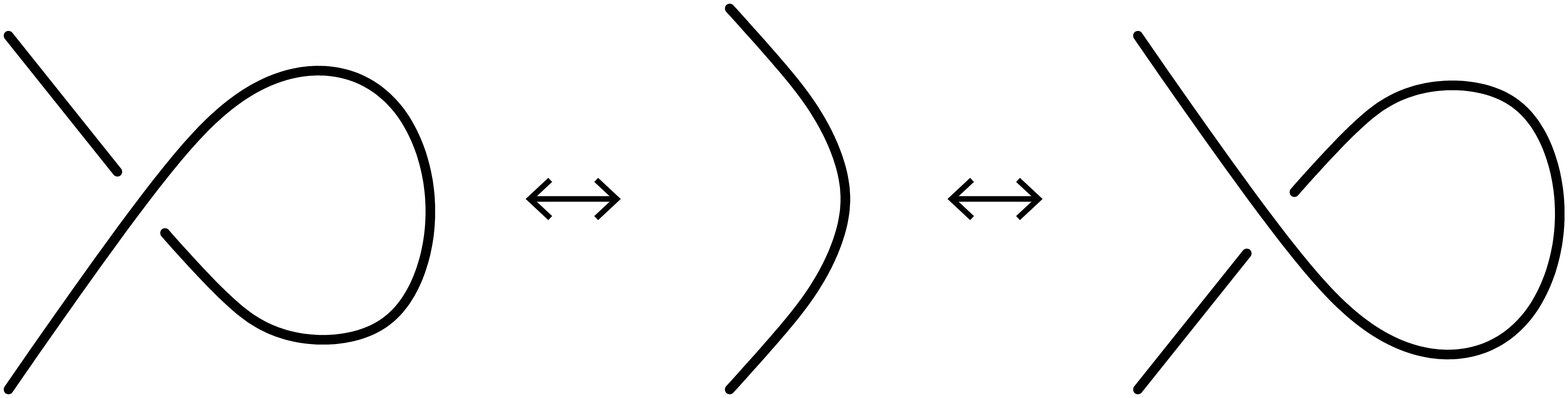,height=0.9cm}&
\psfig{figure=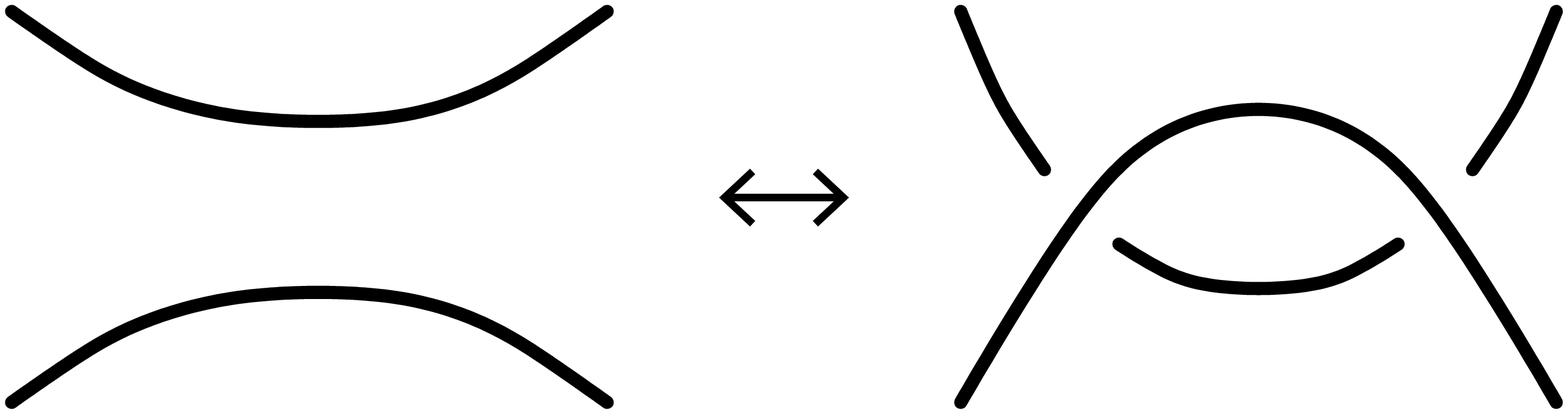,height=0.8cm}&
\psfig{figure=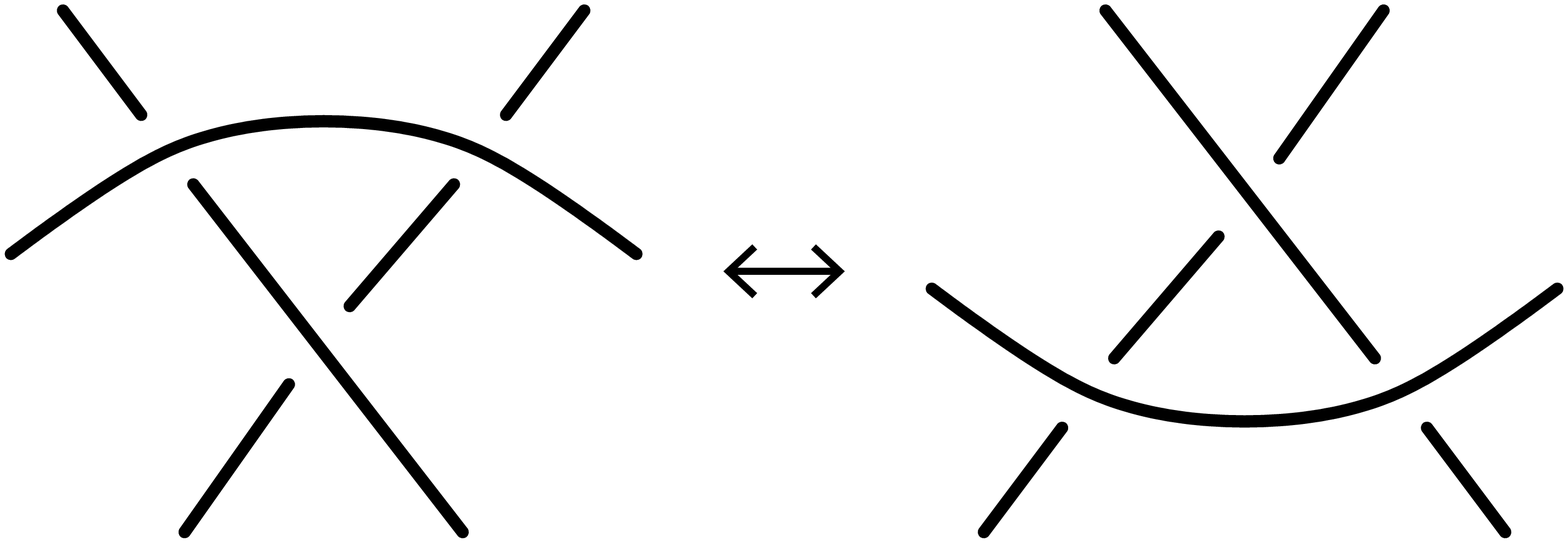,height=0.9cm}
\end{tabular}
}
\caption{The three Reidemeister moves R1, R2 and R3.}
\end{figure}

To classify links up to isotopy, one usually uses {\it link invariants},
i.e. functions whose domain is the set of
links in $\R^3$ and whose value depends only on the isotopy class
of a link. One way  of constructing a link invariant is by
defining it on the level of link diagrams and then showing
that it is invariant under Reidemeister moves.

A {\it cobordism} between two oriented links
$L_0$ and $L_1$ is a compact oriented surface
smoothly embedded in $\R^3\times[0,1]$
whose boundary lies entirely in $\R^3\times\{0,1\}$ and whose ``bottom''
boundary is $-L_0\times\{0\}$ and whose ``top'' boundary is $L_1\times\{1\}$.
For technical reasons, we assume that the surface
is in general position with respect to the projection
onto the last coordinate of $\R^3\times[0,1]$,
and parallel to $[0,1]$ near the boundary.
It convenient to view
the last coordinate of $\R^3\times[0,1]$
as {\it time coordinate}.

Assume $S\subset\R^3\times[0,1]$ is a link cobordism.
By cutting $S$ along hyperplanes
$\R^3\times\{t_i\}$, $0=t_0<t_1<\ldots<t_n=1$,
we can split $S$
into elementary pieces, such that
each piece $S\cap\R^3\times [t_{i-1},t_i]$
contains at most one critical point with
respect to the time coordinate,
and such that all $t_i$ are regular values.
Projecting the oriented links
$L_{t_i}:=S\cap (\R^3\times\{t_i\})$
down to the plane, we obtain
a sequence of oriented link diagrams $D_{t_i}$.
Altering the $t_i$,
we can assume that any two
consecutive diagrams
differ by one of the following transformations: a planar isotopy,
a Reidemeister move, or one of the Morse moves shown
in Figure~\ref{fcapcupsaddle}. %(see \cite[Section~5]{cs}).
In this case, the sequence $\{D_{t_i}\}$
is called a {\it movie presentation} for $S$, and
the individual diagrams $D_{t_i}$ are called the {\it stills}
of the movie presentation.
\begin{figure}[H]
\centerline{\psfig{figure=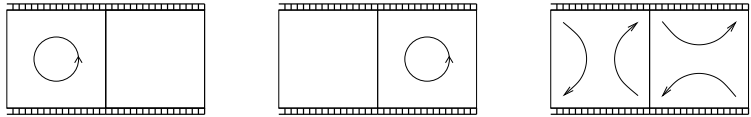,height=1.5cm}}
\caption{Morse moves corresponding to cap, cup and saddle cobordism.}
\label{fcapcupsaddle}
\end{figure}
\noindent
\begin{theorem}[\cite{cs}]\label{tmovie}
1.~Every link cobordism has a movie presentation.
2.~Two movies represent isotopic link cobordisms if and only
if they can be transformed into each other by a
finite sequence of Carter--Saito movie moves
(and by time--reordering different parts of a movie
which ``happen'' at different places).
\end{theorem}
The Carter--Saito movie moves are shown in
Figures~\ref{ftypeI}, \ref{ftypeII} and
\ref{ftypeIII}.
The moves of Type I and II
consist in replacing the
circular movies of Figures~\ref{ftypeI} and \ref{ftypeII}
by identity movies,
i.e. by movies where all stills look the same.
\newpage
\begin{figure}[H]
\centerline{\psfig{figure=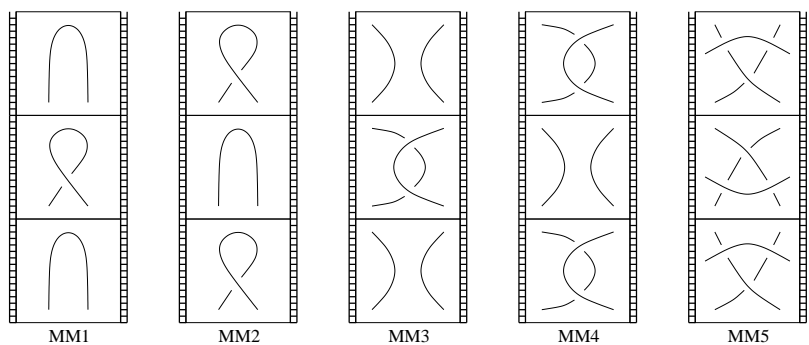,height=4.4cm}}
%\begin{center}
\vsp\par

\caption{Type I moves.}\label{ftypeI}
\end{figure}

\begin{figure}[H]
\centerline{\psfig{figure=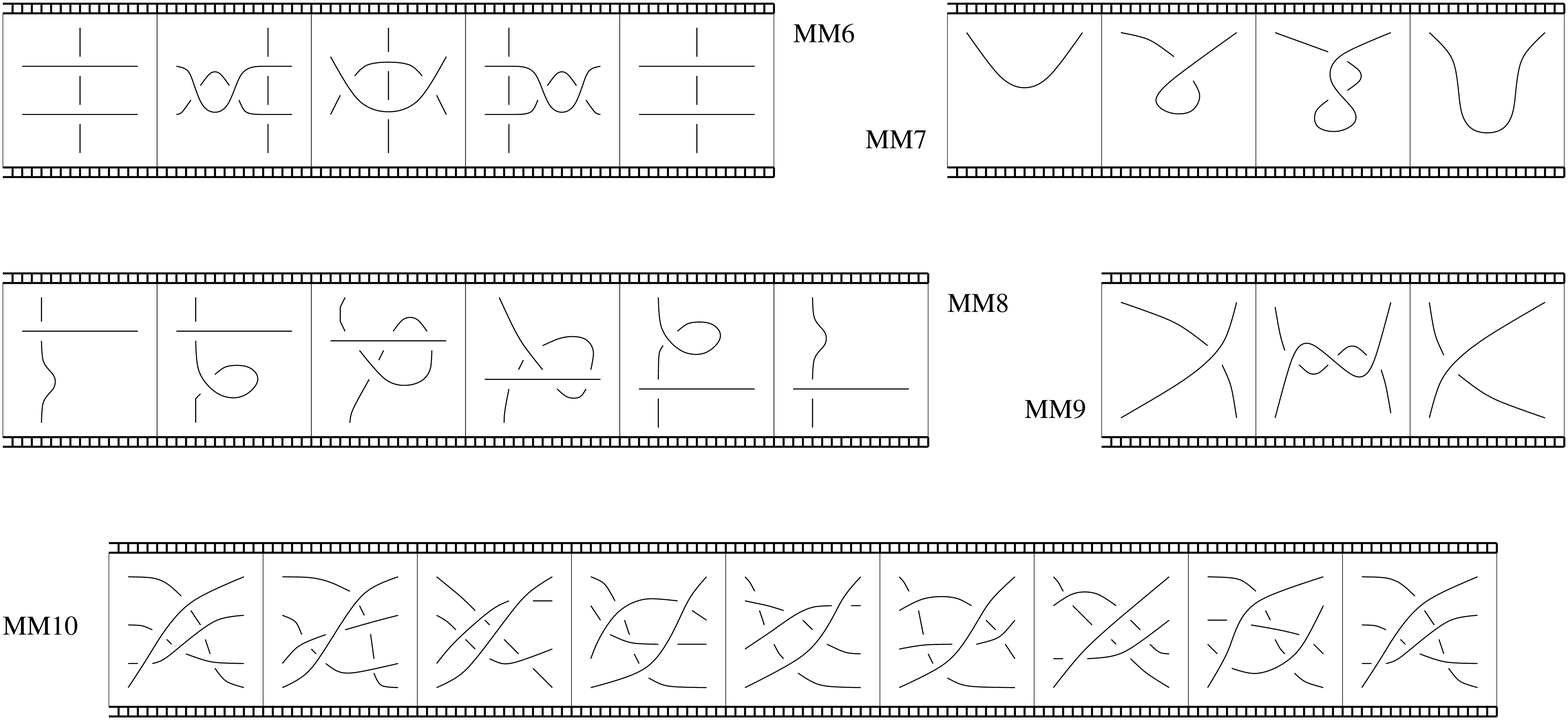,height=5.9cm}}
\vsp\par

\caption{Type II moves.}\label{ftypeII}
%\end{center}
%\vsp\\
\end{figure}

\begin{figure}[H]
\centerline{\psfig{figure=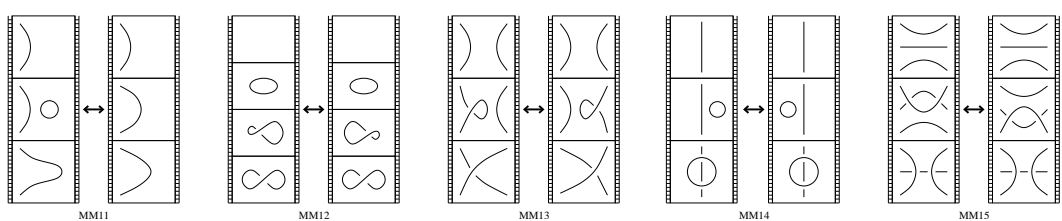,height=2.7cm}}
\vsp\par

\caption{Type III moves.}\label{ftypeIII}
\end{figure}

%To conclude this section,
As it will be needed in Chapter~\ref{cframedlinkcobordisms},
we briefly recall
the definition of the linking number.
Let $L=K_1\cup K_2$ be an oriented $2$--component
link with corresponding diagram $D=D_1\cup D_2$.
Let $c_+'(D)$ and $c_-'(D)$ denote the numbers
of positive and negative crossings
at which $D_1$ and $D_2$ cross.
Note that $c_+'(D)$ and $c_-'(D)$ have
the same parity, because $D_1$ and $D_2$
necessarily cross in an even number of crossings.

The {\it linking number} of $K_1$ and $K_2$
is defined by
$$
\lk(K_1,K_2):=(c_+'(D)-c_-'(D))/2\, \in\Z\; .
$$
It is easy to see that $\lk(K_1,K_2)$
is invariant under Reidemeister moves and
hence determines an invariant of the link $L$.
Geometrically, the linking number
can be interpreted as the algebraic intersection number
of two generic compact oriented
surfaces $S_1,S_2\subset\R^3\times(\infty,0]$
satisfying
$\partial S_i=K_i\times\{0\}\subset\R^3\times\{0\}$.

\section[Kauffman bracket and Jones polynomial]%
{The Kauffman bracket and the Jones polynomial}
\label{sJones}
\noindent
The Jones polynomial is an invariant for oriented links which
was introduced by V.~Jones \cite{jo} in the year 1984.
In \cite{ka:bracket}, L.~Kauffman described an
elementary approach to the Jones polynomial using a state sum.
In this section, we recall
Kauffman's definition of
the Jones polynomial.
We use the normalization conventions of \cite{kh:first}.

Let $D$ be an unoriented link diagram.
A {\it Kauffman state} of $D$ is a diagram obtained
by replacing each crossing $\slashoverback$ of $D$ with
$\smoothing$ or $\hsmoothing$ (so that the result
is a disjoint union of circles embedded in the plane).
We denote by $\K(D)$ the set of all Kauffman states of $D$,
and by $n(D')$ the number of circles in $D'\in\K(D)$.
If $D$ has $c$ crossings, the number of Kauffman states
is $2^c$.
Given a crossing of $D$ (looking like this: $\slashoverback$),
we call $\smoothing$ its $0$--{\it smoothing} and $\hsmoothing$
its $1$--{\it smoothing}. We denote by $r(D,D')$
the number of $1$--smoothings in $D'$, where here
$D'$ can be a Kauffman state of $D$ or more generally
any link diagram obtained from $D$ by smoothing some
of the crossings while leaving the others unchanged.

The {\it Kauffman bracket} of $D$ is the Laurent polynomial
$\la D \ra\in\Z[q,q^{-1}]$ defined by
\begin{equation}\label{fKauffmandef}
\la D\ra:=\sum_{D'\in\K(D)}(-q)^{r(D,D')}(q+q^{-1})^{n(D')}.
\end{equation}
For example, the Kauffman bracket of a crossingless
diagram $D=\bigcirc^n$ consisting of $n$ disjoint circles
is just $\la \bigcirc^n\ra=(q+q^{-1})^n$.
Setting $\la D|D'\ra:=(-q)^{r(D,D')}$, we can rewrite
the above formula as
\begin{equation}\label{fKauffmanrewrite}
\la D\ra:=\sum_{D'\in\K(D)}\la D|D'\ra\la D'\ra\; .
\end{equation}
It is easy to see that the Kauffman bracket satisfies
the following rules:
\begin{eqnarray}
\label{fempty}\la \emptyset\ra &=& 1,\\
\label{fcirc}\la D\sqcup\bigcirc\ra &=& (q+q^{-1})\la D\ra,\\
\label{fskein}\la\slashoverback\ra &=&\la\smoothing\ra-q\la\hsmoothing\ra.
\end{eqnarray}
Rule \eqref{fempty} says that the empty link evaluates to $1$.
Rule \eqref{fcirc} says that $\la D\ra$ is multiplied by $(q+q^{-1})$ when
a disjoint circular component is added to $D$. In the third rule,
the three pictures $\slashoverback$, $\smoothing$ and $\hsmoothing$ stand for
three link diagrams which are identical except in a small disk,
where they look like $\slashoverback$, $\smoothing$ and $\hsmoothing$,
respectively.
The above rules determine the Kauffman bracket
completely (see Chapter~\ref{cspanningtreemodel}).

Using \eqref{fcirc} and \eqref{fskein} one can prove
the following lemma which shows that the Kauffman bracket is invariant under
Reidemeister moves when considered up to multiplication with
a unit of the ring $\Z[q,q^{-1}]$.
\begin{lemma}\label{lbracket}
The Kauffman bracket satisfies
\begin{enumerate}
\item
$\la \epsh{12pt}{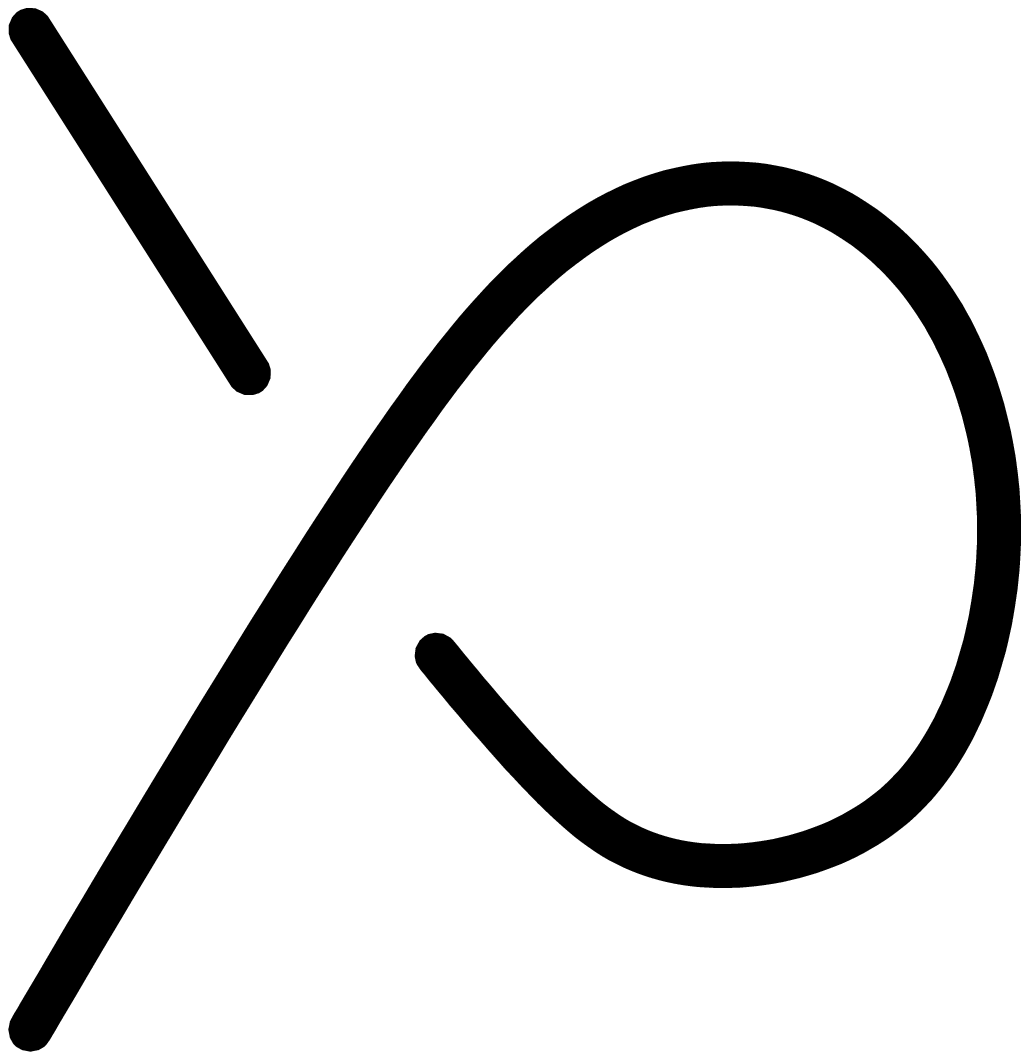} \ra=q^{-1}\la \epsh{12pt}{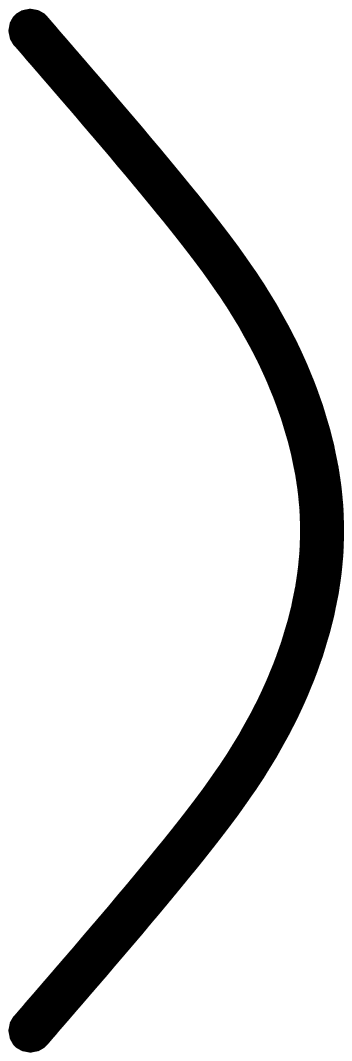} \ra$
and $\la \epsh{12pt}{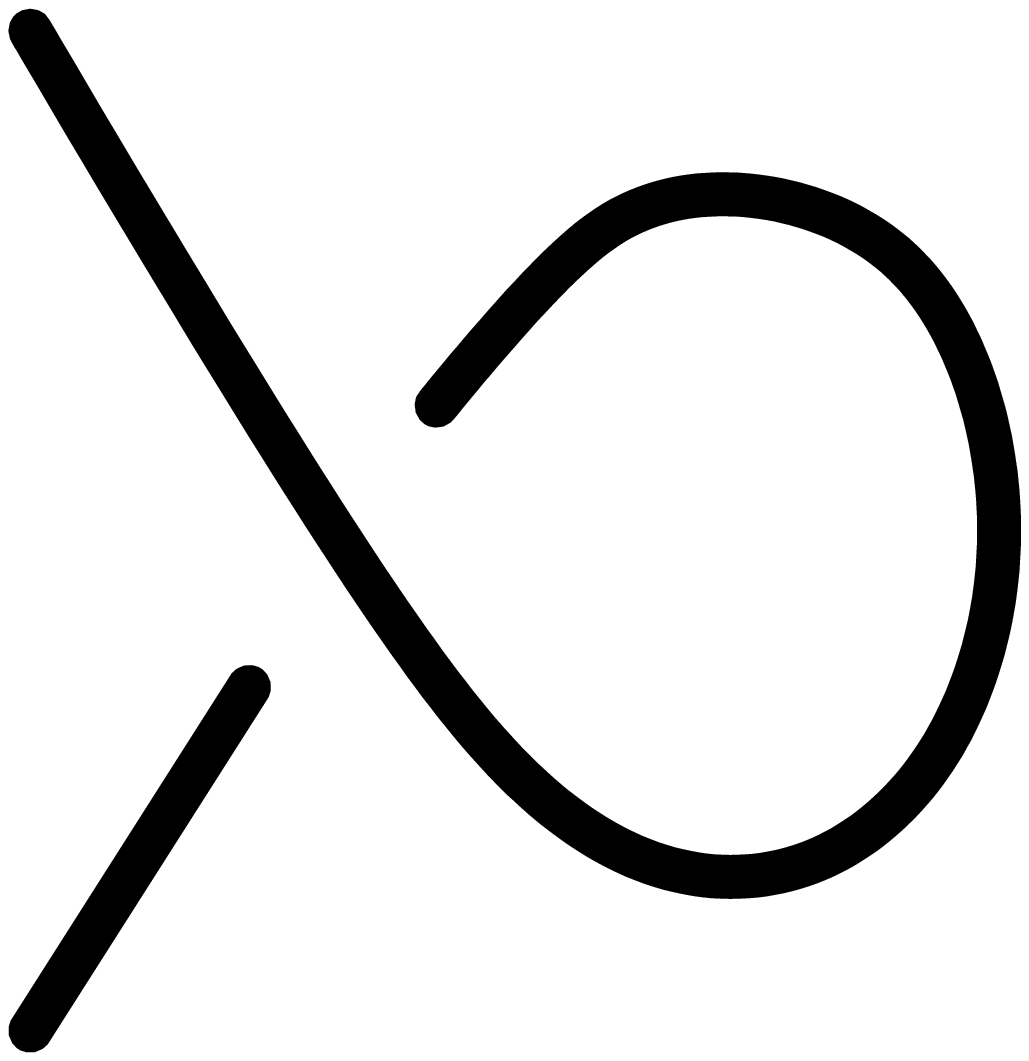} \ra=-q^2\la \epsh{12pt}{noRI} \ra$. \item
$\la \epsh{12pt}{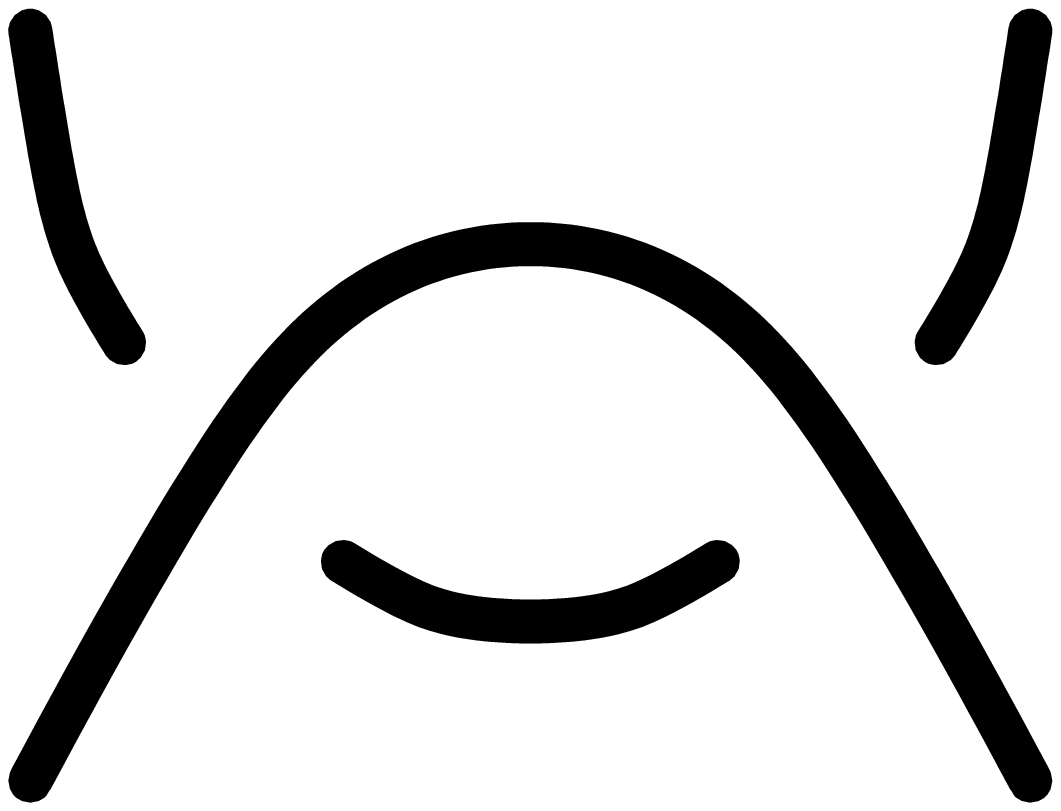} \ra=-q\la \epsh{12pt}{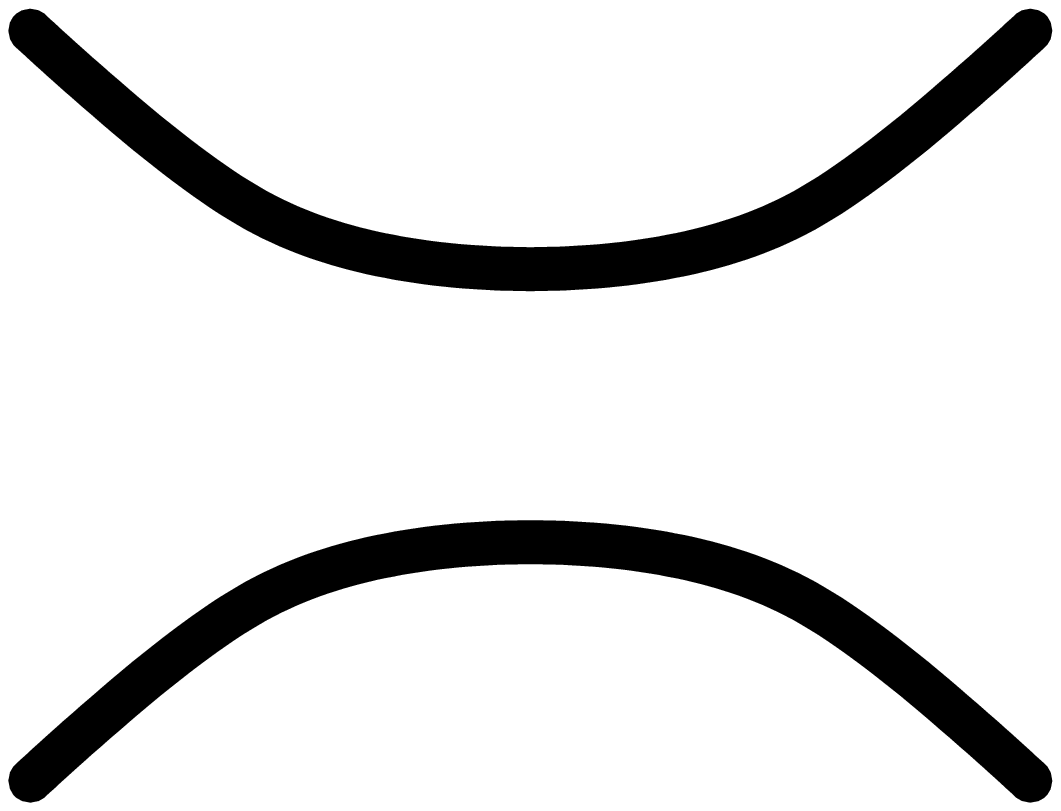} \ra$.
\item
$\la \epsh{14pt}{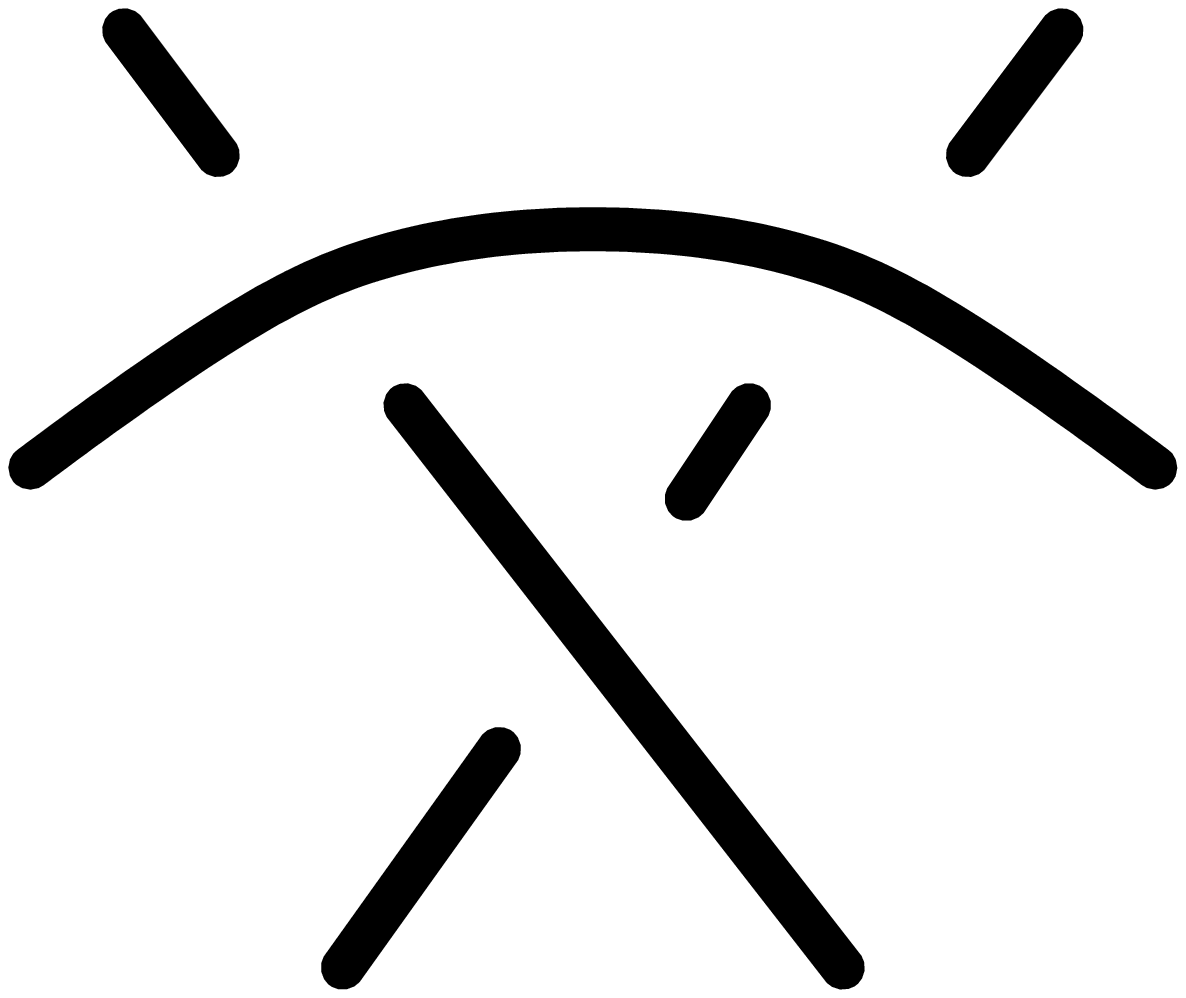}\ra=\la \epsh{14pt}{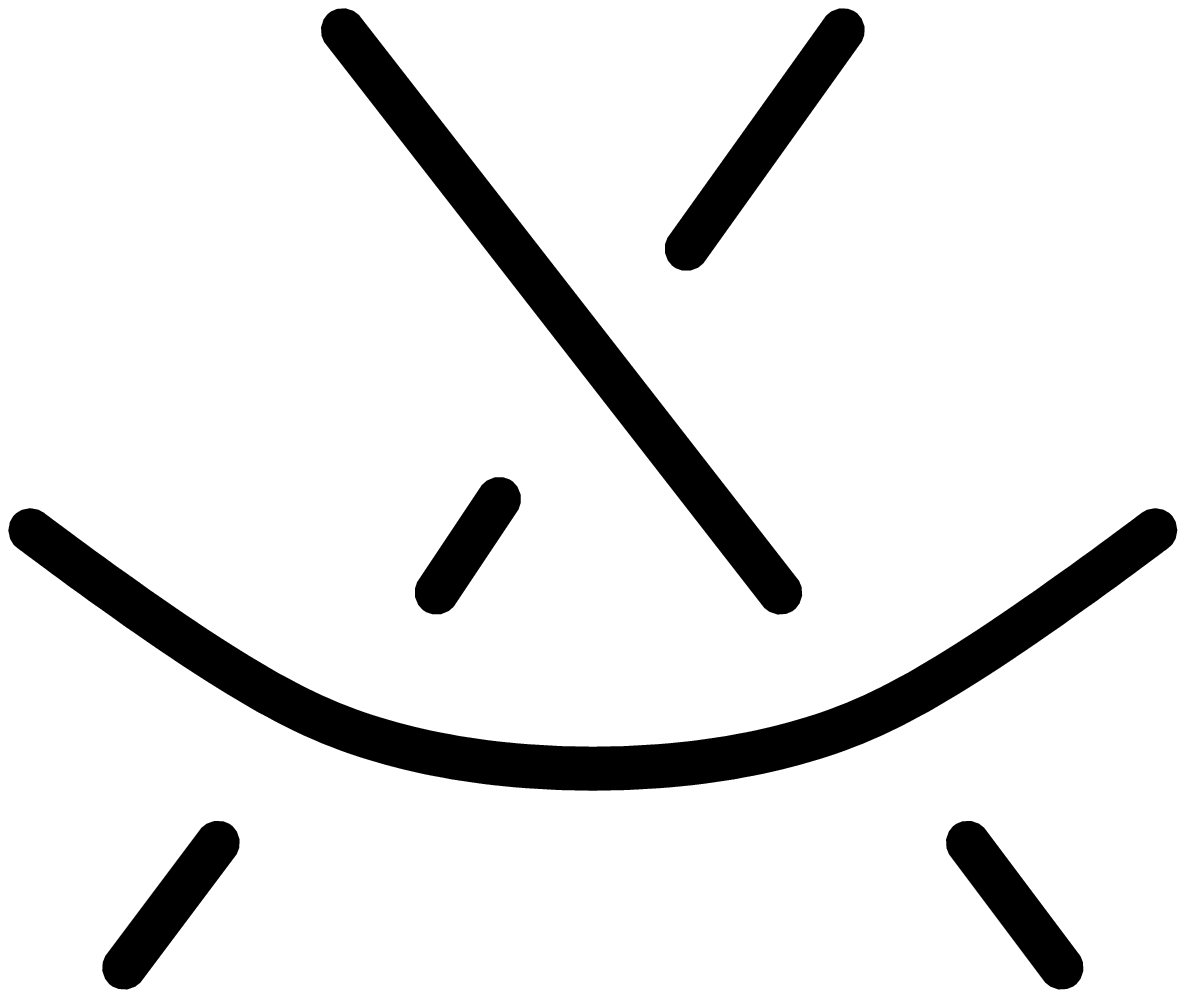}\ra$.
\end{enumerate}
\end{lemma}

If $D$ is the diagram of an oriented link $L$,
we can define
%the quantity
\begin{equation}\label{fJonesdef}
J(D):=(-1)^{c_-(D)}q^{c_+(D)-2c_-(D)}\la D\ra.
\end{equation}
Lemma~\ref{lbracket} implies that $J(D)$ is invariant under Reidemeister moves
and hence an invariant of the link $L$. We denote this invariant
by $J(L)$ and call it the {\it Jones polynomial}
of $L$. The normalization of $J(L)$ is chosen
so that
\begin{equation}\label{fJonesnorm}
J(\emptyset)=1\qquad\text{and}\qquad J(\bigcirc)=q+q^{-1}.
\end{equation}
A triple $(L_+,L_-,L_0)$ of oriented links is called a {\it skein
triple} if the oriented links $L_+$, $L_-$ and $L_0$ possess diagrams which
are mutually identical except in a small disc, where they look like
$\overcrossing$, $\undercrossing$ and $\orsmoothing$, respectively.
Using rule \eqref{fskein}, it is easy to see that the Jones polynomial satisfies
\begin{equation}\label{fJonesskein}
q^{-2}J(L_+)-q^2J(L_-)=(q^{-1}-q)J(L_0)
\end{equation}
for any skein triple $(L_+,L_-,L_0)$. It is known
that the Jones polynomial is determined uniquely by
relations \eqref{fJonesnorm} and \eqref{fJonesskein}
and the fact that it is a link invariant.

Relation \eqref{fJonesskein} implies
that the value of the Jones polynomial depends only on
the skein equivalence class of a link, where
skein equivalence is defined as follows:
\begin{definition}[\cite{kaw}]
The {\normalfont skein equivalence} is the minimal equivalence relation
 ``$\sim$''
on the set of oriented links satisfying $L\sim L'$ whenever $L$ and $L'$
are isotopic and such that
\begin{enumerate}
\item $L_0\sim L'_0$ and $L_-\sim L'_-$ imply $L_+\sim L'_+$,
\item $L_0\sim L'_0$ and $L_+\sim L'_+$ imply $L_-\sim L'_-$,
\end{enumerate}
for any two skein triples $(L_+,L_-,L_0)$ and $(L'_+,L'_-,L'_0)$.
\end{definition}

The Laurent polynomials
$\la D\ra$ and $J(L)$ defined as above
are related to the original Kauffman bracket
$\la D\ra^{\operatorname{ori}}$ and Jones polynomial $V(L)$ by
$$
\left[A^{-c}\la D\ra^{\operatorname{ori}}\right]_{A^{-2}=-q}=\la D\ra
\qquad\text{and}\qquad V(L)_{\sqrt{t}=-q}=\frac{J(L)}{q+q^{-1}},
$$
where $c$ denotes the number of crossings of $D$.

\section{Bar--Natan's formal Khovanov bracket}
\noindent
Khovanov homology was discovered by M.~Khovanov \cite{kh:first}
in the year 1999.
In \cite{ba:tangles},
D.~Bar--Natan
proposed a generalization of
Khovanov's invariant, which he called
the formal Khovanov bracket.

In Subsections~\ref{scomplexes},
\ref{sdottedcobordisms}, \ref{sJonesgrading}
and \ref{sadditiveclosure}, we explain the target
category for the formal Khovanov bracket.
Notice that
the category $\Cobdl^3$ used in
this thesis is not the
original category of \cite{ba:tangles}.
It is similar though to a category introduced
in \cite[Section~11]{ba:tangles},
but more general
and more directly related to
Khovanov's universal rank $2$ Frobenius system \cite{kh:frobenius}.
In Subsections~\ref{sdefinitionKhovanovbracket} and
\ref{stangles}, we review the definition of the
formal Khovanov bracket and discuss how
the formal Khovanov bracket extends to tangles.

\subsection[Complexes in additive categories]%
{Complexes in additive categories.}\label{scomplexes}
Let $\C$ be an additive category. A bounded {\it (co)chain complex}
in $\C$ is a sequence of objects and morphisms of $\C$
$$
K:\quad\ldots\;\longrightarrow\hsp K^i
\hsp\stackrel{d_K^i}{\longrightarrow}\hsp K^{i+1}
\hsp\stackrel{d_K^{i+1}}{\longrightarrow}\hsp K^{i+2}
\hsp\longrightarrow\;\ldots
$$
with the property that $d_K^{i+1}\circ d_K^i=0$ for all $i\in\Z$
and $K^i=0$ for all but finitely many $i\in\Z$.
A {\it chain transformation} $f:K\rightarrow L$ between two complexes
$K$ and $L$ in $\C$ is a sequence of morphisms $f^i:K^i\rightarrow L^i$
such that $d_K^i\circ f^i=f^{i+1}\circ d_L^{i+1}$ for all $i\in\Z$.
Two chain transformations $f,g:K\rightarrow L$ are called {\it homotopic}
if there exists a {\it chain homotopy} between them, i.e.
a sequence of morphisms $h^i:K^i\rightarrow L^{i-1}$
such that $d_L^{i-1}\circ h^i+h^{i+1}\circ d_K^i=f^i-g^i$.
Let $\Kom(\C)$ denote the category whose objects are
bounded complexes in $\C$ and whose morphisms are
chain transformations, and let $\Komh(\C)$ denote the
quotient category $\Kom(\C)/\N$ where $\N$ is the ideal
of chain transformations homotopic to $0$.

Two complexes in $\C$ are said to be
{\it isomorphic} ({\it homotopic}) if they are
isomorphic as objects of $\Kom(\C)$ ($\Komh(\C)$).
A complex which is homotopic to the trivial complex
is called {\it contractible}. Equivalently, a complex
$K$ is contractible if its identity morphism
$\Id_K:K\rightarrow K$ is homotopic to $0$.

Given a complex $K=(K^i,d_K^i)$, we refer to the index $i$
as the {\it homological degree}. For every $n\in\Z$
we denote by $[n]$ the endofunctor of $\Kom(\C)$
which raises %\footnote{Our convention is opposite
%to the convention used in \cite{kh:first}.}
the homological degree by $n$, i.e.
$K[n]^{i+n}=K^i$ and $d_{K[n]}^{i+n}=d_K^i$.
(Note that our convention is opposite to the convention
used in \cite{kh:first}).

Given a chain transformation $f:K\rightarrow L$,
the {\it mapping cone} of $f$
is the complex $\G(f):=K\oplus L[1]$ with the differential
$$
d_{\G(f)}:=
\begin{pmatrix} d_K & 0 \\ f & -d_{L[1]}\end{pmatrix}
$$
i.e. $d_{\G(f)}(x,y)=(d_Kx,fx-d_{L[1]}y)$ for all $(x,y)\in K\oplus L[1]$.
It is easy to see that
the mapping cone of an isomorphism is always contractible.
Indeed, if $f$ is an isomorphism, we can define
a homotopy between $\Id_{\G(f)}$ and $0$ by
$$
h=
\begin{pmatrix} 0 & f^{-1} \\ 0 & 0\end{pmatrix}\; .
$$

Let $K$ and $L$ be two complexes in $\C$. We say that
$K$ {\it destabilizes} to $L$, or $L$ {\it stabilizes} to $K$,
if $K$ is isomorphic to $L\oplus C$ for a complex $C$
which is isomorphic to the mapping cone of an isomorphism.
Moreover, we say that two complexes $K$ and $L$
are {\it stably isomorphic}
if they become isomorphic after stabilizing.
The following lemma is taken from \cite[Lemma~2.1]{we:trees}
(but see also \cite[Lemma~4.5]{ba:tangles}).
\begin{lemma}\label{lcommutecone}
Let $K$, $L$ be complexes such that $K=K_1\oplus K_2$ and $L=L_1\oplus L_2$
for contractible complexes $K_2$ and $L_2$. Then the mapping cone $\G(f)$
of a chain transformation
$$
\begin{CD}
K=K_1\oplus K_2@>{\quad\; \mbox{f = }\begin{pmatrix} f_{11} & f_{12}\\
f_{21}&f_{22}\end{pmatrix}\quad\;}>> L_1\oplus L_2=L
\end{CD}\vsp
$$
is isomorphic to the complex $\G(f_{11})\oplus K_2\oplus L_2[1]$.
In particular, if $K_2$ and $L_2$ destabilize to the trivial complex,
then $\G(f)$ destabilizes to $\G(f_{11})$.
\end{lemma}
\begin{proof}
On the level of objects (i.e. if one ignores the differentials),
the complexes $\G(f)$ and $\G(f_{11})\oplus K_2\oplus L_2[1]$ are
both isomorphic to $K\oplus L[1]$. Thus, to prove the lemma,
it suffices to construct an automorphism
$F:K\oplus L[1]\rightarrow K\oplus L[1]$
which intertwines the differentials. We define $F$ by
$$
F=\begin{pmatrix} \Id_K &0\\-N &\Id_L\end{pmatrix}
$$
where $N:K_1\oplus K_2\rightarrow L_1[1]\oplus L_2[1]$ is given by
$$
N=\begin{pmatrix} 0 &f_{12}h_K\\h_Lf_{21}
&h_Lf_{22}\end{pmatrix}
$$
with $h_K$ and $h_L$ denoting the contracting
homotopies of the complexes $K_2$ and $L_2$, respectively.
A direct computation shows that
$F\circ d_{\G(f)}=(d_{\G(f_{11})}+d_{K_2}+d_{L_2[1]})\circ F$,
so $F$ is indeed an isomorphism between the complexes
$\G(f)$ and $\G(f_{11})\oplus K_2\oplus L_2[1]$.
\end{proof}

If $K=(K^i,d_K^i)$ is a complex in a category of modules over a ring,
one can define the $i$--th {\it homology module} of $K$
by $H^i(K):=(\kernel d_K^i)/(\image d_K^{i-1})$.
It is easy to see that
homotopic complexes have isomorphic homology modules.

\subsection[Dotted cobordisms]{Dotted cobordisms.}\label{sdottedcobordisms}
Let $D_0$, $D_1$ be two closed $1$--manifolds embedded in the plane $\R^2$.
A {\it cobordism} from $D_0$ to $D_1$ is a compact orientable
surface $S\subset\R^2\times[0,1]$ whose boundary lies entirely
in $\R^2\times\{0,1\}$ and whose ``bottom'' boundary
is $D_0\times\{0\}$ and whose ``top'' boundary is $D_1\times\{1\}$.
A {\it dotted cobordism} is a cobordism which is decorated by finitely
many distinct dots, lying in its interior.
(These dots must not be confused with the signed points
which will be
introduced in Chapter~\ref{cframedlinkcobordisms}).
Dotted cobordisms can be composed by
placing them atop of each other.
We denote by $\Cobd^3$ the category whose objects are closed
embedded $1$--manifolds and whose morphisms are dotted
cobordisms, considered up to boundary--preserving isotopy.

We also define a quotient $\Cobdl^3$ of $\Cobd^3$,
as follows. $\Cobdl^3$ has the same objects as $\Cobd^3$
but its morphisms are formal $\Z$--linear combinations of morphisms
of $\Cobd^3$, considered modulo the following local relations:
\begin{figure}[H]
\centerline{
\psfig{figure=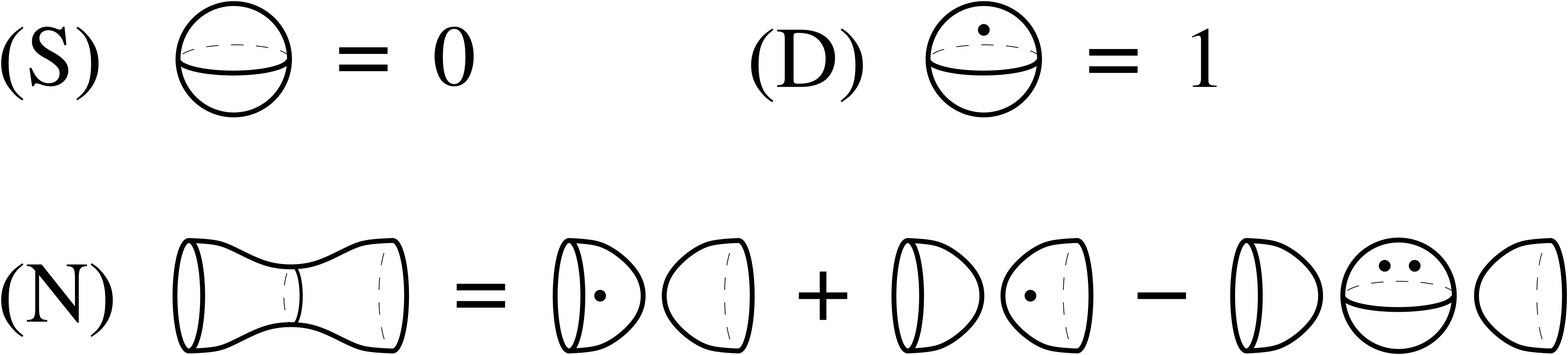,height=2.2cm}
}
\caption{(S), (D) and (N) relation.}
\end{figure}
\noindent
Relation (S) means that any cobordism, which has
a sphere without dots among its connectivity components,
is set to zero. Relation (D) means
that a sphere decorated by a single dot can be
removed from a cobordism without changing the class
of the coborsism in $\Cobdl^3$. Finally,
(N) is the {\it neck--cutting relation}.
It can be used to reduce the genus of a cobordism,
at the expense of introducing some extra dots.
Note that (S), (D) and (N) imply the following relations:
\begin{figure}[H]
\centerline{
\psfig{figure=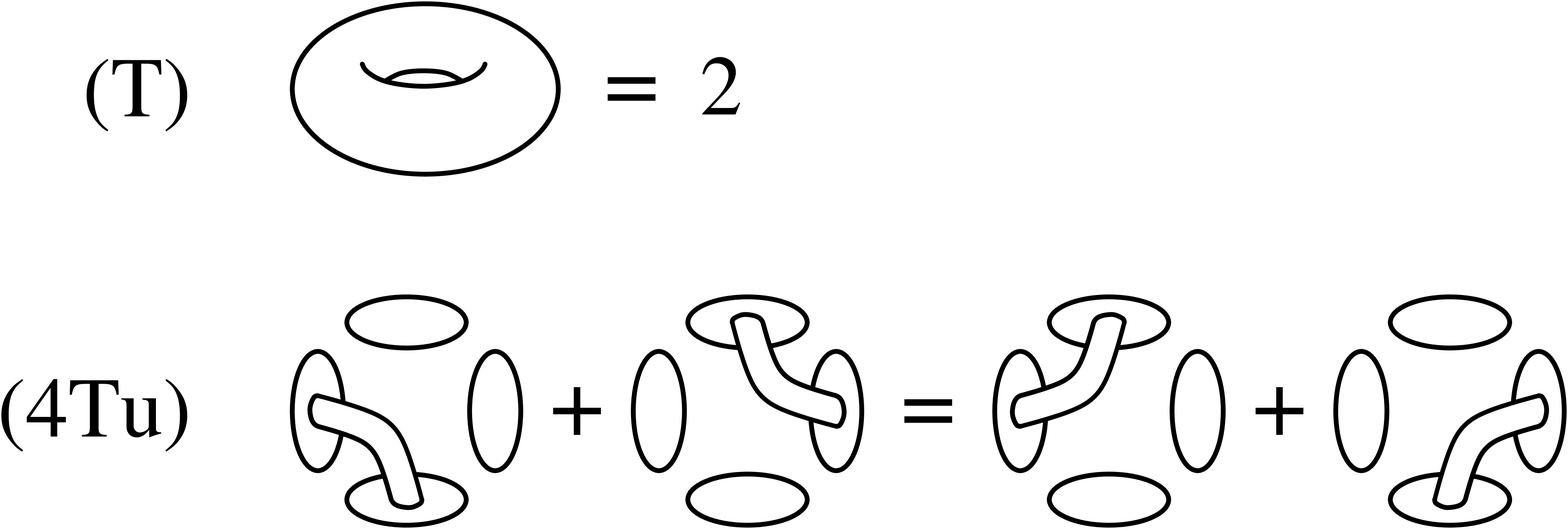,height=3cm}
}
\caption{(T) and (4Tu) relation.}
\end{figure}
\noindent
If we impose the additional relation that a sphere decorated
by exactly two dots is zero
(
$\raisebox{-0.2cm}{\psfig{figure=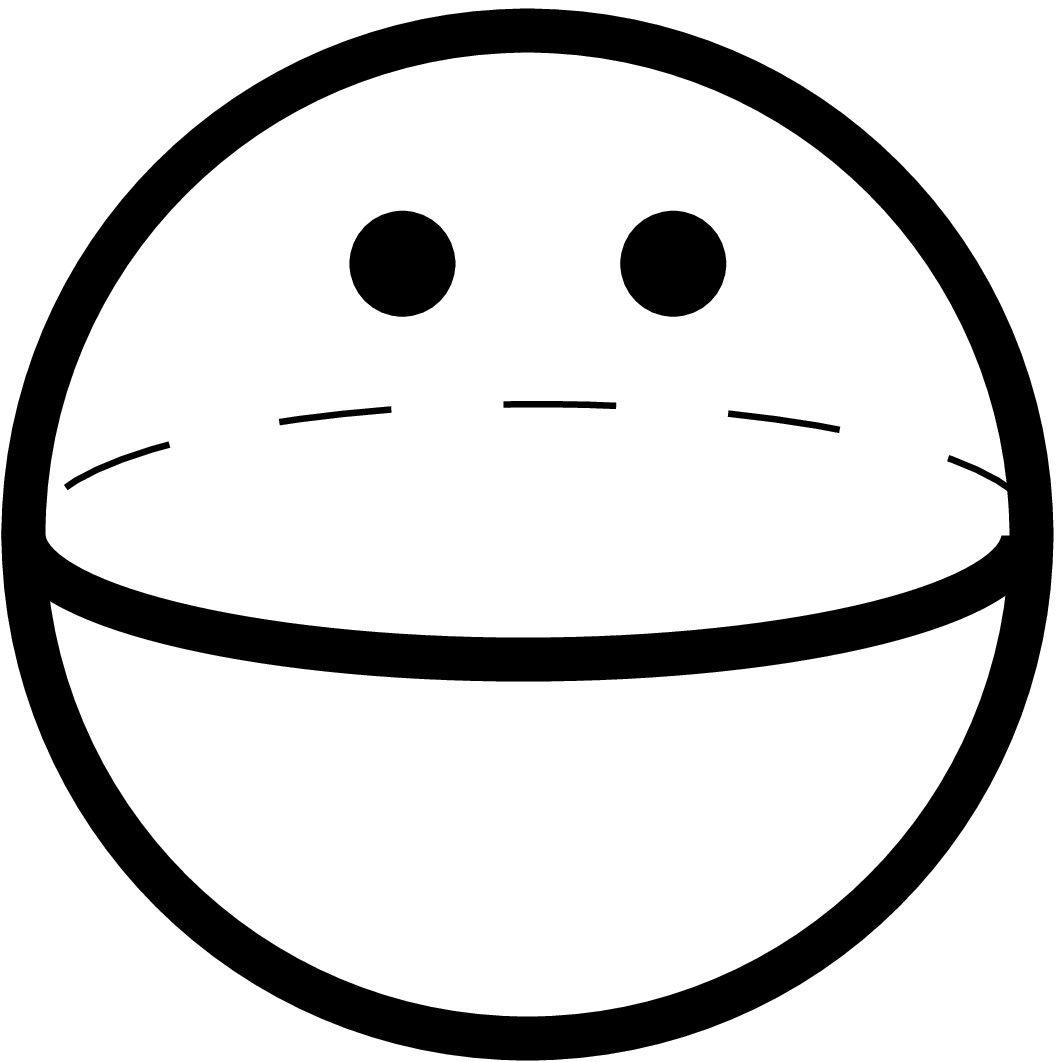,height=0.6cm}}=0$
), then
the (S) relation becomes a consequence
of the relations (D) and (N). Moreover, forming
the connected sum with a torus becomes equivalent to
inserting a dot at the connected sum point and then
multiplying by $2\in\Z$. Hence we essentially get
back the theory of \cite[Section~11]{ba:tangles},
\cite{ba:computations}.

\begin{notations}
We will use the following notations for the
generating morphisms of $\Cobdl^3$.
The symbol
$\HSaddleSymbol$ stands for a
saddle cobordism
from $\smoothing$ to $\hsmoothing$.
More specifically,
$\splitsaddle$ stands for a saddle which splits
a single component into two, and
$\dumbbell$ stands for a saddle which merges
two components into one.
$\fourwheel:\emptyset\rightarrow\bigcirc$
and $\fourinwheel:\bigcirc\rightarrow\emptyset$
denote the cup and the cap cobordism, and
$\dotcircle:\bigcirc\rightarrow\bigcirc$
denotes
the ``multiplication'' of
$\bigcirc$ by a dot, i.e.
the identity cobordism $\bigcirc\times [0,1]$
decorated by a single dot.
\end{notations}

\subsection[Jones grading]{Jones grading.}\label{sJonesgrading}
In this subsection, we enhance the category $\Cobdl^3$
by introducing a grading. We essentially
follow \cite[Section~6]{ba:tangles}.

Given a dotted cobordism $S$,
we define its {\it Jones degree} by
$$
\deg(S):=\chi(S)-2\delta(S)
$$
where $\chi(S)$ denotes the Euler characteristic of
$S$ and $\delta(S)$
denotes
the number of dots on $S$.
Since the (S), (D) and (N) relations are
degree--homogeneous, the Jones degree
descends
to $\Cobdl^3$, turning morphism sets of
$\Cobdl^3$
into graded $\Z$--modules.

We construct a
graded category $(\Cobdl^3)'$.
The objects of $(\Cobdl^3)'$
are pairs $(D,n)$, one
for each object $D\in\Ob(\Cobdl^3)$ and each
integer $n\in\Z$.
As ungraded $\Z$--modules, the morphism sets
of $(\Cobdl^3)'$
are the same as in $\Cobdl^3$, i.e.
$$
\Mor((D_0,n_0),(D_1,n_1)):=\Mor(D_0,D_1)\; .
$$
But the Jones degree of $S\in\Mor((D_0,n_0),(D_1,n_1))$
is defined by
$$
\deg(S):=\chi(S)-2\delta(S)+n_1-n_0\; .
$$
Note that $\deg(S)$ is additive under composition
of morphisms.

For $m\in\Z$, we denote by $\{m\}$ the endofunctor
of $(\Cobdl^3)'$ which ``raises\footnote{Our convention
is opposite to the convention in \cite{kh:first}.}
the grading'' by $m$,
i.e. $(D,n)\{m\}:=(D,n+m)$.
To simplify notations, we will write $D$ instead of
$(D,0)$  (and consequently $D\{n\}$ instead of $(D,n)$).

In what follows,
we suppress the prime from $(\Cobdl^3)'$
and just call it $\Cobdl^3$.
We denote by $\gCobdl^3$
the subcategory of
$\Cobdl^3$ which has the same objects as $\Cobdl^3$, but
whose morphisms are required to be graded of Jones degree $0$.

\subsection[Additive closure and delooping]%
{Additive closure and delooping.}\label{sadditiveclosure}
For every pre--additive category $\C$, there is an associated
additive category $\Mat(\C)$, called its {\it additive closure}.
The objects of $\Mat(\C)$ are finite sequences $(\OC_i)_{i=1}^n$
of objects $\OC_i\in\Ob(\C)$, which we write as
formal direct sums $\bigoplus_{i=1}^n\OC_i$.
The morphisms $F:\bigoplus_j\OC_j\rightarrow\bigoplus_i\OC'_i$
are matrices $[F_{i,j}]$ of morphisms $F_{i,j}:\OC_j\rightarrow\OC'_i$.
Composition of morphisms is modeled on ordinary matrix multiplication:
$$
[F\circ G]_{i,k}:=\sum_j F_{i,j}\circ G_{j,k}
$$

The following lemma is Bar--Natan's Lemma~4.1 \cite{ba:computations},
with the only difference that we use a slightly more general
definition for the category $\Cobdl^3$.

\begin{lemma}\label{ldelooping}
(Delooping)
Let $D'$ be an object in $\gCobdl^3$ containing a circle $\bigcirc$,
and let $D$ be the object obtained by removing this circle from $D'$.
Then $D'$ is isomorphic in $\Mat(\gCobdl^3)$ to $D\{+1\}\oplus D\{-1\}$.
\end{lemma}
\begin{proof} It suffices to show that the circle $\bigcirc$
is isomorphic to
$\emptyset\{+1\}\oplus\emptyset\{-1\}$.
The isomorphisms are given by
\vsp

\ \\
\centerline{\psfig{figure=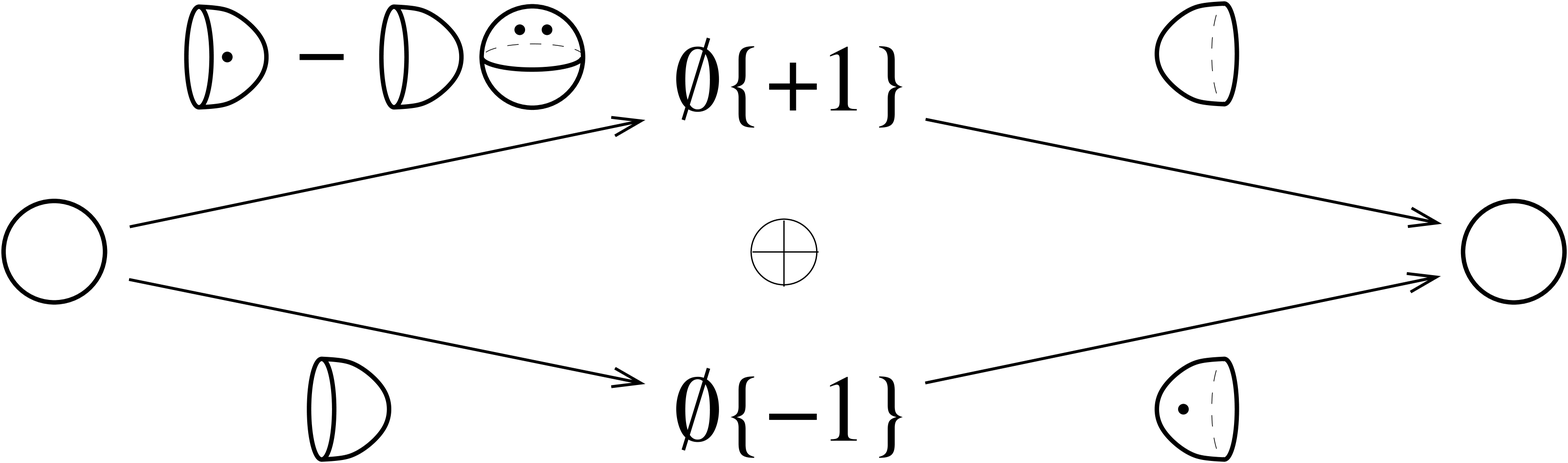,height=3cm}}
\vsp

\noindent
Using relations (S), (D) and (N), it is easy to see that
the above morphisms are mutually inverse isomorphism.
\end{proof}

Let $\Kobd:=\Kom(\Mat(\Cobdl^3))$ denote the category of
bounded complexes in $\Mat(\Cobdl^3)$, and
$\Kobdh:=\Komh(\Mat(\Cobdl^3))$ its homotopy category.
Likewise, let
$\gKobd:=\Kom(\Mat(\gCobdl^3))$
and $\gKobdh:=\Komh(\Mat(\gCobdl^3))$.

\subsection[Definition of the Khovanov bracket]%
{Definition of the Khovanov bracket.}\label{sdefinitionKhovanovbracket}
Let $D$ be an unoriented link diagram with $c$ crossings.
Recall that the Kauffman states of $D$ are the
diagrams obtained by replacing every crossing
of $D$ by its $0$--smoothing or its $1$--smoothing.
After numbering the crossings of $D$, we can
parametrize the Kauffman states of $D$ by $c$--letter
strings of $0$'s and $1$'s, specifying the smoothing
chosen at each crossing.
Let $D_s$ denote the Kauffman
state corresponding to the $c$--letter
string $s\in\{0,1\}^c$, and let
$r(s):=r(D,D_s)$ and $n(s):=n(D_s)$
denote respectively the number of $1$'s in $s$
and the number of circles in $D_s$.
We can arrange the Kauffman states
of $D$ at the vertices of a $c$--dimensional cube.
In Figure~\ref{figcube}, the cube is displayed in such a way that
two vertices which have the same number of $1$'s
(i.e. the same $r(s)$) lie vertically above each
other.
\begin{figure}[H]
%\mbox{
\begin{center}
\epsfysize=6cm \epsffile{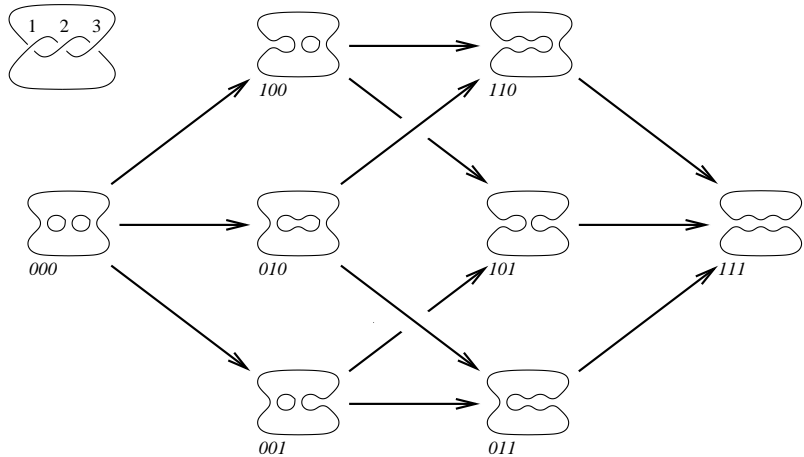}
\end{center}
%}
\caption{The cube of resolutions for the trefoil.}
\label{figcube}
\end{figure}

Two vertices $s$ and $t$ are connected by an
edge (directed from $s$ to $t$) if they differ by a single
letter which is a $0$ in $s$ and a $1$ in $t$.
For such $s$ and $t$ the corresponding Kauffman
states $D_s$ and $D_t$ differ at a single crossing $\slashoverback$
which is a $0$--smoothing in $D_s$ and a $1$--smoothing
in $D_t$. To the edge connecting
$s$ and $t$, we associate a cobordism
$S_s^t:D_s\rightarrow D_t$, defined
as follows: in a neighborhood of the crossing
$\slashoverback$,
the cobordism
$S_s^t\subset\R^2\times [0,1]$ is a
saddle cobordism $\HSaddleSymbol:\smoothing\rightarrow\hsmoothing$.
Outside that neighborhood,
it is vertical (parallel to $[0,1]$).

Regarding the Kauffman states $D_s$ and the
cobordisms $S_s^t$
as objects and morphisms, we can view the above cube
as a commutative diagram in the category $\Cobdl^3$.
Indeed, for every square
\[
\xymatrix{
&D_t\ar[rd]^{S_t^u}\\
D_s\ar[ru]^{S_s^t}\ar[rd]_{S_s^{t'}}&&D_u\\
&D_{t'}\ar[ru]_{S_{t'}^u}
}
\]
we have $S_t^u\circ S_s^t= S_{t'}^u\circ S_s^{t'}$ because
distant saddles can be reordered by isotopy.
We can make all squares of the cube
anticommute by multiplying each morphism
$S_s^t$  by $(-1)^{\la s,t\ra}$, where $\la s,t\ra$ denotes the number
of $1$'s in $s$ (or in $t$) preceding the letter
which is a $0$ in $s$ and a $1$ in $t$.
If we replace each $D_s$ by $D_s\{r(s)\}$,
the Jones degree of
$S_s^t$ becomes $\deg(S_s^t)=\chi(S)+r(t)-r(s)=-1+(r(s)+1)-r(s)=0$,
and hence $S_s^t$ becomes a morphism
in the category $\gCobdl^3$.

Now we ``flatten'' the cube by taking
the direct sum of all objects and morphisms
which lie vertically above each other.
The result is a chain complex in the category
$\Mat(\gCobdl^3)$.
The $i$--th ``chain space'' is
given by
\begin{equation}\label{fformaldef}
\ls D \rs^i:=\bigoplus_{s:r(s)=i}
D_s\{i\}\;\in\Ob(\Mat(\gCobdl^3))
\end{equation}
The $i$--th differential
$d^i:\ls D \rs^i\rightarrow \ls D \rs^{i+1}$
is given as follows:
for two vertices $s$ and $t$ with $r(s)=i$ and
$r(t)=i+1$, the matrix element
$(d^i)_{t,s}$ is equal to
$(-1)^{\la s,t\ra}S_s^t$
whenever $s$ and $t$ are connected by
an edge, and zero otherwise.

Since squares of the cube anticommute, we get
$d^{i+1}\circ d^i=0$, whence
$(\ls D \rs^i,d^i)$ is indeed
a chain complex.
We call this chain complex the
{\it formal Khovanov bracket} of
$D$.

Note that the signs $(-1)^{\la s,t\ra}$ depend on
the numbering of the crossings of $D$.
However,
one can prove that different
%choices
numberings
lead to isomorphic complexes.

\begin{lemma}\label{lkhovanovbracket}
The formal Khovanov bracket satisfies:
\begin{enumerate}
\item
$\ls \epsh{12pt}{leftRI}\rs$ destabilizes to
$\ls \epsh{12pt}{noRI} \rs\{-1\}$. Likewise,
$\ls \epsh{12pt}{rightRI}\rs$ destabilizes to
$\ls \epsh{12pt}{noRI} \rs[1]\{2\}$.
\item
$\ls \epsh{12pt}{RII}\rs$ destabilizes to
$\ls\epsh{12pt}{noRII}\rs[1]\{1\}$.
\item
$\ls \epsh{14pt}{RIIIa}\rs$ is
stably isomorphic to
$\ls\epsh{14pt}{RIIIb}\rs$.
\end{enumerate}
\end{lemma}
In the lemma,
$[.]$ and $\{.\}$ denote the shift of the
homological degree and the Jones degree, respectively.
For a proof of the lemma, see \cite{kh:first} or \cite{ba:tangles}.

If $D$ is an oriented link diagram, we define
\begin{equation}\label{fKhdef}
\Kh(D):=\ls D\rs[-c_-(D)]\{c_+(D)-2c_-(D)\}\;\in\Ob(\gKobd)
\end{equation}
Lemma~\ref{lkhovanovbracket} implies:
\begin{theorem}\label{tkhovanovinvariant}
The complex $\Kh(D)$ is a link invariant up to (graded) isomorphism
and stabilization.
\end{theorem}

\begin{remark}
Assume $\slashoverback$, $\smoothing$ and $\hsmoothing$
are three link diagrams which are identical except
in a small disk, where they look like
$\slashoverback$, $\smoothing$ and $\hsmoothing$,
respectively. Then the cube of the diagram $\slashoverback$
contains two
codimension $1$ subcubes,
which after flattening become the complexes $\ls\smoothing\rs$
and $\ls\hsmoothing\rs [1]\{1\}$. The cobordisms
associated to the edges
connecting the two subcubes can be
assembled to a chain transformation
$\ls\HSaddleSymbol\rs:
\ls\smoothing\rs\rightarrow\ls\hsmoothing\rs\{1\}$,
such that $\ls\slashoverback\rs$
is canonically isomorphic to
the mapping cone of this chain transformation:
\begin{equation}\label{fsaddlecone}
\ls \slashoverback\rs=\G\left(\ls\smoothing\rs
\stackrel{\ls\HSaddleSymbol\rs}\longrightarrow
\ls\hsmoothing\rs\{1\}\right)
\end{equation}
\eqref{fsaddlecone} is an analogue of
the relation
$\la \slashoverback\ra=\la\smoothing\ra-q\la\hsmoothing\ra$
of Section~\ref{sJones}.
\end{remark}

\subsection[Tangles]{Tangles.}\label{stangles}
The formal
Khovanov bracket can be extended to {\it tangles},
i.e. to ``parts of link diagrams'' bounded
within a circle.
\begin{figure}[H]
\begin{center}
\epsfysize=4cm \epsffile{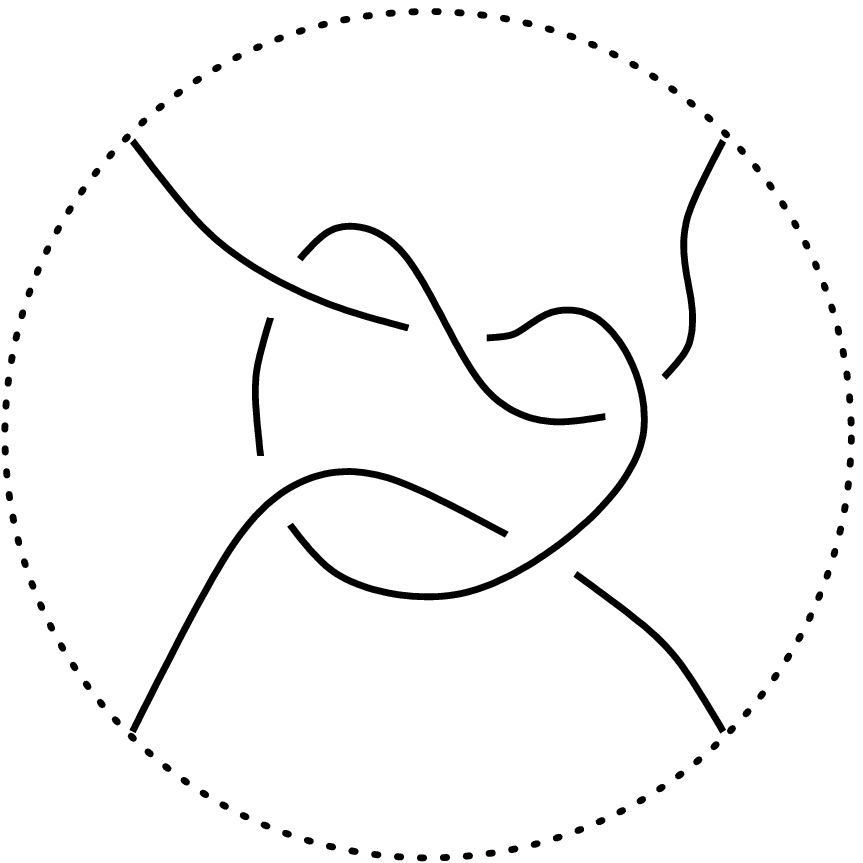}
\end{center}
\caption{A tangle.}\label{figatangle}
\end{figure}\noindent
Assume $T$ is a tangle,
whose boundary $\partial T$
consists of finitely many points lying on the dotted
circle.
Then the Khovanov bracket $\Kh(T)$ is a chain complex
in the category $\Mat(\gCobdl^3(\partial T))$,
where $\gCobdl^3(\partial T)$ is
defined in analogy with $\gCobdl^3$,
with the difference that
now the dotted cobordisms are confined within
a cylinder and that they have a vertical
boundary component $\partial T\times [0,1]$.
The Jones degree of a dotted cobordism
$S:D_0\{n_0\}\rightarrow D_1\{n_1\}$
with vertical boundary $\partial T\times [0,1]$
is defined by
$$
\deg(S)=\chi(S)-2\delta(S)+\frac{1}{2}|\partial T|+n_1-n_0
$$
where $|\partial T|$ denotes the number of points in
$\partial T$.

The Khovanov bracket for tangles has good composition
properties: suppose $T_1$ and $T_2$ are two tangles,
which can be glued side by side to form a bigger
tangle $T_1 T_2$. Then there is a corresponding composition
``$\sharp$'' of formal Khovanov brackets
such that $\Kh(T_1 T_2)=\Kh(T_1)\sharp\Kh(T_2)$
(see \cite{ba:tangles} for details).

\section{Functoriality}\label{sfunctoriality}
\noindent
Let $\Cob^4$ denote the category whose objects are
oriented link diagrams, and whose morphisms are
movie presentations.
Composition of movies is given by ``playing'' one
movie after the other, identifying the last still
of the first movie with the first still of the second.

We can extend the formal Khovanov bracket to a functor
$\Kh:\Cob^4\rightarrow\Kobd$ as follows.
On objects, we define $\Kh$ as in \eqref{fKhdef}.
To define $\Kh$ on morphisms, it suffices
to assign chain transformations to
Reidemeister moves, and to cap, cup and saddle moves.
For the Reidemeister moves, we take the chain transformations
implicit in the proof of Lemma~\ref{lkhovanovbracket}.
For the cap, cup and saddle, we take the
natural chain transformations induced by the corresponding
morphisms
$\fourinwheel$, $\fourwheel$ and $\HSaddleSymbol$
in the category $\Cobdl^3$.

Let $\Cobi^4$ denote the quotient of
$\Cob^4$ by Carter--Saito moves, and
$\Kobdpr$ the projectivization of $\Kobdh$
(i.e. the category which has the same objects as $\Kobdh$,
but where every morphism is identified with its negative).
\begin{theorem}\label{tfunctor}
$\Kh$ descends to a functor $\Kh:\Cobi^4\rightarrow\Kobdpr$.
\end{theorem}
For proofs of Theorem~\ref{tfunctor}, see
\cite{ja}, \cite{kh:cobordisms} and \cite{ba:tangles}.
Jacobsson's proof is based on checking explicitly
that the chain transformations associated to the
two sides of the Carter--Saito moves are
homotopic up to sign. Bar--Natan's proof is
more conceptual and remains valid
in our slightly different setting.

A {\it dotted link cobordism} is a link cobordism
decorated by finitely many distinct dots.
There is a notion of movie presentation for dotted
link cobordisms, allowing us to define a category
$\Cobd^4$ whose objects are oriented link diagrams
and whose morphisms are movie presentations of
dotted link cobordisms. We can extend
the functor $\Kh:\Cob^4\rightarrow\Kobd$
to $\Cobd^4$, by viewing dots
on a link cobordism as dots in $\Cobdl^3$.

Let $\Cobdi^4$ denote the quotient of $\Cobd^4$
by Carter--Saito moves and by displacement
of dots (i.e. by sliding a dot across a crossing).
The following lemma shows that
$\Kh:\Cobd^4\rightarrow\Kobd$
descends to a functor $\Kh:\Cobdi^4\rightarrow\Kobdpr$
if one imposes the
additional relation
$\raisebox{-0.2cm}{\psfig{figure=figs/ddot.eps,height=0.6cm}}=0$
on the category $\Cobdl^3$.
\begin{lemma}[\cite{ba:mutation}]\label{ldotslide}
Assume $\raisebox{-0.2cm}{\psfig{figure=figs/ddot.eps,height=0.6cm}}=0$.
Then the chain transformations
$\Kh(\dotbcrossing)$ and $\Kh(\dotacrossing)$ induced by ``multiplying''
by dot before and after a crossing $\slashoverback$
are homotopic up to sign.
\end{lemma}
More precisely, one can show that $\Kh(\dotbcrossing)$
is homotopic to $-\Kh(\dotacrossing)$.

%%%%%%%%%%% Section: Homology Theories %%%%%%%%%%%%%%%%%%%%%%
\section{Homology theories}
\noindent
Let $\Cob$ denote
the category whose objects are closed oriented $1$--manifolds
and whose morphisms are abstract (i.e. non--embedded) oriented
$2$--cobordisms, considered up to homeomorphism relative
to their boundary. $\Cob$ is a tensor category with tensor
product given by disjoint union.
A {\it $(1+1)$--dimensional topological quantum field theory (TQFT)}
is a monoidal functor
$$
\F:\Cob\longrightarrow R\text{-}\MOD\; ,
$$
where $R\text{-}\MOD$ is the category of finite projective
modules over a commutative unital ring $R$.

Assume $\F$
is a $(1+1)$--dimensional TQFT which extends to dotted
cobordisms, in a way compatible with the (S), (D) and (N)
relations. Then $\F$ induces a functor
$\F:\Cobdl^3\rightarrow R\text{-}\MOD$.
Every such functor extends
to a functor
$$
\F:\Kobd\longrightarrow\Kom(R\text{-}\MOD)\; .
$$
Applying $\F$ to $\Kh(D)\in\Ob(\Kobd)$, we obtain an ordinary chain
complex $\F\Kh(D)$ in the category of $R$--modules.
The isomorphism class of the homology
of this complex is a link invariant, which
is often
more tractable than the
original Khovanov bracket.

Below, we will first recall the well--known
correspondence between $(1+1)$--dimensional TQFTs and
Frobenius systems, and then give examples of TQFTs
descending to $\Cobdl^3$ and discuss their
associated link homology theories.

\subsection[Frobenius systems]{Frobenius systems.}
Algebraically, $(1+1)$--dimensional TQFTs can be described
in terms of (commutative) Frobenius systems.
A (commutative) {\it Frobenius system} is a $4$--tuple
$(R,A,\e,\D)$ where $R$, $A$, $\e$ and $\D$ are
the following objects and morphisms.
$A$ is a commutative unital $R$--algebra,
such that the natural $R$--module map $\iota:R\rightarrow A$
given by $\iota(1)=1$ is injective.
$\e:A\rightarrow R$ is a map of $R$--modules,
and $\D$ is a coassociative and cocommutative
map $\D:A\rightarrow A\otimes_{R} A$ of $A$--bimodules
such that $(\e\otimes\Id)\circ\D=\Id$ (see \cite{kh:frobenius}).

Given a commutative Frobenius system, we can define
a $(1+1)$--dimensional TQFT $\F$ by assigning $R$ to the
empty $1$--manifold, $A$ to the circle, $A\otimes_R A$ to
the disjoint union of two circles etc. On generating
morphisms of $\Cob$ (cup, cap, splitting and merging saddle)
we define $\F$ by
$\F(\fourwheel):=\iota$, $\F(\fourinwheel):=\e$,
$\F(\splitsaddle):=\D$ and
$\F(\dumbbell):=m$, where $m$ is the multiplication of $A$.

For our purposes, we need a TQFT
$\F:\Cob\rightarrow R\text{-}\MOD$ which extends to
dotted cobordisms, in a way compatible with the (S), (D) and (N)
relations. It is easy to see that for such a TQFT
the corresponding Frobenius algebra $A$ has to be
a free $R$--module of rank $2$. Indeed, let
${\bf 1}=\iota(1)\in A$ denote the unit of $A$,
and let $X\in A$ denote the image of $1\in R$
under the map
$\F( \raisebox{-0.2cm}{\psfig{figure=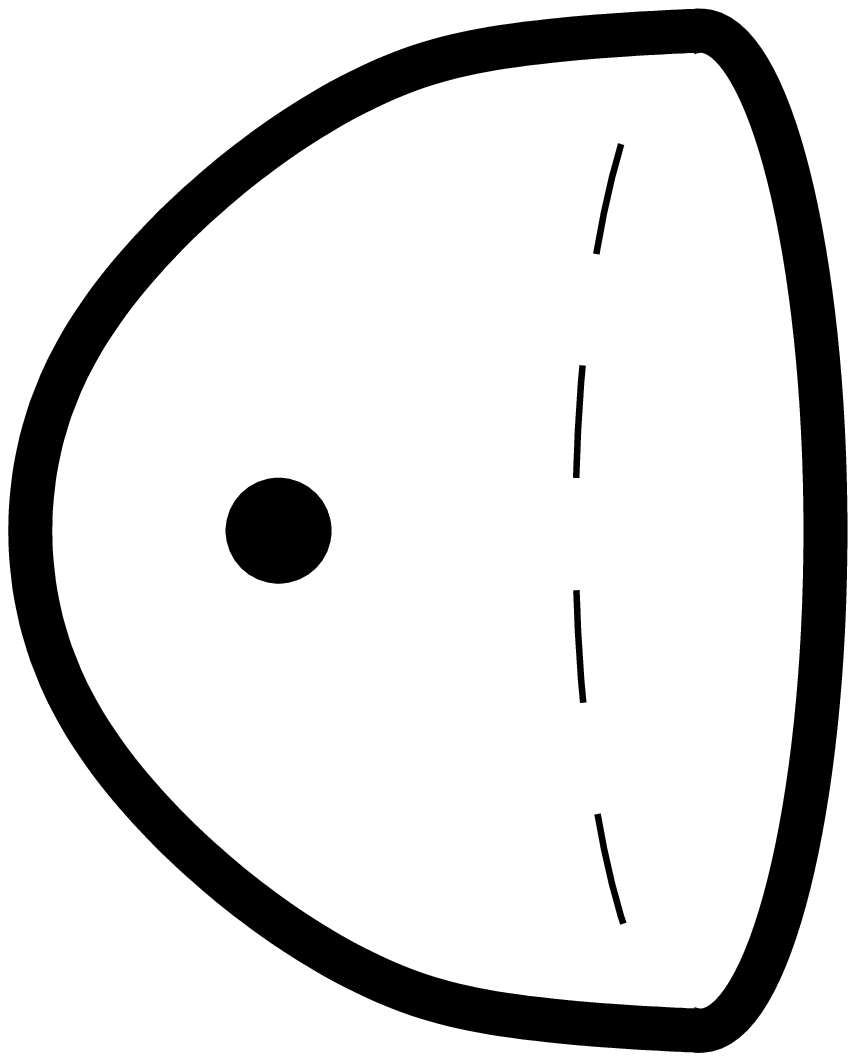,height=0.6cm}} ):
R\rightarrow A$, i.e. under the map induced
by a cup cobordism decorated by a single dot.
A look at the delooping--isomorphism in the proof
of Lemma~\ref{ldelooping} reveals that
$\{{\bf 1},X\}$ is an $R$--basis of $A$.

\subsection[The universal functor]{The universal functor.}
The universal functor
$\F_{\emptyset}:\Cobdl^3\rightarrow R_{\emptyset}\text{-}\MOD$
is defined as follows.
On objects, $\F_{\emptyset}$ is given by
$$
\F_{\emptyset}(D):=\Mor(\emptyset,D)
$$
where $\Mor(\emptyset,D)$ denotes the set of morphisms
from $\emptyset$ to $D$ in the category $\Cobdl^3$.
Note that $\Mor(\emptyset,D)$ is a graded $\Z$--module.
On morphisms, $\F_{\emptyset}$ is defined by composition
on the left. That is, if $S\in\Mor(D,D')$ then
$\F_{\emptyset}(S):\Mor(\emptyset,D)\rightarrow\Mor(\emptyset,D')$
maps $S'\in\Mor(\emptyset,D)$ to $S\circ S'\in\Mor(\emptyset,D')$
(compare \cite[Definition~9.1]{ba:tangles}).

Let us study the Frobenius system
$(R_\emptyset,A_\emptyset,\e_\emptyset,\D_\emptyset)$
associated to $\F_{\emptyset}$.
By definition of $\F_{\emptyset}$, the
ring $R_\emptyset$ and the Frobenius algebra $A_\emptyset$
are given by
$$
R_{\emptyset}=\Mor(\emptyset,\emptyset),
\qquad A_{\emptyset}=\Mor(\emptyset,\bigcirc)
$$
where
the multiplication maps of $R_{\emptyset}$ and $A_{\emptyset}$
are given by disjoint union and by composition with
the merging saddle ($\dumbbell$), respectively.
The
$R_{\emptyset}$--module structure on $A_{\emptyset}$
is induced by disjoint union.
There are isomorphisms
\begin{equation}\label{funiversal1}
R_{\emptyset}\cong\Z[h,t],\qquad
A_{\emptyset}\cong R_{\emptyset}[X]/(X^2-hX-t{\bf 1})
\end{equation}
given as follows.
Under the first isomorphism,
$h$ corresponds to
\raisebox{-0.2cm}{\psfig{figure=figs/ddot.eps,height=0.6cm}}
(a sphere decorated by two dots) and
$t$ corresponds to
\raisebox{-0.2cm}{\psfig{figure=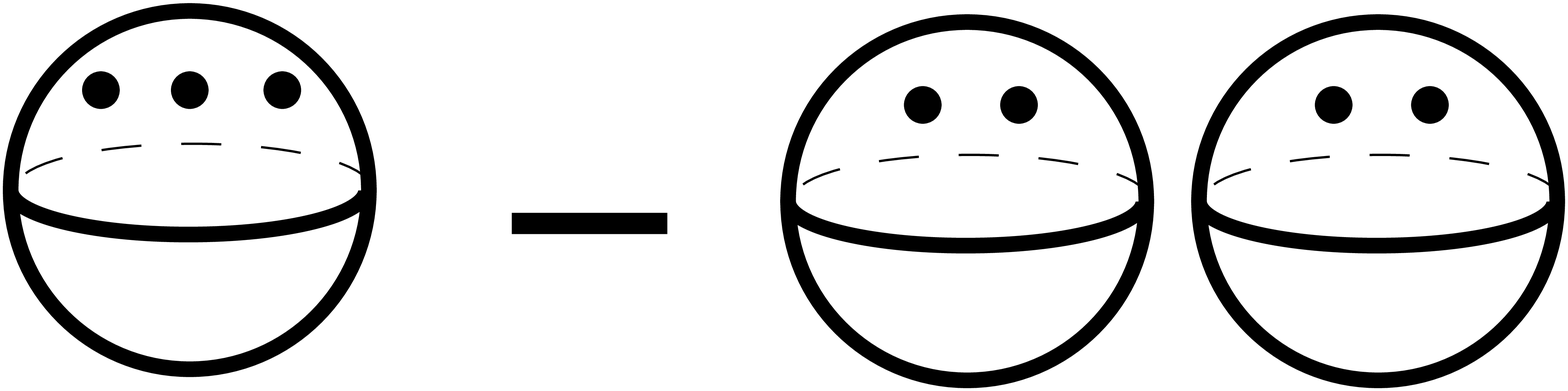,height=0.6cm}}
(a sphere decorated by three dots minus the
disjoint union of two spheres decorated
by two dots).
The second isomorphism in \eqref{funiversal1}
sends a cup decorated
by $n$ dots to $X^n$. In particular, the empty cup
corresponds to ${\bf 1}$.

The isomorphisms become graded
if one defines
$$
\deg(h):=-2,\qquad \deg(t):=-4,\qquad \deg({\bf 1}):=+1,\qquad
\deg(X):=-1\; .
$$
On tensor products
$A_{\emptyset}\otimes_{R_{\emptyset}}\ldots
\otimes_{R_{\emptyset}} A_{\emptyset}$ the grading is given by
$\deg(a_1\otimes\ldots\otimes a_n):=\deg(a_1)+\ldots+\deg(a_n)$.

Khovanov \cite{kh:frobenius} observed that $A_{\emptyset}$
is the polynomial ring in $X$ and $Y:=h-X$,
and $R_{\emptyset}$ is the ring of symmetric functions
in $X$ and $Y$, with $h$ and $-t$ the elementary symmetric
functions.
With this interpretation, we can describe the isomorphism
$R_{\emptyset}\cong\Z[h,t]$ more explicitly,
as follows.
Let $S\in R_\emptyset=\Mor(\emptyset,\emptyset)$
be a closed cobordism. Using the (N) relation, we can reduce
the genus of $S$. Moreover, it is sufficient to
consider the case where $S$ is connected.
Hence we may assume that $S$ is a sphere decorated
by $n$ dots. In this case,
$S\in R_\emptyset$ corresponds to
$$
[n;X,Y]:=\frac{X^n-Y^n}{X-Y}\;\in\Z[h,t]\; .
$$
To see this, compare the recursion relations
\begin{eqnarray*}
 [0;X,Y] &=& 0\, ,\\
\, [1;X,Y] &=& 1\, ,\\
\, [n+1;X,Y] &=& h[n;X,Y]+t[n-1;X,Y]
\end{eqnarray*}
with the (S) and (D) relations and with
the geometric relation corresponding to
$X^2=hX+t{\bf 1}$ (i.e. with the relation saying
that two dots are the same as
$h$ times one dot plus $t$ times no dot).

The structural maps
$\e_{\emptyset}:A_{\emptyset}\rightarrow R_{\emptyset}$ and
$\D:A_{\emptyset}\rightarrow A_{\emptyset}\otimes_{R_{\emptyset}} A_{\emptyset}$
are given by
\begin{equation}\label{funiversal2}
  \e_{\emptyset}:
  \begin{cases}
    {\bf 1} &\mapsto\quad 0 \\
    X       &\mapsto\quad 1
  \end{cases}
  \qquad\quad
  \D_{\emptyset}:
  \begin{cases}
    {\bf 1}&\mapsto\quad
 {\bf 1}\otimes X + X\otimes {\bf 1} - h{\bf 1}\otimes {\bf 1}\\
     X     &\mapsto\quad X\otimes X + t{\bf 1}\otimes {\bf 1}
  \end{cases}
\end{equation}
Khovanov \cite{kh:frobenius} proved that the Frobenius
system $(R_{\emptyset},A_{\emptyset},\e_{\emptyset},\D_{\emptyset})$
determined by \eqref{funiversal1} and \eqref{funiversal2} is universal
among all rank two Frobenius system, in the sense that every other rank
two Frobenius system can be obtained from this one by {\it base change}
(i.e. extending coefficients of $A$ by using a
morphism $\psi:R\rightarrow R'$ of commutative
unital rings to replace $A$ by $A':=A\otimes_R R'$)
and {\it twisting} (replacing $\e(x)$ by $\e(yx)$
and $\D(x)$ by $\D(y^{-1}x)$ for a fixed invertible
element $y\in A$).

\subsection[Khovanov's functor]{Khovanov's functor.}
Khovanov's \cite{kh:first}
functor $\F_{\Kh}$ is obtained
from the universal functor $\F_{\emptyset}$ by setting $h$ and $t$ to zero
(or equivalently by base change via
$\psi:R_{\emptyset}=\Z[h,t]\rightarrow\Z$, $\psi(h)=\psi(t)=0$).
The resulting Frobenius system is
$$
R_{\Kh}=\Z,\qquad
A_{\Kh}=\Z[X]/(X^2)
$$
Since the relations $h=0$ and $t=0$ are homogeneous,
the grading on $A_{\emptyset}$ descends to $A_{\Kh}$.
The degrees of ${\bf 1},X\in A_{\Kh}$ are $\deg({\bf 1})=1$ and $\deg(X)=-1$.
The structure maps are given by
$$
  \e_{\Kh}:
  \begin{cases}
    {\bf 1} &\mapsto\quad 0 \\
    X       &\mapsto\quad 1
  \end{cases}
  \qquad\quad
  \D_{\Kh}:
  \begin{cases}
    {\bf 1}&\mapsto\quad
 {\bf 1}\otimes X + X\otimes {\bf 1}\\
     X     &\mapsto\quad X\otimes X
  \end{cases}
$$

The geometric interpretation of $h=0$ and $t=0$
is as follows: $h=0$ corresponds to
$\raisebox{-0.2cm}{\psfig{figure=figs/ddot.eps,height=0.6cm}}=0$.
The relation
$\raisebox{-0.2cm}{\psfig{figure=figs/ddot.eps,height=0.6cm}}=0$
and the (N) relation imply
that addition of a handle is equivalent to
insertion of a dot followed by multiplication by $2$
(see Subsection~\ref{sdottedcobordisms}).
Moreover, $h=0$
implies $Y^2=X^2=t{\bf 1}$,
and therefore
$$
[2n;X,Y]_{h=0}=t^n[0;X,Y]_{h=0}=0
$$
and
$$
[2n+1;X,Y]_{h=0}=t^n[1;X,Y]_{h=0}=t^n\;.
$$
Geometrically this means that a sphere
decorated by an even number of dots
is set to zero, and a sphere decorated by $2n+1$ dots
is identified with $t^n$.
Combined with $t=0$, this implies that any
sphere containing more than one dot is set to zero.
More generally, every dotted cobordism
containing a closed component $S$ with
$\chi(S)-2\delta(S)<0$ is set to zero.

Let $\C(D):=\F_{\Kh}(\Kh(D))$ and
$\Co(D):=\F_{\Kh}(\ls D\rs)$.
The complexes $\C(D)$ and $\Co(D)$
are Khovanov's original chain complexes
(see \cite{kh:first}, where Khovanov also introduced
a more general theory,
which is related to the theory discussed
here by twisting and base change).
Let $\HC(D):=H(\C(D))$ and $\HCo(D):=H(\Co(D))$
denote the homology groups
of $\C(D)$ and $\Co(D)$, respectively.

Since $A_{\Kh}$ is graded, the chain groups
$\C^i(D)$ are graded $\Z$--modules, i.e.
$\C^i(D)=\bigoplus_{j\in\Z}\C^{i,j}(D)$, and
since
the differentials $d_{\Kh}^i:\C^i(D)\rightarrow\C^{i+1}(D)$
preserve the grading,
there is an induced grading on homology.
The isomorphism class of
$\HC(D)=\bigoplus_{i,j\in\Z}\HC^{i,j}(D)$
is an oriented link invariant, known as {\it Khovanov homology}.

Given a graded $\Z$--module $M=\bigoplus_{j\in\Z} M^j$, Khovanov
assigns a {\it graded dimension} by
$$
\dim_q(M):=\sum_j q^j\dim_{\Q}(M^j\otimes_{\Z}\Q)\; .
$$
For example, $\dim_q(A_{\Kh})=\dim_q(\Z{\bf 1}\oplus\Z X)=q+q^{-1}$.

\begin{theorem}\label{tEuler}
The graded Euler characteristic
$\chi_q(\C(D)):=\sum_i (-1)^i\dim_q(C^i(D))$
is equal to the Jones polynomial $J(D)$.
\end{theorem}
\begin{proof}
Applying $\F_{\Kh}$ to
\eqref{fformaldef}, we get
$$
\Co^i(D)=\bigoplus_{s:r(s)=i}\F_{\Kh}(D_s\{i\})\; .
$$
Since $\dim_q(\F_{\Kh}(D_s\{i\}))=
q^i\dim_q(A_{\Kh}^{\otimes n(s)})=q^i(q+q^{-1})^{n(s)}$,
this implies
$$
\chi_q(\Co(D))=\sum_i (-1)^i \sum_{s:r(s)=i}q^i(q+q^{-1})^{n(s)}=\la D\ra\; .
$$
Now the theorem follows because
$$
\chi_q(\C(D))=
(-1)^{c_-(D)}q^{c_+(D)-2c_-(D)}\chi_q(\Co(D))
$$
and because of the definition of the Jones polynomial.
\end{proof}

Alternatively, Theorem~\ref{tEuler}
can be proved
by observing that
$\chi_q(\Co(D))$ satisfies the
defining rules \eqref{fempty},
\eqref{fcirc} and \eqref{fskein}
for the Kauffman bracket. Indeed,
$\chi_q(\Co(D))$ satisfies
\begin{eqnarray*}
\chi_q(\Co(\emptyset))&=&\dim_q(\Z)=1\, ,\\
\chi_q(\Co(D\sqcup\bigcirc)) &=& (q+q^{-1})\chi_q(\Co(D))\, ,\\
\chi_q(\Co(\slashoverback))
&=&\chi_q(\Co(\smoothing))-q\chi_q(\Co(\hsmoothing))\, .
\end{eqnarray*}
The second equation follows
from Lemma~\ref{ldelooping} (Subsection~\ref{sadditiveclosure})
and the third
equation is a consequence of the mapping
cone formula~\eqref{fsaddlecone}.

\subsection[Lee's functor]{Lee's functor.}\label{sLeesfunctor}
Lee's theory is obtained from the universal theory by
setting $h=0$ and $t=1$ and by changing
coefficients to $\Q$. Hence
$$
R_{\Lee}=\Q,\qquad
A_{\Lee}=\Q[X]/(X^2-1)\; .
$$
The structure maps $\e_{\Lee}$ and $\D_{\Lee}$ are given by
$$
  \e_{\Lee}:
  \begin{cases}
    {\bf 1} &\mapsto\quad 0 \\
    X       &\mapsto\quad 1
  \end{cases}
  \qquad\quad
  \D_{\Lee}:
  \begin{cases}
    {\bf 1}&\mapsto\quad
 {\bf 1}\otimes X + X\otimes {\bf 1}\\
     X     &\mapsto\quad X\otimes X + {\bf 1}\otimes {\bf 1}
  \end{cases}
$$
Note that the grading on $A_\emptyset$ does not descend
to a grading on
$A_{\Lee}$ because the relation $t=1$ is not homogeneous.
However, $A_{\Lee}$ has the structure of a filtered Frobenius
algebra, with filtration given by
$$
0=F^3A_{\Lee}\subseteq F^1A_{\Lee}\subseteq F^{-1}A_{\Lee}=A_{\Lee}
$$
where $F^1A_{\Lee}=\Q {\bf 1}\subset A_{\Lee}$.

Lee's chain complex is defined by $\C'(D):=\F_{\Lee}(\Kh(D))$.
Since $A_{\Lee}$ is filtered, the chain groups ${\C'}^i(D)$
are filtered vector spaces, and the differentials
preserve the filtration. The filtration on $\C'(D)$
induces a filtration on the homology groups ${\HC'}^i(D):=H^i(\C'(D))$.
Explicitly, if
$$
0\;\subseteq\;\ldots\;\subseteq
\;F^{j+2}{\C'}^i(D)\;\subseteq
\;F^{j}{\C'}^i(D)\;\subseteq
\;F^{j-2}{\C'}^i(D)\;\subseteq
\;\ldots\;\subseteq
{\C'}^i(D)
$$
denotes the filtration on chain level, then
$F^j{\HC'}^i(D)\subseteq {\HC'}^i(D)$
is defined as the space of
all homology classes which have a
representative in $F^j{\C'}^i(D)$.
For a homology class $x\in{\HC'}^i(D)$,
we write $\deg(x)=j$ if $x$ has a representative in
$F^j{\C'}^i(D)$ but not in $F^{j+2}{C'}^i(D)$.

Following Lee \cite{le2},
we introduce a new basis $\{a,b\}$ for $A_{\Lee}$, defined by
$a:=X+{\bf 1}$ and $b:=X-{\bf 1}$. Written in this basis, the
expressions for the comultiplication and the multiplication
become a little bit simpler:
\begin{equation}\label{fLee}
  \D_{\Lee}: \begin{cases}
    a \mapsto a\otimes a &\\
    b \mapsto b\otimes b &
  \end{cases}
  \qquad
  m_{\Lee}: \begin{cases}
    a\otimes a\mapsto 2a &
    b\otimes b\mapsto -2b \\
    a\otimes b\mapsto 0 &
    b\otimes a\mapsto 0
  \end{cases}
\end{equation}
Note that the spaces
$\F_{\Lee}(D_s)=A_{\Lee}^{\otimes n(s)}\subset\C'(D)$
are spanned by tensor products of $a$'s and $b$'s.
It is convenient to view such tensor products
as colorings of the circles of $D_s$ by
$a$ or $b$.
We call a Kauffman state $D_s$,
equipped with such a coloring,
an {\it enhanced Kauffman state}\footnote{%
The notion of enhanced Kauffman states was
introduced by O.~Viro \cite{vi} in a slightly different
context.}.
Since the vector space
$\C'(D)$ is the direct sum
$\C'(D)=\bigoplus\F_{\Lee}(D_s)$,
the enhanced Kauffman states of $D$
provide a basis for $\C'(D)$.
Written in this basis, the differential
of Lee's complex takes an easy
form, which is essentially given
by \eqref{fLee}.

%%%%%%%%%%%%%%% Rasmussen invariant for links %%%%%%%%%%%%%%%%%%%%%
\setcounter{footnotebuffer}{\value{footnote}}
\chapter{Rasmussen invariant for links}\label{cRasmussenforlinks}
\setcounter{footnote}{\value{footnotebuffer}}

In this chapter, we give a new proof
of a theorem due to E.~S.~Lee, which states
that the Lee homology of an $n$--component link
has dimension $2^n$ (see \cite{we:trees},\cite{bm}
for similar proofs).
Then we define Rasmussen's invariant for links
and give examples where this invariant
is a stronger obstruction to sliceness
than the multivariable Levine--Tristram
signature.

\section{Canonical generators for Lee homology}
\label{scanonical}
\noindent
Let $L$ be a link with $n$ components
and let $D$ be a diagram of $L$.
According to Subsection~\ref{sLeesfunctor},
the enhanced Kauffman states
of $D$ provide a basis for $\C'(D)$.
In \cite{le2}, Lee used this basis
to construct a bijection
between generators of $\HC'(D)$ and
possible orientations of $L$.

This bijection can be described as follows.
Given an orientation of $L$,
we smoothen all crossings of $D$
in the way consistent with the orientation $o$.
The result is a Kauffman state $D_o$
whose circles are oriented.
We can turn $D_o$ into an enhanced Kauffman
state, as follows.
First,
we color the regions between the
circles of $D_o$ alternately
black and white, so that the unbounded region
is white, and such that
any two adjacent regions are oppositely colored.
Then
we color each oriented circle of $D_o$ with
$a$ or $b$ depending on whether region to its right
is black or white.
We denote the resulting enhanced Kauffman state
by $\s_o$
(cf. \cite{ra}).

\begin{theorem}\label{tdegeneration}
The homology classes $[\s_o]$ form a basis
for Lee homology $\HC'(L)$.
In particular, if $L$ has $n$ components, then
there are $2^n$ possible orientations $o$,
and hence the dimension of $\HC'(L)$ equals
$2^n$.
\end{theorem}
\begin{proof}
The proof is based on
admissible edge--colorings of $D$.
By an {\it admissible edge--coloring}, we
mean a coloring of the edges of $D$ by
the colors $a$ or $b$, such that
every crossing of $D$ admits a smoothing consistent with the
coloring. We say that an
admissible edge--colorings is of
{\it Type~I} if at least
one of the crossings is one--colored
(i.e. all four edges touching at the crossing
have the same color), and of {\it Type~II}
if all crossings are two--colored.

Given an admissible edge--coloring $c$,
we denote by
$V(c)$ the subspace of $\C'(D)$
generated by all enhanced Kauffman states whose
circles are colored in agreement with $c$.
Since Lee's differential preserves the
colors (see \eqref{fLee}),
$V(c)$ is actually
a subcomplex.
Hence we have a decomposition
$$
\HC'(D)=\bigoplus_
{\mbox{{\scriptsize $c$ admissible}}} H(V(c))
$$
where $H(V(c))$ denotes the homology of $V(c)$.
The spaces $H(V(c))$ can be computed explicitly,
as follows.

First, assume that $c$ is of Type~I.
Select a one--colored crossing.
Since both smoothings of this crossing
are consistent with $c$,
the subcomplex
$V(c)$ is isomorphic to the
mapping cone of a chain transformation
between
the two smoothings.
A look at \eqref{fLee} shows that
this chain transformation is an isomorphism.
Hence $V(c)$ is contractible and
consequently $H(V(c))=0$.

Now assume that $c$ is of Type~II.
Then there is a unique enhanced
Kauffman state $\s_c$ consistent with $c$,
and therefore $H(V(c))=V(c)=\Q\s_c$.

To complete the proof, one has to check that the
$\s_c$ arising from Type~II colorings
are precisely the canonical generators $\s_o$.
The proof of this fact is easy and
therefore omitted.
\end{proof}

\begin{remark}
Note that
the decomposition $\HC'(D)=\bigoplus H(V(c))$
does not respect the filtration of $\HC'(D)$.
\end{remark}

\section{The generalized Rasmussen invariant}
\noindent
Let $L$ be an oriented link with diagram $D$,
and let $[\s_o]$ and $[\s_{\bar o}]$
the canonical
generators of the Lee homology corresponding
to the orientation of $L$ and to the opposite
orientation, respectively.

By Lemma~3.5 in \cite{ra}, the filtered degrees of
$[\s_o+\s_{\bar o}]$ and $[\s_o-\s_{\bar o}]$ differ by
two modulo 4.
Further, we can show that they  differ by exactly two.
(Indeed, multiplying by $X\in A_{\Lee}$ at any fixed
edge of $D$ induces an automorphism of $\C'(D)$ of
filtered degree $-2$, which interchanges
$[\s_o+\s_{\bar o}]$ and $[\s_o-\s_{\bar o}]$.
The {\it Rasmussen invariant} $s(L)$ of the link $L$
is given by
$$s(L):=\frac{\deg([\s_o+\s_{\bar o}])+\deg([\s_o-\s_{\bar o}])}{2}.$$
Note that $s(L)={\rm min}(\deg([\s_o+\s_{\bar o}]),
\deg([\s_o-\s_{\bar o}]))+1$ and
that the Rasmussen invariant of the $n$--component unlink
is $1-n$.

%%% Properties %%%
Let $S$ be a link cobordism from $L_1$ to $L_2$ such that
every connected component of $S$ has a boundary in $L_1$.
Then the Rasmussen estimate generalizes to
\begin{equation}\label{re}
|s(L_2)-s(L_1)|\leq -\chi(S)\; .
\end{equation}
Indeed, arguing
as in \cite{ra} we obtain the estimate $s(L_2)\geq s(L_1)+\chi(S)$.
By reflecting $S\subset \R^3\times [0,1]$ along $\R^3\times\{1/2\}$,
we obtain a cobordism from $L_2$ to $L_1$ with the same
Euler characteristic as $S$.
This gives us the estimate $s(L_1)\geq s(L_2)+\chi(S)$.

\begin{lemma}\label{lproperties}
Let $\bar L$ be the mirror image of $L$ and $\#$, $\sqcup$ denote the
connected sum and the disjoint union, respectively. Then
\begin{eqnarray}
\label{eunion} s(L_1\sqcup L_2)&=&s(L_1)+s(L_2)-1 \\
\label{esum} s(L_1)+s(L_2)-2&\leq&s(L_1\#L_2)\;\:\leq\;\: s(L_1)+s(L_2)\\
 \label{emirror} -2|L|+2&\leq &s(L)+s(\bar L)\;\:\leq\;\: 2
\end{eqnarray}
\end{lemma}
Here, $|L|$ denotes the number of components of $L$.
Note that the first inequality of \eqref{emirror} becomes
an equality if $L$ is an unlink.
In the case where
$L_1$, $L_2$ and $L$ are knots, the second
inequality of \eqref{esum} and the first
inequality of \eqref{emirror} are equalities (see \cite{ra}).
\vsp

\begin{proofl}
Let $o_1,o_2$ and $o$ denote the orientations of
$L_1,L_2$ and $L_1\sqcup L_2$, respectively.
The filtered modules $\C'(L_1\sqcup L_2)$ and
$\C'(L_1)\otimes \C'(L_2)$
are isomorphic by an isomorphism
which sends $\s_o$ to $\s_{o_1}\otimes \s_{o_2}$.
Hence \eqref{eunion} follows
from $\deg([\s_o])={\rm min}(\deg([\s_o+\s_{\bar o}]),
\deg([\s_o-\s_{\bar o}]))=s(L_1\sqcup L_2)-1$ and
$\deg([\s_{o_i}])={\rm min}(\deg([\s_{o_i}+\s_{\bar {o_i}}]),
\deg([\s_{o_i}-\s_{\bar {o_i}}]))=s(L_i)-1$
(cf. \cite[Corollary 3.6]{ra}).
\eqref{esum} follows from \eqref{re} and
\eqref{eunion} because $L_1\sqcup L_2$ and $L_1\# L_2$ are related by a
saddle cobordism. Similarly, \eqref{emirror}
can be deduced from \eqref{re}
and \eqref{eunion} because there is a cobordism, consisting of $|L|$
saddle cobordisms, which connects $L\sqcup\bar{L}$ to the
$|L|$--component unlink.
\end{proofl}

\section{Obstructions to sliceness}
\noindent
A knot $K\subset\R^3\times\{0\}$ is called a
{\it slice knot} if it bounds a smooth disk
$S\subset\R^3\times (-\infty,0]$.
The notion of sliceness admits different generalizations to links.
We say that an oriented link $L$ is slice in
{\it the weak sense} if there exists
an oriented smooth connected
surface $S\subset \R^3\times (-\infty,0]$
of genus zero, such that $\partial S=L$.
$L$ is slice in {\it the strong sense}
if every component bounds a smooth
disk in $\R^3\times (-\infty,0]$
and all these disks are disjoint.
Recently, D.~Cimasoni and V.~Florens \cite{cf} unified
different notions of sliceness by introducing colored links.

The Rasmussen invariant of links is an obstruction
to  sliceness.

\begin{lemma}
Let $L$ be slice in the weak sense, then
$$|s(L)|\leq  |L|-1 .$$
\end{lemma}
\begin{proof}
If $L$ is slice in the weak sense,
then there exist an oriented genus $0$ cobordism from $L$ to the unknot.
Applying \eqref{re} to this cobordism we get the result.
\end{proof}

The multivariable Levine--Tristram signature defined in \cite{cf}
is also an obstruction to sliceness. However,
for knots with trivial Alexander polynomial,
the Levine--Tristram signature is constant and equal to the
ordinary signature.
Therefore, for a disjoint union of such knots
the Rasmussen link invariant is often  a
better obstruction than the multivariable signature.
Using Shumakovitch's
list of knots with trivial Alexander polynomial,
but non--trivial Rasmussen invariant \cite{sh3}  and {\it Knotscape},
one can easily construct examples. E.g.
the multivariable signature of
$K_{15n_{28998}}\sqcup K_{15n_{40132}}\sqcup K_{13n_{1496}}$ vanishes
identically,
however
$s(K_{15n_{28998}}\sqcup K_{15n_{40132}}\sqcup K_{13n_{1496}})=4>3-1$,
hence this split link
is not slice in the weak sense.
Similarly, the Rasmussen invariant, but not the signature,
 is an obstruction to
sliceness for the following split links:
$K_{15n_{113775}}\sqcup K_{14n_{7708}}$, $K_{15n_{58433}}\sqcup
K_{15n_{58501}}$, etc.

%%%%%%%%%%%%%%%%%%%%%%%%%%%% Conway Mutation %%%%%%%%%%%%%%%%%%%%%%%%%%%%
\setcounter{footnotebuffer}{\value{footnote}}
\chapter{Conway mutation}\label{cConwaymutation}
\setcounter{footnote}{\value{footnotebuffer}}
In this chapter, we present an easy example of mutant links
with different Khovanov homology.
The existence of such an example is important because it
shows that Khovanov homology cannot be defined with a skein
rule similar to the skein relation for the Jones polynomial.

\section{Definition}
\noindent
The mutation of links was originally defined in \cite{co}. We will use the
definition given in \cite{mu}. In Figure \ref{ftangleinv}, $T$ denotes
an oriented $(2,2)$--tangle (i.e. a tangle which has four endpoints
on the dotted circle, as in Figure~\ref{figatangle}).
\begin{figure}[H]
\centerline{\psfig{figure=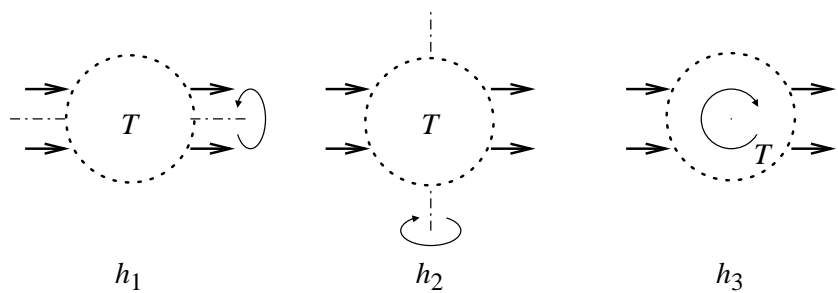,height=3.8cm}}
\caption{The half--turns $h_1$, $h_2$ and $h_3$}\label{ftangleinv}
\end{figure}
\noindent
Let $h_1$, $h_2$ and $h_3$ be the
half--turns about the indicated axes.
Define three involutions $\rho_1$, $\rho_2$ and $\rho_3$ on the
set of oriented $(2,2)$--tangles by
$\rho_1T:=h_1(T)$, $\rho_2T:=-h_2(T)$ and $\rho_3T:=-h_3(T)$
(where $-h_2(T)$ and $-h_3(T)$ are the oriented $2$--tangles
\begin{figure}[h]
\centerline{\psfig{figure=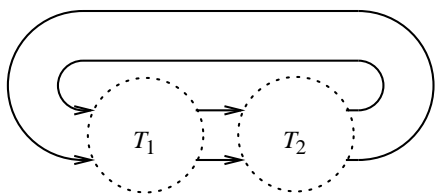,height=2.5cm}}
\caption{The closure of the composition of $T_1$ and $T_2$}\label{ftanglecom}
\end{figure}
obtained from $h_2(T)$ and $h_3(T)$ by reversing the
orientations of all strings).
For two oriented (2,2)--tangles $T_1$ and
$T_2$, denote by $T_1 T_2$ the composition of $T_1$ and $T_2$ and by
$(T_1 T_2)^\wedge$ the closure of $T_1 T_2$  (see Figure \ref{ftanglecom}).

Two oriented links $L$ and $L'$ are called {\it Conway mutants} if
there are two oriented $(2,2)$--tangles $T_1$ and $T_2$ such that for an
involution  $\rho_i$
 $(i=1,2,3)$ the links $L$ and $L'$ are respectively isotopic to $(T_1
 T_2)^\wedge$
 and $(T_1\rho_i T_2)^\wedge$.

\begin{theorem} \label{tconway}
Let $L$ and $L'$ be Conway mutants. Then $L$ and $L'$ are skein equivalent.
\end{theorem}
\begin{proof} The proof goes by induction on the number
$c$ of crossings of $T_2$. For $c\leq 1$, $T_2$ and $\rho_i T_2$ are isotopic,
whence $L\sim L'$. For $c>1$, modify a crossing of $T_2$ to obtain a skein
triple of tangles $(T_+,T_-,T_0)$ (with either $T_+=T_2$ or $T_-=T_2$, depending
on whether the crossing is positive or negative). Denote by $(L_+,L_-,L_0)$ and
$(L'_+,L'_-,L'_0)$ the skein triples corresponding to $(T_+,T_-,T_0)$ and
$(\rho_i T_+,\rho_i T_-,\rho_i T_0)$ respectively (i.e. $L_+=(T_1T_+)^\wedge$,
$L_-=(T_1T_-)^\wedge$ and so on). By induction, $L_0\sim L'_0$. Therefore, by
the definition of skein equivalence, $L_+\sim L'_+$ if and only if
$L_-\sim L'_-$. In other words, switching a crossing of $T_2$ does not affect
the truth or falsity of the assertion. Since $T_2$ can be untied by switching
crossings, we are back in the case $c\leq 1$.
\end{proof}

\begin{corollary}\label{cJonesmutation}
The Jones polynomial is invariant under Conway mutation.
\end{corollary}

\section[Mutation non--invariance]%
{Mutation non--invariance of Khovanov homology}
\noindent
Let $P(L)$ denote the graded Poincar\'e polynomial
of the complex $\C(L)$, i.e. let
$$
P(L)(t,q):=\sum_{i,j}t^iq^j\dim_{\Q}(\HC^{i,j}(L)\otimes\Q)\;
\in\Z[t^{\pm 1},q^{\pm 1}]\; . $$
By Theorem~\ref{tEuler},
we have $P(L)(-1,q)=J(L)(q)$, and by
Corollary~\ref{cJonesmutation},
$J(L)$ is invariant under Conway mutation.
On the other hand, the following
theorem gives examples of mutant links which are separated by
$I(L)(t):=P(L)(t,1)$.

\begin{theorem} \label{tmain}
Let $K_i$ $(i=1,2)$ be a $(2,n_i)$ torus link, with $n_i>2$.
Then the oriented links
$$
L:=\bigcirc\sqcup(K_1\# K_2)\\
\quad \text{and}\quad L':=K_1\sqcup K_2
$$
are Conway mutants with $I(L)\neq I(L')$.
Here, $\bigcirc$ denotes
the trivial knot and $K_1\# K_2$ is the connected sum of the oriented
links $K_1$ and $K_2$. Note that the connected sum is well--defined even if
$K_i$ has two components, because in this case the link $K_i$ is symmetric
in its components.
\end{theorem}
\begin{proof}
From Figure \ref{fmutant} it is apparent that $L$ and $L'$
are Conway mutants.
\begin{figure}
\centerline{\psfig{figure=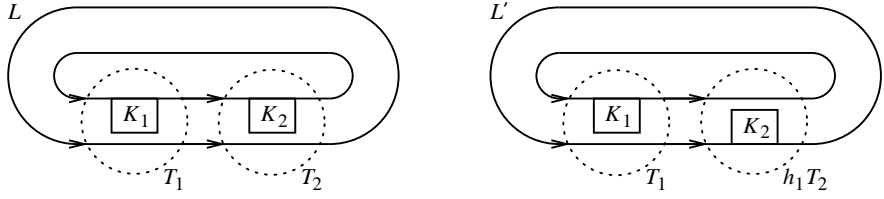,height=2.5cm}}
\caption{Figure : $L$ and $L'$ are Conway mutants}\label{fmutant}
\end{figure}
The Khovanov complex of the trivial knot is
$$
\ldots\hsp\longrightarrow\hsp 0\hsp \longrightarrow\hsp
0\hsp\longrightarrow\hsp A_{\Kh}\hsp\longrightarrow\hsp
0\hsp\longrightarrow\hsp 0\hsp\longrightarrow\hsp\ldots
$$
Since $\operatorname{rank}(A_{\Kh})=2$, we get
$I(\bigcirc)=2$, and
since $P$ is multiplicative under disjoint union
(see \cite[Proposition 33]{kh:first}),
this implies $I(L)=2 I(K_1\# K_2)$. On the
other hand, \cite[Proposition 35]{kh:first} tells us that
$$
I(K_i)=2+t^{-2}+t^{-3}+\ldots+t^{-(n_i-1)}+t^{-n_i}
$$
for odd $n_i$, and
$$
I(K_i)=2+t^{-2}+t^{-3}+\ldots+t^{-(n_i-1)}+2t^{-n_i}
$$
for even $n_i$. Since we assume $n_i>2$,
we get that
$I(K_i)$ is not divisible by $2$.
It follows that
$I(L')=I(K_1)I(K_2)$ is not divisible
by $2$, and hence $I(L')\neq I(L)$.
\end{proof}

Theorems~\ref{tconway} and \ref{tmain} immediately imply:
\begin{corollary}
The skein equivalence class of a link does not determine
its Khovanov homology.
In particular, Khovanov homology is strictly stronger
than the Jones polynomial.
\end{corollary}

\begin{remark} Theorem \ref{tmain} remains true if we allow
$(2,n_i)$ torus links
$K_i$ with $n_i<-2$ (to see this, use \cite[Corollary 11]{kh:first},
which  relates the Khovanov homology of a link to the
Khovanov homology of its mirror image).
However, the condition $|n_i|>2$ is necessary.
In fact, if one of the
$|n_i|$ is $\leq 1$, then the corresponding torus link $K_i$ is trivial and
hence $L$ and $L'$ are isotopic.
If one of  the $|n_i|$, say $|n_2|$, is equal to $2$,
then $K_2$ is a Hopf link
and hence $L$ and $L'$ are related
to $\bigcirc\sqcup K_1$ by Hopf link addition (see
Section~\ref{shopflinkaddition}).
Now it follows from Theorem~\ref{thopflinkaddition}
(Section~\ref{shopflinkaddition})
that $\HC^{i,j}(L)$ and $\HC^{i,j}(L')$
are both isomorphic to
$\HC^{i+2,j+5}(\bigcirc\sqcup K_1)\oplus
 \HC^{i,j+1}(\bigcirc\sqcup K_1)$.
\end{remark}

\begin{remark}
As yet, it is not known whether there are mutant knots ($1$--component
links) with different Khovanov homology.
An argument of D.~Bar--Natan \cite{ba:mutation},
which would show invariance of Khovanov homology under
knot mutation, was remarked to be incomplete
by the author.
\end{remark}

\section{Computer Calculations with KhoHo}
\noindent
Tables \ref{t1} and \ref{t2} show the Khovanov homology of $L$ and
$L'$ for the case $n_1=n_2=3$. The tables where generated using
A.~Shumakovitch's program \texttt{KhoHo} \cite{sh:khoho}.
The entry in the $i$--th column and the
$j$--th row looks like $\hbox{$a$[$b$]}\over\lower 3pt\hbox{$c$}$, where $a$
is the rank of the homology group $\HC^{i,j}$, $b$ the number of factors
$\Z/2\Z$ in the decomposition of $\HC^{i,j}$ into
 $p$--subgroups,
and $c$ the rank of the chain group $\HC^{i,j}$.
The numbers above the horizontal
arrows denote the ranks of the chain differentials.

In the examples,
only $2$--torsion occurs. The reader may verify that not only
the ranks but also the torsion parts of the $\HC^{i,j}$
are different for $L$ and $L'$.
The ranks of $\C^{i,j}(L)$ and $\C^{i,j}(L')$
agree because there is
a natural one--to--one correspondence between the Kauffman
states of $L$ and $L'$.

\begin{table}[b]
\begin{center} %\advance\hsize 1.4truein \advance\hoffset -0.7truein
%\advance\vsize 1.0truein \advance\voffset -0.5truein

\def\Cal#1{{\fam2#1}}

\let\TSp\thinspace
\let\NSp\negthinspace
\def\DSp{\thinspace\thinspace}
\def\QSp{\thinspace\thinspace\thinspace\thinspace}

%\nopagenumbers
\offinterlineskip
\def\gobble#1{}
\def\hline{\noalign{\hrule}}
\def\dblhline{height 0.16667em\gobble&&
&\vrule&
&\vrule&
&\vrule&
&\vrule&
&\vrule&
&\vrule&
\cr
\noalign{\hrule}}

\def\group#1#2#3#4{${\raise 1pt%
\hbox{#4#1\ifnum#2=0\else[#2]\fi}\over\lower 3pt\hbox{#4#3}}$}
\def\putdown#1{\smash{\vtop{\null\hbox{#1}}}}

\null\vfill
$$\vbox{\ialign{%
\vrule\TSp\vrule #\strut&\DSp\hfil #\DSp\vrule\TSp\vrule&\kern.75em
\TSp\hfil #\hfil\TSp&\hfil #\hfil&
\TSp\hfil #\hfil\TSp&\hfil #\hfil&
\TSp\hfil #\hfil\TSp&\hfil #\hfil&
\TSp\hfil #\hfil\TSp&\hfil #\hfil&
\TSp\hfil #\hfil\TSp&\hfil #\hfil&
\TSp\hfil #\hfil\TSp&\hfil #\hfil&
\TSp\hfil #\hfil\TSp\kern.75em\vrule\TSp\vrule\cr
\hline\dblhline
height 11pt depth 4pt&&
-6&\vrule&
-5&\vrule&
-4&\vrule&
-3&\vrule&
-2&\vrule&
-1&\vrule&
0\cr\hline\dblhline
height 0.2em depth 0.2em\gobble&&
&\vrule depth0pt&
&\vrule depth0pt&
&\vrule depth0pt&
&\vrule depth0pt&
&\vrule depth0pt&
&\vrule depth0pt&
\cr
height 0pt\gobble&&
&&
&&
&&
&&
&&
&&\cr
&-2&
&\smash{\vrule height 15pt}&
&\smash{\vrule height 15pt}&
&\smash{\vrule height 15pt}&
&\smash{\vrule height 15pt}&
&\smash{\vrule height 15pt}&
&\smash{\vrule height 15pt}&
\group{1}{0}{1}{\bf}\cr
height 0.3em\gobble&&
&\smash{\vrule height 10pt}&
&\smash{\vrule height 10pt}&
&\smash{\vrule height 10pt}&
&\smash{\vrule height 10pt}&
&\smash{\vrule height 10pt}&
&\smash{\vrule height 10pt}&
\cr
\hline
height 0.2em depth 0.2em\gobble&&
&\vrule depth0pt&
&\vrule depth0pt&
&\vrule depth0pt&
&\vrule depth0pt&
&\vrule depth0pt&
&\vrule depth0pt&
\cr
height 0pt\gobble&&
&&
&&
&&
&\QSp\putdown{2}\QSp&
&\QSp\putdown{4}\QSp&
&\QSp\putdown{2}\QSp&\cr
&-4&
&\smash{\vrule height 15pt}&
&\smash{\vrule height 15pt}&
&\smash{\vrule height 15pt}&
\group{0}{0}{2}{}&\rightarrowfill&
\group{0}{0}{6}{}&\rightarrowfill&
\group{0}{0}{6}{}&\rightarrowfill&
\group{2}{0}{4}{\bf}\cr
height 0.3em\gobble&&
&\smash{\vrule height 10pt}&
&\smash{\vrule height 10pt}&
&\smash{\vrule height 10pt}&
&\smash{\vrule height 10pt}&
&\smash{\vrule height 10pt}&
&\smash{\vrule height 10pt}&
\cr
\hline
height 0.2em depth 0.2em\gobble&&
&\vrule depth0pt&
&\vrule depth0pt&
&\vrule depth0pt&
&\vrule depth0pt&
&\vrule depth0pt&
&\vrule depth0pt&
\cr
height 0pt\gobble&&
&\QSp\putdown{1}\QSp&
&\QSp\putdown{5}\QSp&
&\DSp\putdown{10}\DSp&
&\DSp\putdown{18}\DSp&
&\DSp\putdown{13}\DSp&
&\QSp\putdown{5}\QSp&\cr
&-6&
\group{0}{0}{1}{}&\rightarrowfill&
\group{0}{0}{6}{}&\rightarrowfill&
\group{0}{0}{15}{}&\rightarrowfill&
\group{0}{0}{28}{}&\rightarrowfill&
\group{2}{0}{33}{\bf}&\rightarrowfill&
\group{0}{0}{18}{}&\rightarrowfill&
\group{1}{0}{6}{\bf}\cr
height 0.3em\gobble&&
&\smash{\vrule height 10pt}&
&\smash{\vrule height 10pt}&
&\smash{\vrule height 10pt}&
&\smash{\vrule height 10pt}&
&\smash{\vrule height 10pt}&
&\smash{\vrule height 10pt}&
\cr
\hline
height 0.2em depth 0.2em\gobble&&
&\vrule depth0pt&
&\vrule depth0pt&
&\vrule depth0pt&
&\vrule depth0pt&
&\vrule depth0pt&
&\vrule depth0pt&
\cr
height 0pt\gobble&&
&\QSp\putdown{6}\QSp&
&\DSp\putdown{24}\DSp&
&\DSp\putdown{36}\DSp&
&\DSp\putdown{38}\DSp&
&\DSp\putdown{14}\DSp&
&\QSp\putdown{4}\QSp&\cr
&-8&
\group{0}{0}{6}{}&\rightarrowfill&
\group{0}{0}{30}{}&\rightarrowfill&
\group{0}{0}{60}{}&\rightarrowfill&
\group{0}{0}{74}{}&\rightarrowfill&
\group{2}{2}{54}{\bf}&\rightarrowfill&
\group{0}{0}{18}{}&\rightarrowfill&
\group{0}{0}{4}{}\cr
height 0.3em\gobble&&
&\smash{\vrule height 10pt}&
&\smash{\vrule height 10pt}&
&\smash{\vrule height 10pt}&
&\smash{\vrule height 10pt}&
&\smash{\vrule height 10pt}&
&\smash{\vrule height 10pt}&
\cr
\hline
height 0.2em depth 0.2em\gobble&&
&\vrule depth0pt&
&\vrule depth0pt&
&\vrule depth0pt&
&\vrule depth0pt&
&\vrule depth0pt&
&\vrule depth0pt&
\cr
height 0pt\gobble&&
&\DSp\putdown{15}\DSp&
&\DSp\putdown{45}\DSp&
&\DSp\putdown{44}\DSp&
&\DSp\putdown{28}\DSp&
&\QSp\putdown{5}\QSp&
&\QSp\putdown{1}\QSp&\cr
&-10&
\group{0}{0}{15}{}&\rightarrowfill&
\group{0}{0}{60}{}&\rightarrowfill&
\group{1}{0}{90}{\bf}&\rightarrowfill&
\group{2}{0}{74}{\bf}&\rightarrowfill&
\group{0}{2}{33}{\bf}&\rightarrowfill&
\group{0}{0}{6}{}&\rightarrowfill&
\group{0}{0}{1}{}\cr
height 0.3em\gobble&&
&\smash{\vrule height 10pt}&
&\smash{\vrule height 10pt}&
&\smash{\vrule height 10pt}&
&\smash{\vrule height 10pt}&
&\smash{\vrule height 10pt}&
&\smash{\vrule height 10pt}&
\cr
\hline
height 0.2em depth 0.2em\gobble&&
&\vrule depth0pt&
&\vrule depth0pt&
&\vrule depth0pt&
&\vrule depth0pt&
&\vrule depth0pt&
&\vrule depth0pt&
\cr
height 0pt\gobble&&
&\DSp\putdown{20}\DSp&
&\DSp\putdown{39}\DSp&
&\DSp\putdown{20}\DSp&
&\QSp\putdown{6}\QSp&
&&
&&\cr
&-12&
\group{0}{0}{20}{}&\rightarrowfill&
\group{1}{0}{60}{\bf}&\rightarrowfill&
\group{1}{1}{60}{\bf}&\rightarrowfill&
\group{2}{0}{28}{\bf}&\rightarrowfill&
\group{0}{0}{6}{}&\smash{\vrule height 15pt}&
&\smash{\vrule height 15pt}&
\cr
height 0.3em\gobble&&
&\smash{\vrule height 10pt}&
&\smash{\vrule height 10pt}&
&\smash{\vrule height 10pt}&
&\smash{\vrule height 10pt}&
&\smash{\vrule height 10pt}&
&\smash{\vrule height 10pt}&
\cr
\hline
height 0.2em depth 0.2em\gobble&&
&\vrule depth0pt&
&\vrule depth0pt&
&\vrule depth0pt&
&\vrule depth0pt&
&\vrule depth0pt&
&\vrule depth0pt&
\cr
height 0pt\gobble&&
&\DSp\putdown{15}\DSp&
&\DSp\putdown{13}\DSp&
&\QSp\putdown{2}\QSp&
&&
&&
&&\cr
&-14&
\group{0}{0}{15}{}&\rightarrowfill&
\group{2}{1}{30}{\bf}&\rightarrowfill&
\group{0}{1}{15}{\bf}&\rightarrowfill&
\group{0}{0}{2}{}&\smash{\vrule height 15pt}&
&\smash{\vrule height 15pt}&
&\smash{\vrule height 15pt}&
\cr
height 0.3em\gobble&&
&\smash{\vrule height 10pt}&
&\smash{\vrule height 10pt}&
&\smash{\vrule height 10pt}&
&\smash{\vrule height 10pt}&
&\smash{\vrule height 10pt}&
&\smash{\vrule height 10pt}&
\cr
\hline
height 0.2em depth 0.2em\gobble&&
&\vrule depth0pt&
&\vrule depth0pt&
&\vrule depth0pt&
&\vrule depth0pt&
&\vrule depth0pt&
&\vrule depth0pt&
\cr
height 0pt\gobble&&
&\QSp\putdown{5}\QSp&
&&
&&
&&
&&
&&\cr
&-16&
\group{1}{0}{6}{\bf}&\rightarrowfill&
\group{1}{1}{6}{\bf}&\smash{\vrule height 15pt}&
&\smash{\vrule height 15pt}&
&\smash{\vrule height 15pt}&
&\smash{\vrule height 15pt}&
&\smash{\vrule height 15pt}&
\cr
height 0.3em\gobble&&
&\smash{\vrule height 10pt}&
&\smash{\vrule height 10pt}&
&\smash{\vrule height 10pt}&
&\smash{\vrule height 10pt}&
&\smash{\vrule height 10pt}&
&\smash{\vrule height 10pt}&
\cr
\hline
height 0.2em depth 0.2em\gobble&&
&\vrule depth0pt&
&\vrule depth0pt&
&\vrule depth0pt&
&\vrule depth0pt&
&\vrule depth0pt&
&\vrule depth0pt&
\cr
height 0pt\gobble&&
&&
&&
&&
&&
&&
&&\cr
&-18&
\group{1}{0}{1}{\bf}&\smash{\vrule height 15pt}&
&\smash{\vrule height 15pt}&
&\smash{\vrule height 15pt}&
&\smash{\vrule height 15pt}&
&\smash{\vrule height 15pt}&
&\smash{\vrule height 15pt}&
\cr
height 0.3em\gobble&&
&\smash{\vrule height 10pt}&
&\smash{\vrule height 10pt}&
&\smash{\vrule height 10pt}&
&\smash{\vrule height 10pt}&
&\smash{\vrule height 10pt}&
&\smash{\vrule height 10pt}&
\cr
\hline
\dblhline
}}$$

%\medskip\centerline{Dimensions of $\Cal{H}^{i,j}$ and
%$\Cal{C}^{i,j}$ and ranks of chain differentials}
%\medskip
%\centerline{for disjoint union of unknot and granny-knot}
%\vfill
%\bye

\caption{\vspace {4pt} Ranks of $\HC^{i,j}$ and $\C^{i,j}$ and ranks
of the differentials for the disjoint union of the unknot and the
granny--knot}\label{t1}\end{center}
\end{table}

\begin{table}
\begin{center}%\advance\hsize 1.4truein \advance\hoffset -0.7truein
%\advance\vsize 1.0truein \advance\voffset -0.5truein

\def\Cal#1{{\fam2#1}}

\let\TSp\thinspace
\let\NSp\negthinspace
\def\DSp{\thinspace\thinspace}
\def\QSp{\thinspace\thinspace\thinspace\thinspace}

%\nopagenumbers
\offinterlineskip

\def\gobble#1{}
\def\hline{\noalign{\hrule}}
\def\dblhline{height 0.16667em\gobble&&
&\vrule&
&\vrule&
&\vrule&
&\vrule&
&\vrule&
&\vrule&
\cr
\noalign{\hrule}}

\def\group#1#2#3#4{${\raise 1pt%
\hbox{#4#1\ifnum#2=0\else[#2]\fi}\over\lower 3pt\hbox{#4#3}}$}
\def\putdown#1{\smash{\vtop{\null\hbox{#1}}}}

\null\vfill
$$\vbox{\ialign{%
\vrule\TSp\vrule #\strut&\DSp\hfil #\DSp\vrule\TSp\vrule&\kern.75em
\TSp\hfil #\hfil\TSp&\hfil #\hfil&
\TSp\hfil #\hfil\TSp&\hfil #\hfil&
\TSp\hfil #\hfil\TSp&\hfil #\hfil&
\TSp\hfil #\hfil\TSp&\hfil #\hfil&
\TSp\hfil #\hfil\TSp&\hfil #\hfil&
\TSp\hfil #\hfil\TSp&\hfil #\hfil&
\TSp\hfil #\hfil\TSp\kern.75em\vrule\TSp\vrule\cr
\hline\dblhline
height 11pt depth 4pt&&
-6&\vrule&
-5&\vrule&
-4&\vrule&
-3&\vrule&
-2&\vrule&
-1&\vrule&
0\cr\hline\dblhline
height 0.2em depth 0.2em\gobble&&
&\vrule depth0pt&
&\vrule depth0pt&
&\vrule depth0pt&
&\vrule depth0pt&
&\vrule depth0pt&
&\vrule depth0pt&
\cr
height 0pt\gobble&&
&&
&&
&&
&&
&&
&&\cr
&-2&
&\smash{\vrule height 15pt}&
&\smash{\vrule height 15pt}&
&\smash{\vrule height 15pt}&
&\smash{\vrule height 15pt}&
&\smash{\vrule height 15pt}&
&\smash{\vrule height 15pt}&
\group{1}{0}{1}{\bf}\cr
height 0.3em\gobble&&
&\smash{\vrule height 10pt}&
&\smash{\vrule height 10pt}&
&\smash{\vrule height 10pt}&
&\smash{\vrule height 10pt}&
&\smash{\vrule height 10pt}&
&\smash{\vrule height 10pt}&
\cr
\hline
height 0.2em depth 0.2em\gobble&&
&\vrule depth0pt&
&\vrule depth0pt&
&\vrule depth0pt&
&\vrule depth0pt&
&\vrule depth0pt&
&\vrule depth0pt&
\cr
height 0pt\gobble&&
&&
&&
&&
&\QSp\putdown{2}\QSp&
&\QSp\putdown{4}\QSp&
&\QSp\putdown{2}\QSp&\cr
&-4&
&\smash{\vrule height 15pt}&
&\smash{\vrule height 15pt}&
&\smash{\vrule height 15pt}&
\group{0}{0}{2}{}&\rightarrowfill&
\group{0}{0}{6}{}&\rightarrowfill&
\group{0}{0}{6}{}&\rightarrowfill&
\group{2}{0}{4}{\bf}\cr
height 0.3em\gobble&&
&\smash{\vrule height 10pt}&
&\smash{\vrule height 10pt}&
&\smash{\vrule height 10pt}&
&\smash{\vrule height 10pt}&
&\smash{\vrule height 10pt}&
&\smash{\vrule height 10pt}&
\cr
\hline
height 0.2em depth 0.2em\gobble&&
&\vrule depth0pt&
&\vrule depth0pt&
&\vrule depth0pt&
&\vrule depth0pt&
&\vrule depth0pt&
&\vrule depth0pt&
\cr
height 0pt\gobble&&
&\QSp\putdown{1}\QSp&
&\QSp\putdown{5}\QSp&
&\DSp\putdown{10}\DSp&
&\DSp\putdown{18}\DSp&
&\DSp\putdown{13}\DSp&
&\QSp\putdown{5}\QSp&\cr
&-6&
\group{0}{0}{1}{}&\rightarrowfill&
\group{0}{0}{6}{}&\rightarrowfill&
\group{0}{0}{15}{}&\rightarrowfill&
\group{0}{0}{28}{}&\rightarrowfill&
\group{2}{0}{33}{\bf}&\rightarrowfill&
\group{0}{0}{18}{}&\rightarrowfill&
\group{1}{0}{6}{\bf}\cr
height 0.3em\gobble&&
&\smash{\vrule height 10pt}&
&\smash{\vrule height 10pt}&
&\smash{\vrule height 10pt}&
&\smash{\vrule height 10pt}&
&\smash{\vrule height 10pt}&
&\smash{\vrule height 10pt}&
\cr
\hline
height 0.2em depth 0.2em\gobble&&
&\vrule depth0pt&
&\vrule depth0pt&
&\vrule depth0pt&
&\vrule depth0pt&
&\vrule depth0pt&
&\vrule depth0pt&
\cr
height 0pt\gobble&&
&\QSp\putdown{6}\QSp&
&\DSp\putdown{24}\DSp&
&\DSp\putdown{36}\DSp&
&\DSp\putdown{38}\DSp&
&\DSp\putdown{14}\DSp&
&\QSp\putdown{4}\QSp&\cr
&-8&
\group{0}{0}{6}{}&\rightarrowfill&
\group{0}{0}{30}{}&\rightarrowfill&
\group{0}{0}{60}{}&\rightarrowfill&
\group{0}{0}{74}{}&\rightarrowfill&
\group{2}{2}{54}{\bf}&\rightarrowfill&
\group{0}{0}{18}{}&\rightarrowfill&
\group{0}{0}{4}{}\cr
height 0.3em\gobble&&
&\smash{\vrule height 10pt}&
&\smash{\vrule height 10pt}&
&\smash{\vrule height 10pt}&
&\smash{\vrule height 10pt}&
&\smash{\vrule height 10pt}&
&\smash{\vrule height 10pt}&
\cr
\hline
height 0.2em depth 0.2em\gobble&&
&\vrule depth0pt&
&\vrule depth0pt&
&\vrule depth0pt&
&\vrule depth0pt&
&\vrule depth0pt&
&\vrule depth0pt&
\cr
height 0pt\gobble&&
&\DSp\putdown{15}\DSp&
&\DSp\putdown{45}\DSp&
&\DSp\putdown{44}\DSp&
&\DSp\putdown{28}\DSp&
&\QSp\putdown{5}\QSp&
&\QSp\putdown{1}\QSp&\cr
&-10&
\group{0}{0}{15}{}&\rightarrowfill&
\group{0}{0}{60}{}&\rightarrowfill&
\group{1}{0}{90}{\bf}&\rightarrowfill&
\group{2}{0}{74}{\bf}&\rightarrowfill&
\group{0}{2}{33}{\bf}&\rightarrowfill&
\group{0}{0}{6}{}&\rightarrowfill&
\group{0}{0}{1}{}\cr
height 0.3em\gobble&&
&\smash{\vrule height 10pt}&
&\smash{\vrule height 10pt}&
&\smash{\vrule height 10pt}&
&\smash{\vrule height 10pt}&
&\smash{\vrule height 10pt}&
&\smash{\vrule height 10pt}&
\cr
\hline
height 0.2em depth 0.2em\gobble&&
&\vrule depth0pt&
&\vrule depth0pt&
&\vrule depth0pt&
&\vrule depth0pt&
&\vrule depth0pt&
&\vrule depth0pt&
\cr
height 0pt\gobble&&
&\DSp\putdown{20}\DSp&
&\DSp\putdown{40}\DSp&
&\DSp\putdown{20}\DSp&
&\QSp\putdown{6}\QSp&
&&
&&\cr
&-12&
\group{0}{0}{20}{}&\rightarrowfill&
\group{0}{0}{60}{}&\rightarrowfill&
\group{0}{2}{60}{\bf}&\rightarrowfill&
\group{2}{0}{28}{\bf}&\rightarrowfill&
\group{0}{0}{6}{}&\smash{\vrule height 15pt}&
&\smash{\vrule height 15pt}&
\cr
height 0.3em\gobble&&
&\smash{\vrule height 10pt}&
&\smash{\vrule height 10pt}&
&\smash{\vrule height 10pt}&
&\smash{\vrule height 10pt}&
&\smash{\vrule height 10pt}&
&\smash{\vrule height 10pt}&
\cr
\hline
height 0.2em depth 0.2em\gobble&&
&\vrule depth0pt&
&\vrule depth0pt&
&\vrule depth0pt&
&\vrule depth0pt&
&\vrule depth0pt&
&\vrule depth0pt&
\cr
height 0pt\gobble&&
&\DSp\putdown{15}\DSp&
&\DSp\putdown{13}\DSp&
&\QSp\putdown{2}\QSp&
&&
&&
&&\cr
&-14&
\group{0}{0}{15}{}&\rightarrowfill&
\group{2}{1}{30}{\bf}&\rightarrowfill&
\group{0}{1}{15}{\bf}&\rightarrowfill&
\group{0}{0}{2}{}&\smash{\vrule height 15pt}&
&\smash{\vrule height 15pt}&
&\smash{\vrule height 15pt}&
\cr
height 0.3em\gobble&&
&\smash{\vrule height 10pt}&
&\smash{\vrule height 10pt}&
&\smash{\vrule height 10pt}&
&\smash{\vrule height 10pt}&
&\smash{\vrule height 10pt}&
&\smash{\vrule height 10pt}&
\cr
\hline
height 0.2em depth 0.2em\gobble&&
&\vrule depth0pt&
&\vrule depth0pt&
&\vrule depth0pt&
&\vrule depth0pt&
&\vrule depth0pt&
&\vrule depth0pt&
\cr
height 0pt\gobble&&
&\QSp\putdown{6}\QSp&
&&
&&
&&
&&
&&\cr
&-16&
\group{0}{0}{6}{}&\rightarrowfill&
\group{0}{2}{6}{\bf}&\smash{\vrule height 15pt}&
&\smash{\vrule height 15pt}&
&\smash{\vrule height 15pt}&
&\smash{\vrule height 15pt}&
&\smash{\vrule height 15pt}&
\cr
height 0.3em\gobble&&
&\smash{\vrule height 10pt}&
&\smash{\vrule height 10pt}&
&\smash{\vrule height 10pt}&
&\smash{\vrule height 10pt}&
&\smash{\vrule height 10pt}&
&\smash{\vrule height 10pt}&
\cr
\hline
height 0.2em depth 0.2em\gobble&&
&\vrule depth0pt&
&\vrule depth0pt&
&\vrule depth0pt&
&\vrule depth0pt&
&\vrule depth0pt&
&\vrule depth0pt&
\cr
height 0pt\gobble&&
&&
&&
&&
&&
&&
&&\cr
&-18&
\group{1}{0}{1}{\bf}&\smash{\vrule height 15pt}&
&\smash{\vrule height 15pt}&
&\smash{\vrule height 15pt}&
&\smash{\vrule height 15pt}&
&\smash{\vrule height 15pt}&
&\smash{\vrule height 15pt}&
\cr
height 0.3em\gobble&&
&\smash{\vrule height 10pt}&
&\smash{\vrule height 10pt}&
&\smash{\vrule height 10pt}&
&\smash{\vrule height 10pt}&
&\smash{\vrule height 10pt}&
&\smash{\vrule height 10pt}&
\cr
\hline
\dblhline
}}$$

%\medskip\centerline{Dimensions of $\Cal{H}^{i,j}$ and
%$\Cal{C}^{i,j}$ and ranks of chain differentials}
%\medskip
%\centerline{for disjoint union of two trefoil knots}
%\vfill
%\bye

\caption{\vspace{4pt} Ranks of $\HC^{i,j}$ and $\C^{i,j}$ and ranks
of the differentials for the disjoint union of two trefoil
 knots}\label{t2}\end{center}
\end{table}

%%%%%%%%%%%%%%%%%%%%%% The spanning tree model %%%%%%%%%%%%%%%%%%%%%%%%%%%%
\setcounter{footnotebuffer}{\value{footnote}}
\chapter{The spanning tree model}\label{cspanningtreemodel}
\setcounter{footnote}{\value{footnotebuffer}}
In \cite{th}, M.~Thistlethwaite described a relation
between the Kauffman bracket of a knot diagram $D$
and the Tutte polynomial of the Tait graph of $D$.
He showed that the Kauffman bracket admits an
expansion as a sum over terms corresponding to
spanning trees of the Tait graph.

In \cite{we:trees}, the author constructed
an analogue of this expansion for Khovanov homology.
Independently,
A.~Champanerkar and I.~Kofman \cite{ck}
proposed a similar construction,
based on a technically different argument.

In this chapter, we first
review the spanning tree expansion for
the Kauffman bracket.
Our approach is different from Thistlethwaite's,
making no explicit reference to the Tutte polynomial.
In Section~\ref{s3},
we show how our ideas lead to a
spanning tree model for the Khovanov bracket.
In the remaining sections, we give several applications,
among these a new proof of E.~S.~Lee's \cite{le1} theorem
on the support of the Khovanov homology of alternating knots,
and a short proof of a theorem on the behavior of
the Khovanov bracket under Hopf link addition.

\section{Spanning tree model for the Kauffman bracket}\label{streeKauffman}
\subsection[A simpler formula for the Kauffman bracket]%
{A simpler formula for the Kauffman bracket.}
Suppose $D$ is an unoriented link diagram
whose crossings are numbered.
Recall that the Kauffman bracket of $D$ satisfies
\begin{equation}\label{fKauffmanrewritea}
\la D\ra=\sum_{D'\in\K(D)}\la D|D'\ra\la D'\ra
\end{equation}
where $\la D|D'\ra=(-q)^{r(D,D')}$.
Formula~\eqref{fKauffmanrewritea}
can be deduced recursively from the rule
$\la\slashoverback\ra=\la\smoothing\ra-q\la\hsmoothing\ra$,
as follows:
first, we expand $\la D\ra$
as a sum of two terms by applying
$\la\slashoverback\ra=\la\smoothing\ra-q\la\hsmoothing\ra$
to crossing number $1$. Next, we expand each these
two terms by applying
$\la\slashoverback\ra=\la\smoothing\ra-q\la\hsmoothing\ra$
to crossing number $2$. Continuing like this,
we finally reach the Kauffman states and
hence recover \eqref{fKauffmanrewritea}.
The procedure is visualized in the binary
tree below.

\begin{figure}[H]
\centerline{\psfig{figure=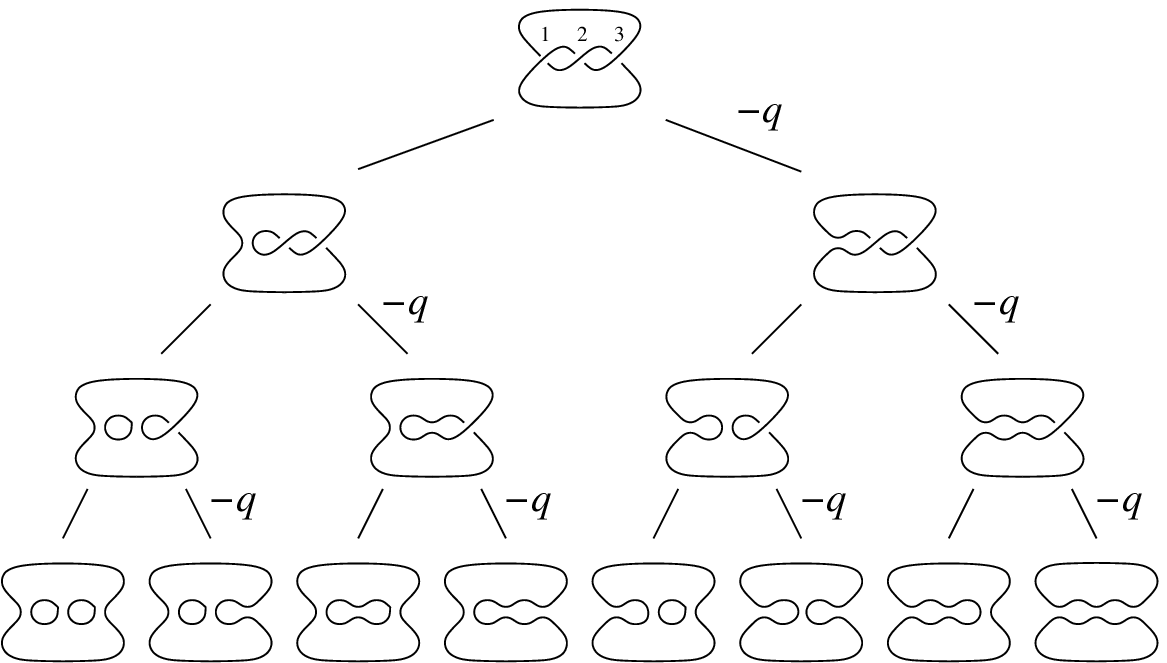,height=5.5cm}}

\caption{Binary tree used to deduce \eqref{fKauffmanrewritea}
from $\la\slashoverback\ra=\la\smoothing\ra-q\la\hsmoothing\ra$.}
\end{figure}

In case $D$ is connected, we can compute
the Kauffman bracket of $D$ more efficiently,
by modifying the above procedure as follows:
as before, we successively expand terms by
applying the relation
$\la\slashoverback\ra=\la\smoothing\ra-q\la\hsmoothing\ra$
to the crossings.
But before expanding a term,
we check the connectivity
of the two diagrams $\smoothing$ and $\hsmoothing$
appearing on the right--hand side of
$\la\slashoverback\ra=\la\smoothing\ra-q\la\hsmoothing\ra$.
If one of them is disconnected, we do not expand the
crossing $\slashoverback$ in the given term,
and instead continue with the next crossing.
The improved procedure is visualized in Figure~\ref{figimproved}.

\begin{figure}[H]
\centerline{\psfig{figure=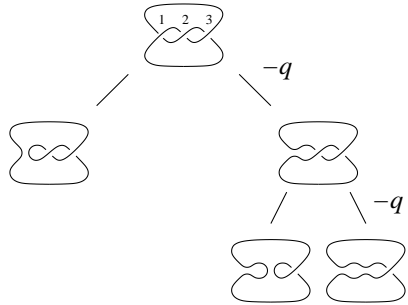,height=4cm}}

\caption{Binary tree used to deduce \eqref{fKauffmanefficient}.}
\label{figimproved}
\end{figure}
The improved procedure leads to the expansion
\begin{equation}\label{fKauffmanefficient}
\la D\ra=\sum_{D'\in\T(D)}\la D|D'\ra\la D'\ra\; .
\end{equation}
where $\T(D)$ denotes the set of all link diagrams
sitting at the leaves of the tree in Figure.
Note that $\T(D)$ depends on the
numbering of the crossings of $D$.

To turn \eqref{fKauffmanefficient} into an explicit
formula, we have to calculate the Kauffman brackets
$\la D'\ra$.
Let $D'$ be an element of $\T(D)$.
By construction,
$D'$ is connected and
every crossing of $D'$ is {\it splitting}
(i.e. connects two otherwise disconnected parts of $D'$).
Therefore, $D'$ represents the unknot
and it can be transformed into the trivial
diagram using Reidemeister move R1 only.
We call a diagram with this property {\it R1--trivial}.
After orienting $D'$ arbitrarily, we get
$J(D')=J(\bigcirc)=q+q^{-1}$ and hence
\begin{equation}\label{fKauffmantrivial}
\la D'\ra=(-1)^{c_-(D')}q^{2c_-(D')-c_+(D')}(q+q^{-1})\; .
\end{equation}
Inserting \eqref{fKauffmantrivial} into
\eqref{fKauffmanefficient}, we obtain
\begin{equation}\label{fKauffmanexplicit}
\la D\ra=\sum_{D'\in\T(D)}(-q)^{r(D,D')}
(-1)^{c_-(D')}q^{2c_-(D')-c_+(D')}(q+q^{-1})\; .
\end{equation}

Note that the set of Kauffman states
$\K(D)$ is the disjoint union of all
sets $\K(D')$, for all $D'\in\T(D)$.
We construct a map
$$
\begin{array}{ccc}
\K(D) &\longrightarrow &\T(D) \\
S       &\longmapsto &D_S
\end{array}
$$
by defining $D_S$ to be the unique element of $\T(D)$
satisfying $S\in\K(D_S)$.
Let $\K_1(D)\subset\K(D)$ denote the set of all
Kauffman states which consist of exactly one circle.
When restricted to
$\K_1(D)\subset\K(D)$,
the above map becomes a bijection.
Indeed, since $D'\in\T(D)$ is R1--trivial,
we have $\#\K_1(D')=1$ and hence $D'$ has
a unique preimage in $\K_1(D)$.

We may rewrite \eqref{fKauffmanefficient} as
\begin{equation}\label{fKauffmanefficient1}
\la D\ra=\sum_{S\in \K_1(D)}\la D|D_S\ra\la D_S\ra.
\end{equation}
Formula \eqref{fKauffmanexplicit} becomes
\begin{equation}\label{fKauffmanexplicit1}
\begin{split}
\la D\ra &=\sum_{S\in \K_1(D)}
(-q)^{r(D,D_S)}(-1)^{c_-(D_S)}q^{2c_-(D_S)-c_+(D_S)}(q+q^{-1})\\
&=\sum_{S\in \K_1(D)}
(-1)^{r(D,S)-w(D_S)}q^{r(D,S)-2w(D_S)}(q+q^{-1})
\end{split}
\end{equation}
where the second equality follows by
observing that
$r(D,D_S)=r(D,S)-r(D_S,S)=r(D,S)-c_+(D_S)$
and by writing $w(D_S)$
for $c_+(D_S)-c_-(D_S)$.\footnote{Note
that $w(D_S)$ was defined with opposite
sign in \cite{we:trees}.}

\subsection[The relation with spanning trees]
{The relation with spanning trees.}\label{sreltrees}
Assume that the regions of $D$ are colored black and
white in a checkerboard fashion, such that any two neighbored
regions have opposite colors, and such that the
unbounded region is colored white.
The {\it Tait graph} $\G_D$ is the planar graph whose vertices
are the black regions and whose edges correspond to the crossings
of $D$ (see Figure~\ref{figtaitgraph}).

\begin{figure}[H]
\centerline{\psfig{figure=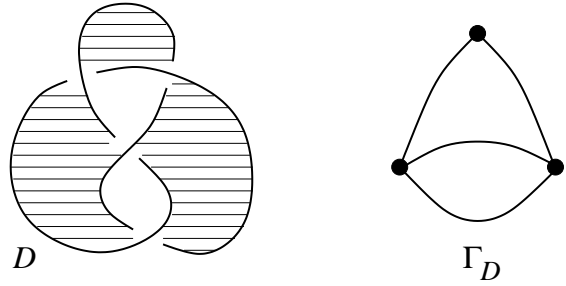,height=3.7cm}}

\caption{The Tait graph.}
\label{figtaitgraph}
\end{figure}
\noindent
Given a smoothing of a crossing of $D$, we call it a
{\it black} or a {\it white smoothing}
depending on whether it connects black or white regions of $D$.

Let $\T(\G_D)$ denote the set of all spanning trees of $\G_D$.
There is a bijection
$$
\begin{array}{ccc}
\T(\G_D) &\longrightarrow &\K_1(D) \\
T       &\longmapsto &S_T
\end{array}
$$
defined as follows:
to a tree $T$ we associate
the connected Kauffman state $S_T$ obtained
by choosing the black smoothing for
precisely those crossings which correspond to an edge
of $T$, and the white smoothing for all
other crossings. Using the above bijection,
we can can rewrite formula \eqref{fKauffmanefficient1} as
$$
\la D\ra=\sum_{T\in \T(\G_D)}\la D|D_T\ra\la D_T\ra
$$
where we have abbreviated $D_T$ for $D_{S_T}$.

The correspondence between spanning trees
and elements of $\K_1(D)$ leads to an easy proof of
the following lemma.
\begin{lemma}\label{lblack}
The number of black smoothings is the same in all $S\in\K_1(D)$.
\end{lemma}
\begin{proof}
Since black smoothings in $S_T$
correspond to edges of $T$,
it suffices to show that all spanning
trees of $\G_D$ have the same number of edges.
But this is obvious, because the number of edges
in any spanning tree is just one less
than the number of vertices of $\G_D$.
\end{proof}
An alternative proof of Lemma~\ref{lblack} uses
Kauffman's Clock Theorem \cite{ka:formal}.
By the Clock Theorem,
any two elements
of $\K_1(D)$ are related by a finite sequence of
state transpositions (see Figure~\ref{figclock}).
The lemma follows because
state transpositions
do not change the number of black smoothings.

Of course, the lemma also implies that the
number of white smoothings is the
same in all $S\in\K_1(D)$.

\begin{figure}
\centerline{
\begin{tabular}{c@{\qquad}c@{\qquad}c}
\psfig{figure=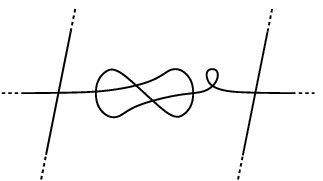,width=3.7cm}&
\psfig{figure=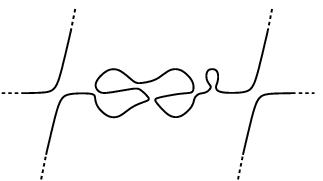,width=3.7cm}&
\psfig{figure=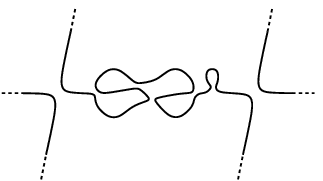,width=3.7cm}
\end{tabular}
}

\caption{A knot projection, and two smoothings related by
a state transposition.}\label{figclock}
\end{figure}

\section{Spanning tree model for the Khovanov bracket}\label{s3}
\noindent
In this section, we discuss how the construction of
Section~\ref{streeKauffman} transfers to the formal Khovanov bracket.
The main result is stated in the following theorem.

\begin{theorem}\label{t2.2}
Let $D$ be a connected link diagram. Then
the formal Khovanov bracket $\ls D\rs$ destabilizes
to a subcomplex $ST(D)$. On the level of objects
(i.e. if one ignores the differential),
$ST(D)$ is isomorphic to
\begin{equation}\label{fspanningtree}
ST(D)\cong\bigoplus_{S\in \K_1(D)}
U[r(D,S)-w(D_S)]\{r(D,S)-2w(D_S)\}
\end{equation}
where $U:=\ls\bigcirc\rs$ denotes the formal
Khovanov bracket of the trivial diagram consisting
of a single circle.
We call $ST(D)$ the
{\normalfont spanning tree subcomplex}
of $\ls D\rs$.
\end{theorem}
Theorem~\ref{t2.2} can be viewed as
a ``categorification'' of formula~\eqref{fKauffmanexplicit1}.
Indeed, since
$U\cong\emptyset\{1\}\oplus\emptyset\{-1\}$
by Lemma~\ref{ldelooping},
the shifts of the gradings in \eqref{fspanningtree}
agree with the powers of $-1$ and $q$ in \eqref{fKauffmanexplicit1}.
Before proving the theorem, we
mention two corollaries.

\begin{corollary}\label{cspanningtree}
Let $D$ be a connected link diagram. Then
$\Co(D)$ destabilizes to the subcomplex
$\F_{\Kh}(ST(D))\subset\Co (D)$. As a bigraded
module, $\F_{\Kh}(ST(D))$ is isomorphic to
$$
\F_{\Kh}(ST(D))\cong
\bigoplus_{S\in \K_1(D)}
A_{\Kh}[r(D,S)-w(D_S)]\{r(D,S)-2w(D_S)\}\; .
$$
$\F_{\Kh}(ST(D))$ will be called the
{\normalfont spanning tree subcomplex}
of $\Co(D)$.
\end{corollary}

Using that
$A_{\Kh}=\Z {\bf 1}\oplus\Z X$
and $\deg({\bf 1})=+1$ and $\deg(X)=-1$,
we get
the following estimate
for the ranks of the Khovanov homology groups:
\begin{corollary}\label{ctreeestimate}
Let $D$ be a connected link diagram. Then
$$
\dim_{\Q}(\HCo(D)\otimes\mathbb Q)\leq 2(\#\K_1(D))\; .
$$
Moreover, the rank of
$\HCo^{i,j}(D)$ is bounded from above by the number of
$S\in\K_1(D)$ with $r(D,S)-w(D_S)=i$ and $r(D,S)=2i-j\pm 1$.
\end{corollary}
Corallary~\ref{ctreeestimate}
shows that the ranks of the homology
groups $\HCo^{i,j}(D)$
tend to be
much smaller
than the ranks of the chain groups $\Co^{i,j}(D)$.
This is consistent with
Bar--Natan's experimental observation \cite{ba:first}.
\vsp\vsp

\begin{proofth}{t2.2}
To prove the theorem, we reformulate
the arguments which led us to formula
\eqref{fKauffmanexplicit1} in Section~\ref{streeKauffman}
in the setting of the formal Khovanov bracket.

First, we
consider a diagram $D'\in\T(D)$ sitting
at a leaf of the binary tree of Figure~\ref{figimproved}.
Since $D'$ is R1--trivial,
part 1 of Lemma~\ref{lkhovanovbracket}
(Subsection~\ref{sdefinitionKhovanovbracket})
implies that
$\ls D'\rs$ destabilizes to a subcomplex
isomorphic to
$
U[c_-(D')]\{2c_-(D')-c_+(D')\}
$.
Comparing
this with \eqref{fKauffmantrivial},
we see that the theorem is true for the diagrams
sitting at the leaves of the tree.

Now we proceed inductively, going
up the tree.
Let $D_1$ be a diagram
sitting at an internal node of the tree,
and let $D_2$ and $D_3$ be the two diagrams
sitting right below that node.
By induction,
the complexes $\ls D_2\rs$ and $\ls D_3\rs$
destabilize to subcomplexes
$ST(D_2)$ and $ST(D_3)$.
Moreover,
$\ls D_1\rs$ is isomorphic to the mapping
cone of a chain transformation between
$\ls D_2\rs$ and $\ls D_3\rs\{1\}$
(see \eqref{fsaddlecone}).
By Lemma~\ref{lcommutecone}
(Subsection~\ref{scomplexes}),
forming the mapping cone ``commutes''
with destabilization.
Therefore, $\ls D_1\rs$ destabilizes to
a subcomplex $ST(D_1)$ which is isomorphic
to the mapping cone of a chain transformation
between $ST(D_2)$ and $ST(D_3)\{1\}$.
In particular, on the level of objects
we have
$ST(D_1)\cong ST(D_2)\oplus ST(D_3)\{1\}[1]$.
Using this as a substitute for the relation
$\la D_1\ra=\la D_2\ra-q\la D_3\ra$,
and arguing as in Section~\ref{streeKauffman}, we get the theorem.
\end{proofth}

\begin{remark}
Let $D$ be a link diagram.
After selecting a point
$P$ on an edge of $D$,
we can endow $\Co(D)$ with the
structure of an
$A_{\Kh}$--module, as follows:
multiplication by ${\bf 1}\in A_{\Kh}$
is the identity map;
multiplication by $X\in A_{\Kh}$
is induced by ``multiplying'' with a dot
at point $P$.
The {\it reduced Khovanov complexes}
are the complexes
$\Co(D)\otimes_{A_{\Kh}} \Z X$ and
$\Co(D)\otimes_{A_{\Kh}}(A_{\Kh}/\Z X)$,
where $\Z X\subset A_{\Kh}$ denotes
the $A_{\Kh}$--submodule of $A_{\Kh}$
generated by $X\in A_{\Kh}$.
If one performs the R1 moves in the
proof of Theorem~\ref{t2.2}
far away from $P$,
the isomorphism in
Corollary~\ref{cspanningtree}
becomes an isomorphism of
$A_{\Kh}$--modules.
By tensoring with $\Z X$ and $A_{\Kh}/\Z X$,
one gets spanning tree models
for the reduced Khovanov complexes.
\end{remark}

\begin{remark}
Spanning trees of the Tait graph
also appear as generators of
the knot Floer complex \cite{os}.
Hence the spanning tree model might shed
some light on the relation between
Khovanov homology and knot Floer homology.
\end{remark}

\section{Hopf link addition}\label{shopflinkaddition}
\noindent
In this section, we apply the spanning tree model
to prove a theorem, which was originally proved
(for Khovanov homology) by M.~Asaeda and J.~Przytycki \cite{ap}.
As mentioned in \cite{we:mutation}, the theorem
also follows from \cite[Corollary~10]{kh:first}.

\begin{theorem}\label{thopflinkaddition}
Assume the link diagram $D\#H$ is obtained from
a link diagram $D$ by Hopf link
addition (see Figure~\ref{fighopfadd}).
Then the complex $\ls D\#H\rs$
destabilizes to the direct sum
$\ls D \rs[0]\{-1\}\oplus\ls D\rs[2]\{3\}$.
\end{theorem}

\begin{figure}[H]
\centerline{\psfig{figure=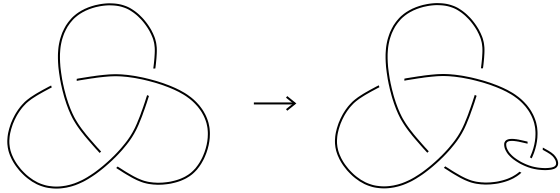,height=2.5cm}}

\caption{Hopf link addition.}\label{fighopfadd}
\end{figure}

To prove the theorem, we observe
that the spanning tree model extends
to $(1,1)$--tangles, i.e. tangles having exactly
two boundary points, as in Figure~\ref{fighprime}. The only
difference is that
the Tait graph of a $(1,1)$--tangle has
a distinguished vertex (the vertex which
corresponds to the black region
adjacent to the dotted circle),
and hence the spanning trees are rooted.

\begin{figure}[H]
\centerline{\psfig{figure=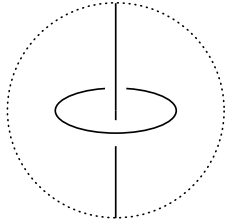,height=2.9cm}}

\caption{A $(1,1)$--tangle.}\label{fighprime}
\end{figure}

Let $H'$ denote the $(1,1)$--tangle shown in
Figure~\ref{fighprime}.
Inserting $H'$ into an edge of $D$ has the
same effect as summing a Hopf link to that
edge of $D$.
The Tait graph of $H'$ has exactly two
spanning trees.

\begin{figure}[H]
\centerline{\psfig{figure=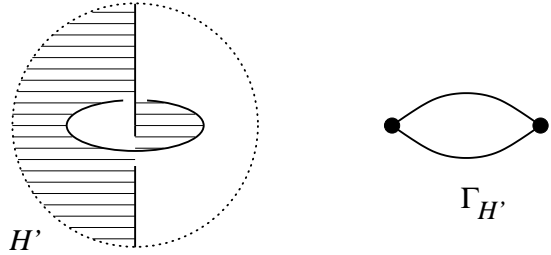,height=3.3cm}}

\caption{Tait graph of $H'$.}
\end{figure}
\noindent
Now the spanning tree model
tells us that $\ls H'\rs$ destabilizes to
a subcomplex $ST(H')$, which is isomorphic on
the level of objects to
$\ls\; |\;\rs[0]\{-1\}\oplus\ls\; |\;\rs[2]\{3\}$. Here
``$\, |\,$'' denotes the trivial $(1,1)$--tangle, consisting
of a single vertical line.
Note that the homological gradings of the two
summands
in $\ls\; |\;\rs[0]\{-1\}\oplus\ls\; |\;\rs[2]\{3\}$
differ by two. Therefore, $ST(H')$ must have trivial differential,
and so the isomorphism
$ST(H')\cong\ls\; |\;\rs[0]\{-1\}\oplus\ls\; |\;\rs[2]\{3\}$
is actually an isomorphism of complexes.
Using the good composition properties of the Khovanov
bracket with respect to gluing of tangles, we get
the theorem.

\section{Alternating knots}\label{s4}
\noindent
The theorems in this section were conjectured by D.~Bar--Natan,
S.~Garoufalidis
and M.~Khovanov \cite{ba:first}
and proved by E.~S.~Lee \cite{le1}. We give new proofs
using the spanning tree model.
For short proofs, see also \cite{ap}.

A knot diagram is said to be {\it alternating}
if one alternately over-- and undercrosses
other strands as one goes along the knot in that diagram.
A knot is called {\it alternating}
if it possesses an alternating diagram.

\begin{lemma}\label{lone}
Let $D$ be an alternating knot diagram.
Then the number of $1$--smoothings
in $S$ is the same for all $S\in\K_1(D)$.
\end{lemma}
\begin{proof}
Since $D$ is alternating, we necessarily have
one of the following two situations:
either the $1$--smoothings coincide with the black smoothings,
or the $1$--smoothings coincide with the white smoothings.
Hence Lemma~\ref{lone} follows from Lemma~\ref{lblack}
of Subsection~\ref{sreltrees}.
\end{proof}

Given an an alternating knot diagram $D$,
we denote by $n_1(D):=r(D,S)$ the number of $1$--smoothings
in any $S\in\K_1(D)$. Corollary~\ref{ctreeestimate}
implies:

\begin{theorem}\label{talternating}
Let $D$ be an alternating knot diagram.
$\HCo^{i,j}(D)$ is zero
unless the pair $(i,j)\in\Z^2$ lies on one of the
two lines $j=2i-n_1(D)\pm 1$.
\end{theorem}

Let $i_-$ and $i_+$ denote
the smallest and largest integer $i$ for
which there is an $S\in\K_1(D)$ such
that $r(D,S)-w(D_S)=i$.
Let $j_-:=2i_--n_1(D)-1$ and $j_+:=2i_+-n_1(D)+1$.

Since the spanning tree subcomplex $\F_{\Kh}(ST(D))$
of an alternating knot diagram $D$ is concentrated
on the two lines $j=2i-n_1(D)\pm 1$, and
since the differential has bidegree $(1,0)$,
we get the following theorem.
\begin{theorem}\label{talternatingtorsion}
Let $D$ be an alternating knot diagram. Then
\begin{enumerate}
\item
$\HCo^{i,j}(D)$ is zero unless $i_-\leq i\leq i_+$.
\item
$\HCo^{i,j}(D)$ is torsion free unless $j=2i-n_1(D)-1$.
\item
$\HCo^{i_-,j_-}(D)$ and $\HCo^{i_+,j_+}(D)$ are
non--zero and torsion free.
\end{enumerate}
\end{theorem}

Recall that a crossing of $D$ is called {\it splitting}
if it connects two otherwise disconnected parts of $D$.

\begin{theorem}\label{tminmax}
Let $D$ be an alternating knot diagram with $c$ crossings.
Assume that no crossing of $D$ is splitting.
Then $\HCo^{i_-,j_-}(D)=\HCo^{i_+,j_+}(D)=\Z$.
Moreover, $i_-=0$ and $i_+=c$.
\end{theorem}

\begin{proof}
By part 3 of the previous theorem, we know
that $\HCo^{i_-,j_-}(D)$ and $\HCo^{i_+,j_+}(D)$
are free abelian groups of rank at least one.
To show that the rank is exactly one, it
suffices to show that there is only
one $S\in\K_1(D)$ contributing to
the lowest degree $i_-$, i.e. such that
$r(D,S)-w(D_S)=i_-$, and likewise only
one $S\in\K_1(D)$ contributing
to the highest degree $i_+$, i.e. such that
$r(D,S)-w(D_S)=i_+$.

Actually, we prove something slightly different.
Recall that the spanning tree construction depends
on a numbering of the crossings of $D$.
In particular, the diagram $D_S$ associated to
$S\in\K_1(D)$
depends on the numbering of the crossings.
What we show is that for any $S\in\K_1(D)$, there exists
a numbering such that $S$ is the unique state
contributing to lowest/highest degree.
This is the content of the following lemma.
\end{proof}

\begin{lemma}\label{lnumbering}
Let $D$ be an alternating knot diagram with $c$ crossings,
all of which are non--splitting,
and let $S$ be an element of $\K_1(D)$.
\begin{enumerate}
\item
There is a numbering of the crossings of
$D$ such that $-w(D_S)=-n_1(D)$, and $-w(D_{S'})>-n_1(D)$
for all $S'\in\K_1(D)$ with $S'\neq S$.
\item
Likewise, there is a numbering of the crossings
of $D$ such that $-w(D_S)=c-n_1(D)$, and
$-w(D_{S'})<c-n_1(D)$ for all $S'\in\K_1(D)$ with $S'\neq S$.
\end{enumerate}
\end{lemma}
\begin{proof}
1. Let $S$ be an element of $\K_1(D)$.
Assume that the crossings of $D$ are numbered in such a
way that the crossings which are $0$--smoothings in $S$
precede those which are $1$--smoothings in $S$.
We claim that for this numbering,
the relations in part~1
of Lemma~\ref{lnumbering} are satisfied, i.e.
$-w(D_S)=-n_1(D)$, and $-w(D_{S'})>-n_1(D)$
for all $S'\in\K_1(D)$ with $S'\neq S$.

To see this, we consider the link diagrams
$D_k$, $0\leq k\leq c-n_1(D)$, obtained
from $D$ by replacing the first $k$
crossings of $D$ by their $0$--smoothings,
while leaving the remaining $c-k$
crossings unchanged.
We denote by $D'$ the diagram $D':=D_{c-n_1(D)}$.
Note that if one replaces
all crossings in $D'$ by their
$1$--smoothings, the result is the state $S$.
Since $S$ is connected, so is $D'$, and so
are all $D_k$ with $k\leq c-n_1(D)$.

Because $D$ is alternating, we may
assume without loss of generality that the
$1$--smoothings in $S$ are the black smoothings,
and hence correspond to the edges of the
spanning tree associated to $S$. Using that
every edge in a tree connects two otherwise
disconnected parts, we get that every
crossing of $D'$ is splitting, i.e.
connects to otherwise disconnected parts
of $D'$.

\begin{claim}
$D_S=D'$.
\end{claim}
\begin{proofc}
Recall the binary tree of Figure~\ref{figimproved},
which was
used to deduce the spanning tree expansion.
If at all $D'$ appears in this tree,
then the afore mentioned properties
of $D'$ imply that it must
be the leaf $D_S$.

Thus, it suffices
show that the sequence
$D=D_0,D_1,\ldots,D_{c-n_1(D)}=D'$
appears along a path going down the
binary tree.
Since $D_{k+1}$ results from
$D_k$ by resolving the first
crossing of $D_k$,
we only have to check that for all
$k< c-n_1(D)$ the first crossing of
$D_k$ is non--splitting.
This can be done by observing
that the last $n_1(D)$ crossings
of $D_k$ form the edges of a spanning tree
(same argument as used above for $D'$),
and using that the Tait graph
of $D$ is loop--less because
all crossings of $D$ are non--splitting.
We leave the details to the reader.
\end{proofc}

So we have that $D_S=D'$, and we
also know that the (unsmoothened)
crossings in $D'$
are the $1$--smoothings in $S$.
Since $D_S$ is connected and since all
of its crossings are splitting, this implies that
all crossings of $D_S$ must be positive
with respect to an arbitrary
orientation of $D_S$.
We conclude $w(D_S)=c_+(D_S)=n_1(D)$.

Now consider $S'\in\K_1(D)$ with $S'\neq S$.
Recall that $S$ and $S'$ both have exactly
$n_1(D)$ $1$--smoothings.
In $S$ the $1$--smoothings come
after the $0$--smoothings.
Therefore, the first crossing of $D$
where $S$ and $S'$ differ has to be a $0$--smoothing
in $S$ and a $1$--smoothing in $S'$.
Being a $0$--smoothing in $S$,
this crossing is smoothened in $D_S=D'$.
We leave it to the reader to conclude
that it also has to be smoothened in $D_{S'}$.
Thus we have found a $1$--smoothing in $S'$
which is smoothened in $D_{S'}$.
This implies $c_+(D_{S'})<n_1(D)$
(cf. previous paragraph)
and hence
$-w(D_{S'})\geq -c_+(D_{S'})>-n_1(D)$.

2. The second part of the lemma is proved analogously,
by numbering the crossings of $D$
in such a way that the crossings which are $1$--smoothings
in $S$ precede those which are $0$--smoothings in $S$.
\end{proof}

The above proof was inspired by \cite{th}.
For a different proof of a similar statement,
see \cite[Section~7.7]{kh:first}.

\begin{corollary}
If a knot possesses an alternating diagram with $c$ crossings, all of
which are non--splitting, then the knot does not admit a diagram with fewer than
$c$ crossings.
\end{corollary}
\begin{proof}
By part 1 of Theorem~\ref{talternatingtorsion},
$i_+(D)$ and $i_-(D)$ are equal to the highest and the lowest
homological degree in which $\HCo(D)$ is non--zero.
Therefore,
the difference $i_+(D)-i_-(D)$ is a lower bound
for the number of crossings of $D$.
Moreover, $i_+(D)-i_-(D)$ is a knot invariant.
Now assume that $D$ is an alternating diagram with $c$
crossings, all of which are non--splitting.
By Theorem~\ref{tminmax}, we have
$i_+(D)-i_-(D)=c$, and hence
the corollary follows.
\end{proof}

\begin{remark}
For alternating knots,
the spanning tree model allows
to calculate the reduced Khovanov homology
completely.
Indeed, for an alternating knot the reduced spanning tree
subcomplex is supported on a single line $j=2i+const$
in the $ij$--plane.
Since the differential has bidegree $(1,0)$, it must
vanish.
Therefore,
the reduced spanning tree subcomplex
is isomorphic to the reduced Khovanov
homology of the alternating knot.
\end{remark}

\begin{remark}
We can
also
consider the subcomplex
$\F_{\Lee}(ST(D))\subset\C'(D)$.
While we know
explicit generators for Lee homology from
Section~\ref{scanonical},
the spanning tree description
of Lee's complex has the advantage
that it also makes a statement about the filtration,
and that it works well for Lee homology over
$\Z$ coefficients.
Theorem~\ref{talternating} and part 1
of Theorem~\ref{talternatingtorsion} remain
valid for Lee homology.
\end{remark}
\begin{example}
Let $D$ be a standard diagram of the left
handed trefoil. Let $\F^{\Z}_{\Lee}$
denote Lee's functor with $\Z$ coefficients,
and let $A^{\Z}_{\Lee}:=\F^{\Z}_{\Lee}(\bigcirc)$
(so $A^{\Z}_{\Lee}=A_{\Kh}$,
except that $A^{\Z}_{\Lee}$ is filtered whereas
$A_{\Kh}$ is graded).
We have
$$\F^{\Z}_{\Lee}(ST(D))\cong
A^{\Z}_{\Lee}[0]\{-1\}\oplus
A^{\Z}_{\Lee}[2]\{3\}\oplus A^{\Z}_{\Lee}[3]\{5\}
$$
The differential is zero on $A^{\Z}_{\Lee}[0]\{-1\}$,
and it maps ${\bf 1},X\in A^{\Z}_{\Lee}[2]\{3\}$ to
$2X,2\cdot {\bf 1}\in A^{\Z}_{\Lee}[3]\{5\}$.
Hence
$$
\HC'(D;\Z)\cong
A^{\Z}_{\Lee}[0]\{-1\}\oplus
(A^{\Z}_{\Lee}/2 A^{\Z}_{\Lee})[3]\{5\}
$$
Note that there is $2$--torsion in
bidegree $(3,6)$, despite the fact that
the pair $(3,6)$ lies on the upper of the two
lines mentioned in Theorem~\ref{talternating},
and despite the fact that $(3,6)=(i_+,j_+)$.
Hence parts 2 and 3 of Theorem~\ref{talternatingtorsion}
do not transfer to Lee homology with $\Z$ coefficients.
\end{example}

%%%%%%%%%%%%%%%%%%%%%% Framed link cobordisms %%%%%%%%%%%%%%%%%%%%%%%%
\setcounter{footnotebuffer}{\value{footnote}}
\chapter{Framed link cobordisms}\label{cframedlinkcobordisms}
\setcounter{footnote}{\value{footnotebuffer}}
In this chapter, we introduce movie presentations
and movie moves for framed link cobordisms.

\section{Framed links}\label{sframedlinks}
\noindent
Let $L\subset\R^3$ be a link.
A {\it framing} of $L$ is a homotopy class of
trivializations
of the normal bundle of $L$ in $\R^3$. Equivalently,
a framing can be defined as homotopy
class of non--singular normal vector fields on $L$.
A link equipped with a framing is called
a {\it framed link}.

Let $K\subset\R^3$ be a knot and let
$f$ be a framing of $K$.
Represent $f$ by a non--singular
normal vector field, and
assume that
the vectors are sufficiently short,
so that
their tips trace out a knot $K'$ parallel
to $K$. The {\it framing coefficient} of $f$
is the linking number
$n(f):=\lk(K,K')$
of $K$ and $K'$.
One can show
that $f$ is completely determined
by its framing coefficient $n(f)$.

If $L$ is a link, a framing of $L$ can
be specified
by specifying a framing $f_i$
for each component $L_i$ of $L$.
The {\it total framing coefficient} of $L$
is defined by
$$
n(f):=\sum_i n(f_i)+2\sum_{i<j}\lk(L_i,L_j)\; .
$$

There are several methods for describing framed links.
One possibility is to take an ordinary link
diagram $D$ and then think of it as presenting a framed link,
framed by the {\it blackboard framing},
i.e. by the framing which is given by a vector field which is
everywhere parallel to the plane of the picture.

\begin{figure}[H]
\centerline{\psfig{figure=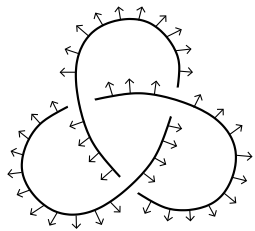,height=3cm}}

\caption{The blackboard framing.}
\end{figure}
It is easy to see that
the framing coefficient
of the blackboard framing
is equal to the writhe of $D$.
Since the writhe
changes by $\pm 1$ under move R1,
the blackboard framing is not
invariant under this move.
However, it is invariant under the
move FR1 shown in Figure~\ref{figframedR1}.
In fact, if one uses the blackboard framing
to present framed links,
then two link diagrams represent isotopic framed
links if and only if they are related by a finite
sequence of the moves FR1, R2 and R3.
\begin{figure}[H]
\centerline{\psfig{figure=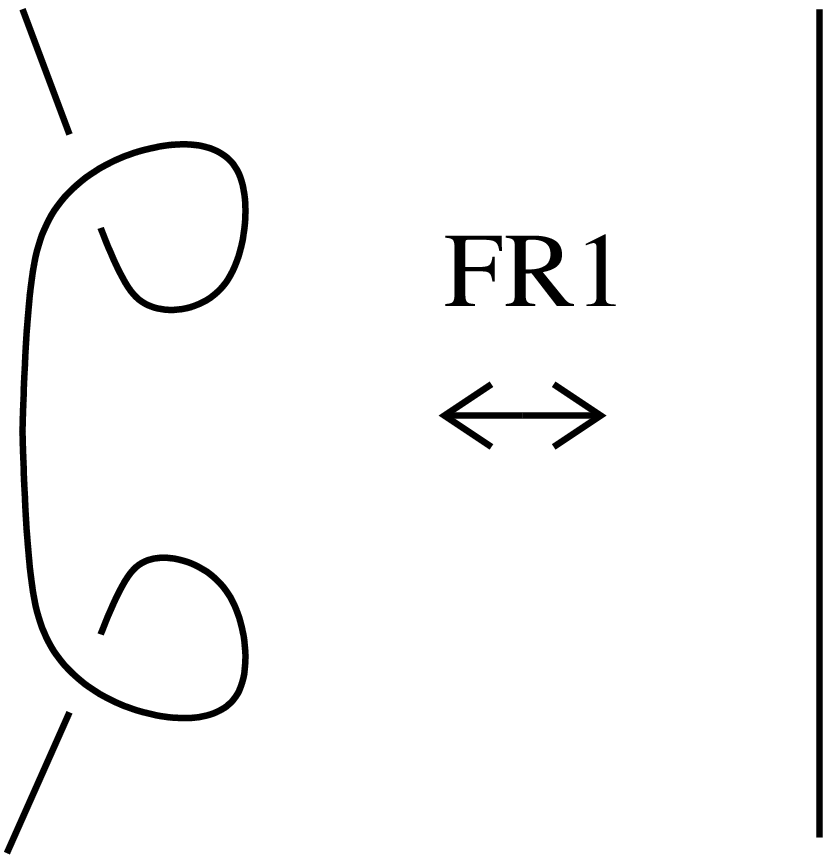,height=2.8cm}}

\caption{The framed Reidemeister move FR1.}
\label{figframedR1}
\end{figure}
\noindent

Another way of presenting framed links
uses
{\it link diagrams with signed points}\footnote{
Link diagrams with signed points were
introduced in \cite{bw}, where they were
called ``link diagrams with marked points''.}.
A link diagram with signed points is
a link diagram $D$, together with a finite collection
of distinct points, lying on the interiors of the edges
of $D$, and labelled by $+$ or $-$.
Such a diagram presents a framed link,
with framing $f_D$ given as follows:
$f_D$ is represented
by a vector field which is everywhere
parallel to the drawing plane, except
in a small neighborhood of the signed points, where
it winds around the link, in
such a way that each positive point contributes $+1$
to $n(f_D)$ and each negative point contributes
$-1$.
Note that $n(f_D)=w(D)+t(D)$, where
$t(D)$ denotes the difference between the
numbers of positive and negative signed points
in $D$.

\begin{figure}[H]
\centerline{\psfig{figure=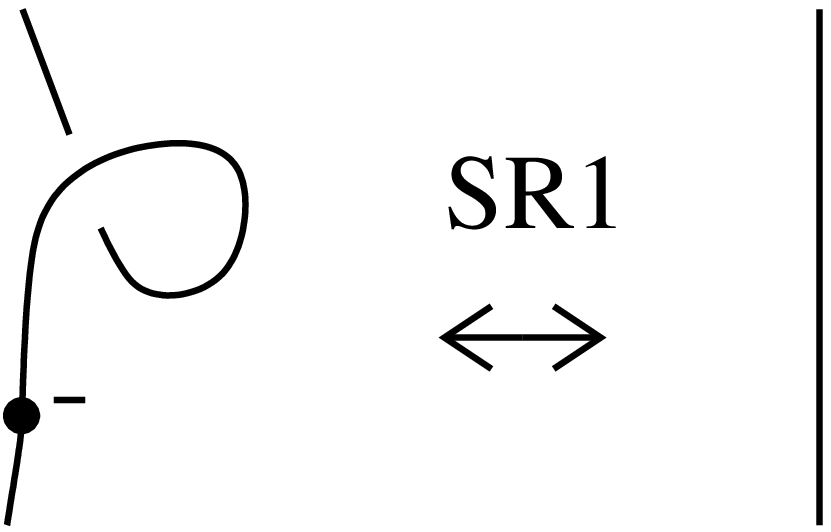,height=2.2cm}}

\caption{The signed Reidemeister move SR1.}
\label{figsignedmove}
\end{figure}
\noindent
The signed first Reidemeister move SR1, shown above,
leaves $n(f_D)$ unchanged. It follows that two link diagrams
with signed points describe isotopic framed links if
and only if they are related by a finite sequence of the
following moves: the moves
SR1, R2 and R3, as well as creation/annihilation of
pairs of nearby oppositely signed points,
and sliding signed points past crossings.

The $m$--{\it cable} of a framed knot $K$
is the
$m$--component link $K^m$, obtained
by replacing $K$ by $m$ parallel strands,
pushed off in the direction of the framing vector field.

\begin{figure}[H]
\centerline{\psfig{figure=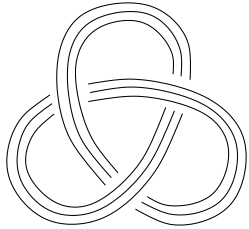,height=3cm}}

\caption{$3$--cable of a framed knot (framed by the blackboard framing).}
\end{figure}

\section{Framings on submanifolds of codimension $2$}
\label{scodimensiontwo}
\noindent
The concept of framings is not restricted to links.
In this section, we study framings on
arbitrary submanifolds of codimension $2$.

Let $M$ be a smooth oriented $(n+2)$--manifold and let $N\subset M$ be
a smooth oriented compact submanifold of $M$ of dimension $n$.
A {\it framing} of $N$ is a homotopy class of trivialization of the
normal bundle $\nu_N$ of $N\subset M$.
If $N$ has non--empty boundary and
a trivialization $t$ of $\nu_N|_{\partial N}$
is specified,
we define a
{\it relative framing} of $N$ (relative to $t$)
as a homotopy class of trivializations of $\nu_N$ which agree
with $t$ over $\partial N$.
\begin{lemma}\label{affine}
Let $t$ be a trivialization of $\nu_N|_{\partial N}$.
If non--empty, the set of relative framings of $N$
(relative to $t$) is an affine space over $H_{n-1}(N)$.
\end{lemma}
\begin{proof} Since $\nu_N$ is an oriented $2$--plane bundle,
its structural group is $SO(2)$. Therefore
the difference between two relative framings is given
by a homotopy class of maps from the pair $(N,\partial N)$
to the pair $(SO(2),1)$, i.e. by an element of $[N,\partial N;SO(2),1]$.
Using that $SO(2)$ is a $K(\Z,1)$ space, we can identify
$[N,\partial N;SO(2),1]$ with $H^1(N,\partial N)$.
And by Poincar\'e
duality, $H^1(N,\partial N)$ is isomorphic to $H_{n-1}(N)$.
\end{proof}

Let us
consider pairs $(E,t)$ where $E$ is an oriented $2$--plane bundle
over $N$, and $t$ is a trivialization of $E|_{\partial N}$.
We call two such pairs $(E,t)$ and $(E',t')$ {\it isomorphic} if there
is an isomorphism $F:E\rightarrow E'$ of oriented $2$--plane bundles
such that $t'\circ F=t$ over $\partial N$.
\begin{lemma}
Isomorphism classes of pairs $(E,t)$ correspond
bijectively to elements of $H_{n-2}(N)$.
\end{lemma}
\begin{proof}
Isomorphism classes of pairs $(E,t)$ are classified by
homotopy classes of maps from the pair $(N,\partial N)$ to
the pair $(BSO(2),p_0)$, where $p_0\in BSO(2)$ is an arbitrary basepoint.
Since $BSO(2)$ is a $K(\Z,2)$ space,
we obtain $[N,\partial N;BSO(2),p_0]=H^2(N,\partial N)=H_{n-2}(N)$.
\end{proof}

Let $e(E,t)\in H_{n-2}(N)$ denote the homology class
corresponding to the pair $(E,t)$. We immediately obtain:
\begin{lemma}
$N$ admits a framing (relative to $t$) if and
only if $e(\nu_N,t)=0$.
\end{lemma}

We are mainly interested in
the case where $N$ is a connected surface $S$,
embedded in a $4$--manifold $M$.
In this case, $e(\nu_S,t)$ is an integer
$e(\nu_S,t)\in H_0(S)=\mathbb{Z}$ which can
be described as follows:
identify $S$ with the zero section of $\nu_S$,
and consider a section $S'$ of $\nu_S$,
whose restriction to the boundary $\partial S$
is non--vanishing and constant with
respect to the trivialization $t$.
Then $e(\nu_S,t)=S\cdot S'$ where $S\cdot S'$ denotes
the algebraic intersection number of the surfaces $S$ and $S'$
in the total space of $\nu_S$.
Since $S$ has a tubular neighborhood in $M$
which is diffeomorphic to the total space of $\nu_S$,
we can view $e(\nu_S,t)$ as a
relative self--intersection number of $S$ in $M$.

Now assume $e(\nu_S,t)=0$. Then the set of relative
framings on $S$ is non--empty and hence
an affine space over $H_1(S)$ (by Lemma~\ref{affine}).
The action of $H_1(S)$ on framings
can be described as follows.
Let $c$ be an oriented simple closed curve on
$S$ representing an element of $H_1(S)$.
Consider a tubular neighborhood $U\subset S$ of
$c$, diffeomorphic to $c\times [0,2\pi]$.
Let $\chi_c$ be the map from $S$ to $SO(2)$
which is trivial on the complement of $U$ and maps
a point $(\theta,\varphi)\in U=c\times [0,2\pi]$ to
rotation by $\varphi$. Now
$[c]$ acts on framings by sending the
framing given by a vector field $v(z)$ to
the framing given by the vector field $\chi_c(z)v(z)$.
Note that the
Poincar\'e dual $\operatorname{PD}^{-1}[c]\in H^1(S,\partial S)$
has the following interpretation: let $c'$ be a properly
embedded simple curve on $S$ representing an element
of $H_1(S,\partial S)$. The restriction $\chi_c|_{c'}$
is a closed curve in $SO(2)$,
which winds around $SO(2)$
once at every intersection point of $c'$ with $c$.
Hence the class of $\chi_c|_{c'}$
in $\pi_1(SO(2),1)=\mathbb{Z}$
is given by $[\chi_c|_{c'}]=c\cdot c'=
\langle\operatorname{PD}^{-1}[c],[c']\rangle$.

\section{Framed link cobordisms}
\noindent
Now let $S\subset\R^3\times [0,1]$ be a connected
link cobordism between two oriented framed links $L_0$ and $L_1$.
\begin{lemma}
$S$ admits a relative framing,
relative to the given framings of $L_0$ and $L_1$,
if and only if the
total framing coefficients of $L_0$ and $L_1$ agree.
\end{lemma}
\begin{proof}
For the sake of simplicity, we restrict
to the case where $S$ is a cobordism between
knots $K_0$ and $K_1$. Let
$K_0'$ and $K_1'$ be parallels of $K_0$ and $K_1$,
which are pushed off in the direction of the framing.
Choose a cobordism
$S_0\subset\R^3\times (-\infty,0]$
from the empty link to $K_0\subset\R^3\times\{0\}$
and a cobordism $S_1\subset\R^3\times[1,\infty)$ from
$K_1\subset\R^3\times\{1\}$ to the empty link.
Consider small perturbations $S',S_0',S_1'$
of $S,S_0,S_1$, whose boundaries are $K_0'$ and $K_1'$.
Then $F:=S_0\cup S\cup S_1$ and
$F':=S_0'\cup S'\cup S_1'$
are closed oriented surfaces in $\R^4$.
Since $F$ is null--homologous in $\R^4$,
we obtain
$$
0=F\cdot F'=S_0\cdot S_0'+S\cdot
S'+S_1\cdot S_1'
=n(f_0)+e(\nu_S,f_0\cup f_1)-n(f_1)
$$
where $f_0$ and $f_1$ denote the framings
of $K_0$ and $K_1$, respectively.
Hence we have $n(f_0)=n(f_1)$ if and only
if $e(\nu_S,f_0\cup f_1)=0$, if and only if
$S$ admits a relative framing.
\end{proof}

\section[Movies for framed link cobordisms]%
{Movie presentations for framed link cobordisms}
\noindent
A {\it framed movie}
is a sequence of oriented link diagrams, such that
any two consecutive diagrams differ either
by isotopy, a Morse move, a Reidemeister move R2 or R3,
or the framed Reidemeister move FR1.
We can use
such a sequence to describe
a framed link cobordism. Indeed,
it is clear that such as sequence
presents a link cobordism, and
a framing can be specified
by equipping every
diagram of the sequence with the vector
field which is everywhere perpendicular
to the plane of the picture.

{\it Signed movies} are defined
similarly to framed movies.
The only difference is that here the link
diagrams contain signed points and that two
consecutive diagrams may differ by SR1 instead
of FR1, and also by annihilation/creation
of signed points and by sliding signed
points past a crossing.
Like framed movies, signed movies can be
used to present framed link cobordisms.

\begin{theorem}[\cite{bw}]\label{signedmovie}
1.~Every framed link cobordism has a signed movie presentation.
2.~Two signed movies present isotopic framed link cobordisms if an only
if there is a sequence of signed movie moves SM1--SM20
which takes one movie to the other.
\end{theorem}
\noindent
The signed movie moves SM1--SM20
are shown in
Figures~\ref{figsigned15} and \ref{figsigned20}.

\vsp\vsp\par

\begin{figure}
\centerline{\psfig{figure=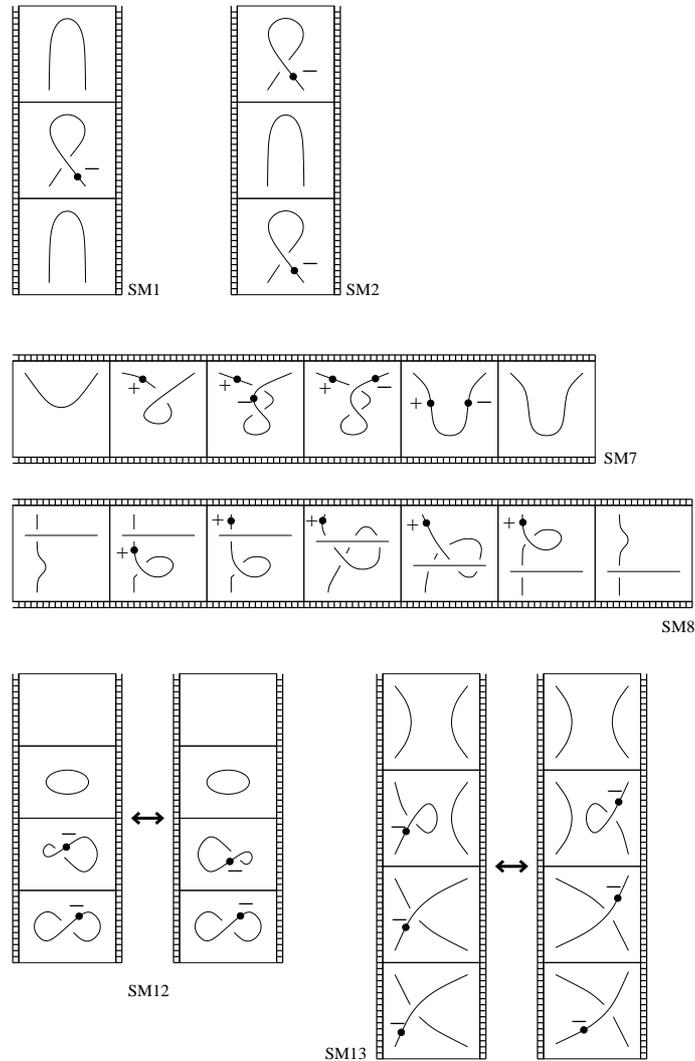,height=14cm}}
\vsp\vsp\par

\caption{Signed movie moves SM1--SM15. These moves
are obtained by inserting signed points into
the Carter--Saito moves MM1--MM15,
in such a way that
each R1 move becomes an SR1 move.
The moves SM3--SM6, SM9--SM11, SM14 and SM15
are not displayed
because they are identical with
the corresponding unsigned moves.
When lifting an R1 move to an SR1 move,
one has two possibilities where to insert
the signed point (one can place
the signed point
on either of the two sides of the curl).
Only one possibility is shown above.
}
\label{figsigned15}
\end{figure}
\begin{figure}[t]
\centerline{\psfig{figure=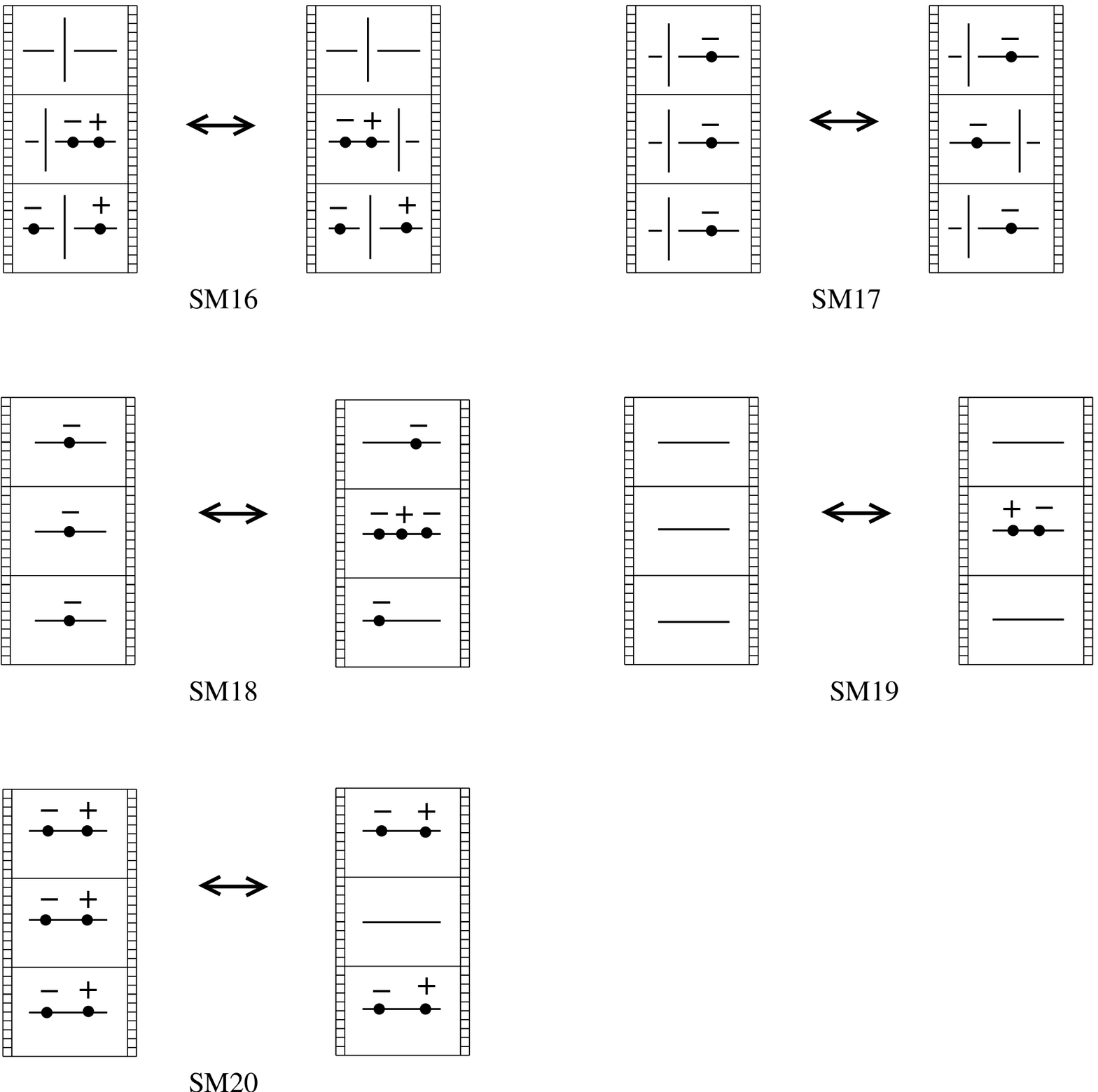,height=9.5cm}}
\vsp\par

\caption{Signed movie moves SM16--SM20.}
\label{figsigned20}
\end{figure}

\begin{proofth}{signedmovie}
1. Let $S$ be a link cobordism and let $f$ be a framing on $S$.
By Theorem~\ref{tmovie}, there is
an unsigned movie $M$
representing
the unframed cobordism $S$.
Inserting signed points into $M$,
in such a way that
every R1 move in $M$ becomes an SR1 move,
we obtain a signed movie $M'$,
which represents the cobordism $S$,
framed by some framing $f'$.
It remains to show
that $f'$ can be changed to
$f$ by inserting additional signed points into $M'$.

To see this, note that the signed points in
the movie $M'$ trace out curves on the cobordism $S$.
These curves can be oriented consistently,
by declaring that positive (negative) points ``move''
backwards (forwards) in time.
Conversely,
if $c$ is an oriented closed
curve on $S$,
we can think of $c$
as being traced out by signed points.
By inserting these points into the movie $M'$,
we can change the framing represented by $M'$.
\begin{figure}[H]
\centerline{\psfig{figure=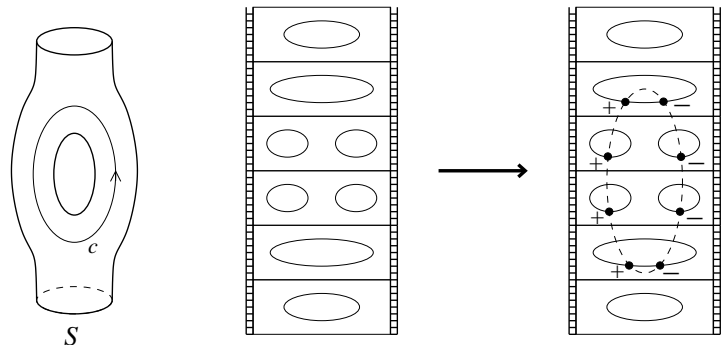,height=4.5cm}}
\caption{Inserting an oriented closed curve $c$ into a movie.}
\end{figure}
\noindent
Hence we obtain an action of oriented closed curves
on the set of framings of $S$. It is easy to see
that this action coincides with the $H_1(S)$--action
discussed in Section~\ref{scodimensiontwo}. Since the latter
action is transitive,
it follows that
we can find a configuration
of signed points whose insertion into $M'$ changes $f'$ into $f$.

\vsp

2. Let  $M'$ and $M''$ be two signed movies
representing isotopic framed cobordisms.
Let $U(M')$ and $U(M'')$ denote the unsigned
movies underlying
$M'$ and $M''$ (i.e. the movies $M'$
and $M''$ without the signed points).
By Theorem~\ref{tmovie},
there is a sequence of unsigned
movies $M_1,M_2,\ldots,M_m$, such that
$M_1=U(M')$ and $M_m=U(M'')$, and such that $M_i$
differs from $M_{i-1}$ by one of
the Carter--Saito moves MM1--MM15.

By definition,
the moves SM1--SM15 are signed
analogues of the moves MM1--MM15.
Hence we can lift
the sequence $U(M')=M_1,M_2,\ldots,M_m$
movie by movie to
a sequence of signed movies $M'=M'_1,M'_2,\ldots,M'_m$,
such that $U(M'_i)=M_i$ and
such that $M'_i$ differs from $M'_{i-1}$
by one of the moves SM1--SM15, and possibly
some of the additional moves SM16--SM20.

Let us explain the role of the additional moves.
Assume we have already lifted the first $i-1$ movies
$M_1,M_2,\ldots,M_{i-1}$
to a sequence $M'_1,M'_2,\ldots,M'_{i-1}$.
Since $M_i$ differs from $M_{i-1}$ by one of
the moves MM1--MM15,
it should be possible to insert signed points
into $M_i$,
so that the result is a signed movie $M'_i$
differing from
$M'_{i-1}$ by
one of the signed moves SM1--SM15.
However, it might happen that the signed move
is not directly applicable,
for example because
$M'_{i-1}$ contains unwanted signed points,
lying in the region of the cobordism
where the signed move should take place.
In this case, it is helpful to
think of the unwanted points
as oriented curves on the cobordism,
as in the proof of part 1.
By performing an isotopy,
we can remove
these curves from %can be removed from
the relevant region of the cobordism.
Back on the level of movies, %presentations,
this isotopy becomes %of curves corresponds
a sequence of additional
moves SM16--SM18.
There are other cases,
where moves SM19--SM20 are needed
as well.

Now assume that we have lifted the entire sequence.
Then it remains to show that $M'_m$ %the last movie $M'_m$
and $M''$ are related by signed movie moves.
Being lifts of the movie $M_m$, the movies
$M'_m$ and $M''$ agree, except
possibly for the signed points.
Moreover,
since $M'_m$ and $M''$
represent equivalent framings,
the oriented curves
$c'_m$ and $c''$
coming from signed
points in $M'_m$ and in $M''$
must be homologous.
To complete the proof,
verify that any two homologous curves
on a link cobordism
can be related by a sequence of
local modifications, which
become the moves SM16--SM20
when seen on the level of movie presentations.
\end{proofth}

Let FM1--FM20 denote the framed movie moves, obtained
by replacing the signed points in SM1--SM20 by curls.
Note that FM19 and FM20 are identical
with FM1 and FM2.

\newpage
\begin{corollary} 1. Every framed link cobordism
has a framed movie presentation.
2. Two framed movies present isotopic framed link cobordisms
if and only if there is a sequence of framed movie moves
FM1--FM18 which takes one movie to the other.
\end{corollary}
\newpage\thispagestyle{empty}
%%%%%%%%%%%%%%%%% The colored Khovanov bracket %%%%%%%%%%%%%%%%%%%%%%%
\setcounter{footnotebuffer}{\value{footnote}}
\chapter{The colored Khovanov bracket}\label{ccoloredKhovanovbracket}
\setcounter{footnote}{\value{footnotebuffer}}
The {\it colored Jones polynomial} is the Reshetikhin--Turaev
invariant \cite{rt} for
oriented framed
links whose components are colored
by irreducible representations of $U_q(\sltwo)$.
If
all components are colored by the fundamental
representation $V_1$, the colored Jones polynomial
specializes to the ordinary Jones polynomial.
The colored Jones polynomial plays an important
role in the definition of the $\sltwo$ quantum
invariant for $3$--manifolds
and is conjecturally related to the hyperbolic
volume of the knot complement.

Khovanov \cite{kh:colored}
proposed two homology theories
which
have the colored Jones polynomial
as the Euler characteristic.

In this chapter, we focus on
Khovanov's first theory,
for the non--reduced
colored Jones polynomial.
We introduce a generalization
of Khovanov's theory,
which we call the {\it colored
Khovanov bracket}.
We show that this
theory is well--defined over $\Z$.
Further, we introduce
modifications of the
colored Khovanov bracket,
and study conditions under which
colored framed link cobordisms induce
chain transformations between
our modified colored Khovanov brackets.

\section{Colored Jones polynomial}
\noindent
Let $\fn=(n_1,\ldots,n_l)$ be a finite sequence of non--negative
integers. Let $(L,\fn)$ denote an oriented framed $l$--component
link $L$ whose $n_i$--th component is colored by the
$(n_i+1)$--dimensional irreducible representation $V_{n_i}$
of quantum $\sltwo$.
Given a sequence $\mathbf{m}=(m_1,\ldots,m_l)$ of
non--negative integers, we denote by $L^{\mathbf{m}}$ the
$\mathbf{m}$--cable of $L$.
When forming the $m_i$--cable of a component, we orient
the strands by alternating the original and the opposite
direction (starting with the original direction), so
that neighbored strands are always oppositely oriented.
The colored Jones polynomial $J(L,\fn)$ of the link $L$
can be expressed in terms of the Jones polynomial of
its cables:
\begin{equation}\label{fcoloredJones}
J(L,\fn)=\sum^{\lfloor {\fn}/2\rfloor}_{{\fk}=\mathbf{0}}
(-1)^{|\fk|} \left(\begin{array}{c} \fn-\fk \\ \fk\end{array}\right)
J(L^{\fn-2\fk})
\end{equation}
where $|\fk|=\sum_i  k_i$ and
$$
\left(
\begin{array}{c} \fn-\fk \\ \fk\end{array}\right)=\prod^l_{i=1}
\left(
\begin{array}{c} n_i-k_i \\ k_i\end{array}\right)\, .
$$
In \eqref{fcoloredJones}
the sum ranges over all $\fk=(k_1,\ldots,k_l)$
such that $0\leq k_i\leq\lfloor n_i/2\rfloor$
for all $i$.
Formula~\eqref{fcoloredJones} is a consequence
of the following relation,
which holds in the representation ring of $U_q(\sltwo)$
(for generic $q$), and which can be proved inductively
using $V_n\otimes V_1\cong V_{n+1}\oplus V_{n-1}$:
$$
V_n=\sum^{\lfloor n/2\rfloor}_{k=0}
(-1)^k \left(\begin{array}{c} n-k \\ k\end{array}\right)
V_1^{\otimes(n-2k)}\; .
$$
Note that for $\fn=(1,\ldots,1)$, we have
$J(L,\fn)=J(L)$.

\subsection[Graph $\G_{\fn}$]{Graph $\G_{\fn}$.}
The binomial coefficient
$\left(\begin{array}{c} n-k \\ k\end{array}\right)$
equals the number of ways to select $k$ pairs of neighbors from $n$
dots placed on a vertical line, such that each dot appears in at most
one pair. Analogously,
$\left(\begin{array}{c} \fn-\fk \\ \fk\end{array}\right)$
is the number of ways to select $\fk$ pairs of neighbors
on $l$ lines. We call such a selection of $\fk$ pairs a
$\fk$--{\it pairing}. Given a
$\fk$--pairing $\fs$, we denote by $D^{\fs}$
the cable diagram containing only components corresponding
to unpaired dots. Hence $D^{\fs}$ is isotopic to
$D^{\fn-2\fk}$.

\begin{figure}[H]
\centerline{\epsfysize=8cm \epsffile{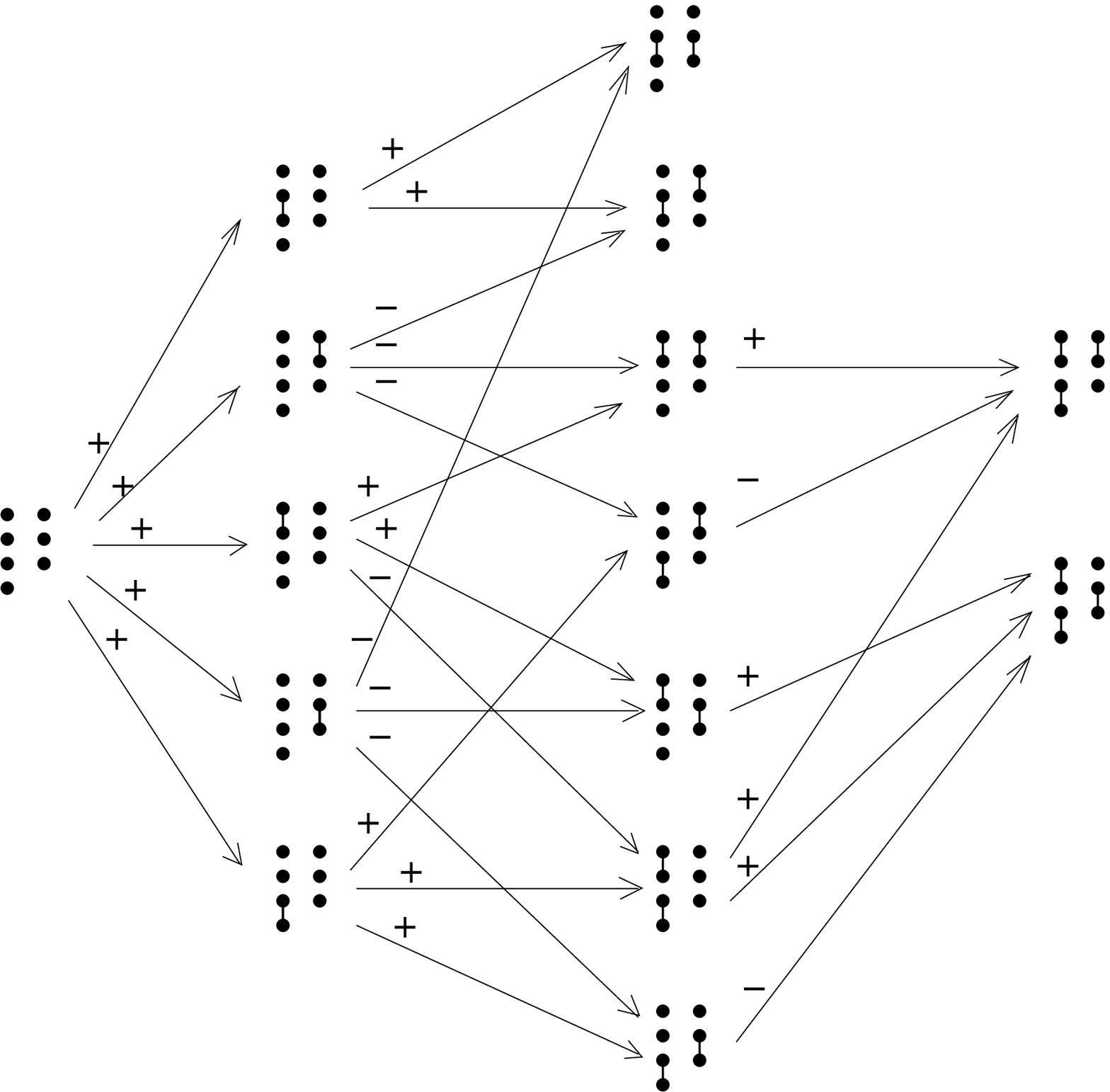}}

\caption{The graph $\G_{4,3}$.}
\label{figgraph}
\end{figure}

Let $\G_\fn$ be the graph, whose vertices correspond to $\fk$--pairings.
Two vertices of $\G_\fn$ are connected by an edge if the corresponding
pairings can be related to each other by adding/removing  one pair of
neighboring points. The ``degree'' of a vertex labeled by a
$\fk$--pairing is equal $|\fk|$. The edges are directed
towards increasing of degrees (see Figure~\ref{figgraph}).

\section{Colored Khovanov bracket}
\noindent
Let $D$ be a diagram of an oriented framed link $L$
(framed by the blackboard framing) and let $\fn=(n_1,\ldots,n_l)$
be a coloring of the components of $L$ by non--negative integers.
To $(D,\fn)$ we associate a complex $\Kh(D,\fn)$ in the category
$\Mat(\gKobdh)$. The construction goes as follows:

At each vertex of the graph $\G_{\fn}$ labeled by a
$\fk$--pairing $\fs$ we put the complex $\Kh(D^{\fs})$ defined
as in \eqref{fKhdef}, but viewed as an object of the
homotopy category $\gKobdh$. With an edge $e$ of $\G_{\fn}$
connecting a $\fk$--pairing $\fs$ to a $\fk'$--pairing
$\fs'$, we associate a morphism
$A_e:\Kh(D^{\fs})\rightarrow\Kh(D^{\fs'})$
in the category $\gKobdh$, as follows.
Let $C$ and $C'$ denote the two neighbored strands of
the cable $L^{\fs}$ which form a pair in $\fs'$ but
not in $\fs$.
Consider a standard annulus glued between
$C$ and $C'$ (i.e. such that $C$ and $C'$
are its boundary components, see \cite{kh:colored}).
Assume that $L^{\fs}$ is embedded in
$\R^3\times[0,1]$ as $L^{\fs}\times\{0\}$, and
that the interior of the annulus is pushed
into the interior of $\R^3\times[0,1]$.
Let $S_e$ denote the link cobordism
from $L^{\fs}$ to $L^{\fs'}$
which is given by the annulus on
$C$ and $C'$
and by the identity cobordism
on all other strands of $L^{\fs}$.
By
Section~\ref{sfunctoriality},
$S_e$
induces a morphism
$\Kh(S_e):\Kh(D^{\fs})\rightarrow\Kh(D^{\fs'})$
in $\Kobdh$, which is well--defined
up to sign. Define $A_e:=\Kh(S_e)$ to
be this morphism. Note that
$A_e$ is graded of Jones degree zero because the
Euler characteristic of an annulus is zero.

The sign of $A_e$ depends on the choice of
the movie presentation for the annulus, but
for any choice of movie presentations,
the squares of $\G_{\fn}$ commute
up to sign, because cobordisms given by gluing
of annuli in a different order are isotopic.
We call a choice of signs for the morphisms $A_e$
{\it satisfactory} if all squares anticommute.

Given a satisfactory choice of signs, we define
a chain complex $\Kh(D,\fn)$ in the category
$\Mat(\gKobdh)$ as follows. The $k$--th
``chain space'' is given by
$$
\Kh(D,\fn)^k:=\bigoplus\Kh(D^{\fs})\;\in\Ob(\Mat(\gKobdh))
$$
where the sum ranges over all $\fk$--pairings
$\fs$ with $|\fk|=k$. The $k$--th differential
$d_{\fn}^k:\Kh(D,\fn)^k\rightarrow\Kh(D,\fn)^{k+1}$ is
given by $(d^k_{\fn})_{\fs',\fs}:=A_e$ whenever $\fs$
and $\fs'$ are connected by an edge $e$, and
$(d^k_{\fn})_{\fs',\fs}:=0$ otherwise
(here $\fs$ denotes a $\fk$--pairing with
$|\fk|=k$, and $\fs'$ denotes a $\fk'$--pairing
with $|\fk'|=k+1$).
Since all squares of $\G_{\fn}$ anticommute, we
get $d_{\fn}^{k+1}\circ d_{\fn}^k=0$.

\begin{lemma}\label{signs1}
There exists a satisfactory choice of signs making all
squares of $\G_{\fn}$ anticommute.
\end{lemma}
For the proof of Lemma~\ref{signs1}, we need the following auxiliary
observation.
\begin{claim}
Let $e_1,e_2,\ldots,e_m$ be a sequence of oriented edges in $\G_\fn$,
such that the starting point $e_{i+1}$ agrees with the
endpoint of $e_i$. Then the composition
$A:=A_{e_m}\circ\ldots\circ A_{e_2}\circ A_{e_1}$
is non--zero in $\gKobdh$.
More generally, $N$ times $A$
%$A_{e_m}\circ\ldots\circ A_{e_2}\circ A_{e_1}$
is non--zero for any integer $N\in\Z$, $N\neq 0$.
In particular, $A$ %$A_{e_m}\circ\ldots\circ A_{e_2}\circ A_{e_1}$
is not equal to its negative.
\end{claim}
\begin{proofc}
We only prove the first statement (i.e. that
$A\neq 0$), but
the more general statement follows
in exactly the same way.

Let us start with the case where $D$ is the diagram
of a knot colored by $n=2$. In this case, $\Gamma_\fn$ has a single edge
$e$, and we have to show that $A_e\neq 0$. Recall that $A_e$ is
given by an embedded annulus in $\R^3\times[0,1]$. Let $\bar{A}_e$
denote the chain transformation obtained by turning $A_e$ upside--down
(i.e. by reflecting $A_e$ along $\R^3\times\{1/2\}\subset \R^3\times[0,1]$).
Then $\bar{A}_e \circ A_e:\Kh(\emptyset)\rightarrow
\Kh(\emptyset)$ is induced by an embedded torus, which
is isotopic in $\R^4$ to
a trivially embedded torus. Using the (T) relation, we get
$\bar{A}_e \circ A_e=\pm 2 \Id$, where $\Id$ denotes the
identity morphism of $\Kh(\emptyset)$,
and hence $A_e\neq 0$
\footnote{
Here we use the following fact:
assume $D$ is any link diagram
and $N$ any non--zero integer.
Then $N$ times the
identity morphism of $\Kh(D)$
is non--zero in $\gKobdh$. To see this,
use e.g. that $\HC'(D)\neq 0$
(see Chapter~\ref{cRasmussenforlinks}).
}.

Similarly, if $D$ is a diagram of a knot colored by $n>2$,
we can pre- and postcompose $A_{e_m}\circ\ldots\circ A_{e_2}\circ
A_{e_1}$ with suitable cobordisms, in such a way that the result is either
isotopic to a trivially embedded torus in $\R^4$ or to the identity
cobordism of the knot. In both cases we get
$A_{e_m}\circ\ldots\circ A_{e_2}\circ A_{e_1}\neq 0$.

Finally, if $D$ represents a link with more than one component,
we can apply the above argument to the different components of the
link individually. If necessary, we can use
the (N)~relation to unlink the resulting
embedded tori from identity cobordisms. Details are left
to the reader.
\end{proofc}

\begin{proofle}{signs1}
Let us first show that we can make all squares commute.
We define a $1$--cochain $\zeta\in C^1(\G_\fn,\Z/2\Z)$
as follows. For any square $S\subset\G_\fn$, we put $\zeta(S)=0$ if
$S$ is commutative, and $\zeta(S)=1$ if $S$ is anticommmutative.
Note that $\zeta$ is well--defined
because of the above claim.
Using that $\Z/2\Z$ is a field, we can extend $\zeta$ to a
$1$--cochain. Now it is easy to see that
all squares of $\G_\fn$ become commutative
if we replace $A_e$ by $(-1)^{\zeta(e)}A_e$.

Once all squares commute, we can make
them anticommute as follows.
For each edge $e$, connecting two pairings $\fs$ and $\fs'$,
we multiply the morphism $A_e$ by $(-1)^{(\fs,\fs')}$, where
$(\fs,\fs')$ denotes the number of pairs in $\fs$, which lie
either on the same vertical line as unique pair in
$\fs'\backslash\fs$ and above that pair, or on one of the
vertical lines to the right of that pair (see Figure~\ref{figgraph}).
\end{proofle}

\begin{lemma}\label{signs2}
Different satisfactory choices of signs lead to isomorphic complexes.
Moreover, for any two satisfactory choices of signs there is a preferred
isomorphism between the corresponding complexes.
\end{lemma}
\begin{proof}
Consider two choices of signs, given by two $1$--cochains $\zeta$ and $\zeta'$
as in the proof of Lemma~\ref{signs1}. If both choices of signs are
satisfactory, we must have $\zeta(S)=\zeta'(S)$ for all squares
$S\subset\G_\fn$. Since the space $Z_1(\G_\fn,\Z/2\Z)$ of
$1$--cycles of $\G_\fn$ is
generated by squares, $\zeta$ and $\zeta'$ must coincide on
$Z_1(\G_\fn,\Z/2\Z)$, and therefore $\zeta-\zeta'=\delta\g$ for a
$0$--chain $\g\in C^0(\G_\fn,\Z/2\Z)$.
Now
note that for every edge $e$ of $\G_\fn$
with boundary $s-s'$, we have $\zeta(e)-\zeta'(e)=\g(s)-\g(s')$.
Therefore, the morphisms
$(-1)^{\gamma(s)}\Id_{\Kh(D^s)}$ define an isomorphism between
the complex associated to $\zeta$ and the complex associated to $\zeta'$.

To see that there is a preferred choice for the isomorphism between the
$\zeta$-- and the $\zeta'$--complex, observe that any two $0$--cochains
$\g$ as above must differ by a
$0$--cocycle. Since the space of $0$--cocycles of $\G_\fn$ is
isomorphic to $\Z/2\Z$, there are only two possible choices for $\g$.
The preferred $\g$ is the one which maps the left--most vertex of
$\G_\fn$ to $0$.
\end{proof}
Alternatively, Lemma~\ref{signs2} can
be proved by
constructing the preferred isomorphism explicitly,
by defining it to be the identity on the
left--most vertex of $\G_\fn$ and
then extending it arrow by arrow to the right.

Lemmas~\ref{signs1} and \ref{signs2} show that $\Kh(D,\fn)$
is well--defined up to canonical isomorphism.
We call $\Kh(D,\fn)$ the {\it colored Khovanov bracket}
of $(D,\fn)$.

\begin{remark}
By definition,
the colored Khovanov bracket is an element
of $\Kom(\Mat(\Komh(\gCobdl^3)))$, and hence a
``complex of complexes''. However it is not a
bicomplex, because the $A_e$ are just homotopy
classes of chain transformations (rather
than honest chain transformations).
We do not know whether it is possible to lift
$\Kh(D,\fn)$ to a bicomplex by choosing suitable
representatives for the homotopy classes $A_e$.
\end{remark}

\begin{theorem}\label{tcoloredkhovanovinvariant}
The isomorphism class of $\Kh(D,\fn)$ is an invariant of
the colored oriented framed link $(L,\fn)$.
\end{theorem}
\begin{proof}
If $(D,\fn)$ and $(D',\fn)$ represent
isotopic colored framed links $(L,\fn)$ and $(L',\fn)$,
then the cables $L^{\fs}$ and $L'^{\fs}$ are isotopic
as well.
In particular, the complexes
$\Kh(D^{\fs})$ and $\Kh(D'^{\fs})$
are isomorphic as objects of $\gKobdh$.
This shows that $\Kh(D,\fn)$ and $\Kh(D',\fn)$
are isomorphic on the level of objects.
The isotopy between $L^{\fs}$ and ${L'}^{\fs}$
extends to an isotopy between the annuli
appearing in the definition of the differentials.
Hence Theorem~\ref{tfunctor}
and Lemma~\ref{signs2} imply that
$\Kh(D,\fn)$ and $\Kh(D',\fn)$ are isomorphic
as complexes.
\end{proof}

Let $\C(D,\fn):=\F_{\Kh}(\Kh(D,\fn))$ and
$\C'(D,\fn):=\F_{\Lee}(\Kh(D,\fn))$.
The total graded Euler characteristic
of $\C(D,\fn)$ is defined by
$$
\tilde{\chi}_q(\C(D,\fn))
:=\sum_{k,i,j}(-1)^{k+i}q^j\dim_{\Q}(\C^{k,i,j}(D,\fn)\otimes\Q)
$$
where $k$,$i$ and $j$ respectively refer to the
homological grading of $\C(D,\fn)$,
the homological grading of the complexes $\C(D^{\fn-2\fk})$,
and the Jones grading of the complexes $\C(D^{\fn-2\fk})$.

\begin{theorem}
The total graded Euler characteristic
of $\C(D,\fn)$
is equal to the colored Jones polynomial $J(L,\fn)$.
\end{theorem}
\begin{proof}
We have
\begin{equation*}
\begin{split}
\tilde{\chi}_q(\C(D,\fn))
 &=\sum_{k,i,j}(-1)^{k+i}q^j\dim_{\Q}
 (\C^{k,i,j}(D,\fn)\otimes\Q)\\
 &=
 \sum_k(-1)^k\sum_{|\fk|=k}\sum_{\fs\in {I_\fk}}
 \chi_q(\C(D^{\fn-2\fk}))\\
 &=\sum^{\lfloor\fn/2\rfloor}_{\fk=\bf 0} (-1)^{|\fk|}
  \left(
 \begin{array}{c} \fn-\fk \\ \fk\end{array}\right)
 \chi_q(\C(D^{\fn-2\fk}))
\end{split}
\end{equation*}
where in the second line $I_{\fk}$ denotes
the set of all $\fk$--pairings.
Taking into account that
 $\chi_q(\C(D^{\fn-2\fk}))=J(L^{\fn-2\fk})$
and comparing with \eqref{fcoloredJones}
we get the result.
\end{proof}

\section{Modified colored Khovanov bracket}
\noindent
In the following, we assume that the
additional relation
$\raisebox{-0.2cm}{\psfig{figure=figs/ddot.eps,height=0.6cm}}=0$
is imposed on the category
$\Cobdl^3$.

\subsection[Modified differentials]{Modified differentials.}
Let us generalize the definition of $\Kh(D,\fn)$ as follows.
As before, we put  $\Kh(D^{\fn-2\fk})$ at  vertices  of $\G_\fn$
labeled by $\fk$--pairings. But
we modify the morphisms associated to edges of $\G_\fn$.
With an edge $e$ connecting $\fk$-- and $\fk'$--pairings
we associate the morphism
$$
A'_e:=\al A_e+\be\Ad_e\; ,
$$
where $\al,\be\in\Z$ are fixed integers integers
and where $\Ad_e:=A_e\circ X_e$.
The morphism $X_e$ will be defined below.
Note that
the sign of $A_e$ depends on the choice
of a movie presentation for the annulus,
but
the relative sign between $A_e$ and $\Ad_e$
is independent of any choice.
Given a satisfactory choice of signs,
the result is
a chain complex which we denote $\Kh(D,\fn)_{\al,\be}$.
Observe that $\Kh(D,\fn)_{1,0}=\Kh(D,\fn)$.
The morphism $X_e$ is graded of
Jones degree $-2$. Hence if $\be$ is non--zero,
then the morphisms $A'_e$ do not respect the
Jones degree anymore.

\subsection[The morphism $X_e$]{The morphism $X_e$.}\label{smorphismxe}
$X_e$ is defined as follows. Let $C_i$ and $C_{i+1}$ be
the two strands of the cable $D^{\fn-2\fk}$
which are annihilated by $A_e$,
i.e. which do not appear in $D^{\fn-2\fk'}$
anymore. For a point $P$ on $C_i$,
let $X_P$ denote
the endomorphism of $\Kh(D^{\fn-2\fk})$ induced
``multiplying'' with a dot at the
point $P$ (i.e. $X_P$ is induced by the
identity cobordism of $D^{\fn-2\fk}$,
decorated by a single dot, located
near the point $P$).
According to Lemma~\ref{ldotslide} (Section~\ref{sfunctoriality}),
$X_P$ changes its sign when $P$ slides across
a crossing. To fix the sign,
we checkerboard color the regions of $D^{\fn-2\fk}$,
such that the unbounded region is colored
white, and we define $\sigma(P):=+1$ or $\sigma(P):=-1$,
depending on whether
the region between $C_i$ and $C_{i+1}$, which lies next
to $P$, is black or white.
Now the product
$\sigma(P)X_P$ is
independent
of the choice of $P$.
We define $X_e:=\sigma(P)X_P$.
If $C_i$ and $C_{i+1}$ belong to the cable
of a component $K$ of the link represented by $D$,
we will also use the notation $X(K,i)$ for $X_e$.

\section{Towards functoriality}
\noindent
Let $\Cobf^4$ be the category of colored framed movie presentations.
The objects are diagrams of colored links and the morphisms
movie presentations of colored framed links,
i.e. sequences of colored link diagrams,
where between two consecutive diagrams one of the following
transformations occurs: FR1, R2 or R3, or a saddle, a cap or a cup.
Note that here we need to distinguish between two saddle moves:
a ``splitting'' saddle which splits one colored component into
two of the same color, and a ``merging'' saddle which
merges two components of the same color into one component.

We are interested in a construction of a functor
$$
\Kh_{\al,\be}:\Cobf^4\rightarrow \Kom(\Mat(\Kobdh))\; .
$$
Given two colored link diagrams
$(D,\fn)$ and $(D_0,\fn_0)$ which are related by a
Reidemeister move, a cap, a cup or
a saddle, we would like to associate
a chain transformation
$$
F:\Kh_{\al,\be}(D,\fn)\longrightarrow\Kh_{\al,\be}(D_0,\fn_0)\; .
$$
We can do this by specifiying ``matrix elements''
$$
F_{\fs_0,\fs}:\Kh_{\al,\be}(D^{\fs})\longrightarrow\Kh_{\al,\be}(D_0^{\fs_0})\;
$$
for all pairings $\fs$ of $\fn$
and all pairings $\fs_0$ of $\fn_0$.
For Reidemeister moves, we can take the matrix elements
implicit in the proof of Theorem~\ref{tcoloredkhovanovinvariant}.
Subsections~\ref{scupcap}, \ref{smsaddle} and \ref{sssaddle} are
devoted to the definition of
matrix elements corresponding to cap, cup and saddles.

Throughout this section, we assume the
additional relation
$\raisebox{-0.2cm}{\psfig{figure=figs/ddot.eps,height=0.6cm}}=0$.
Moreover, we assume that $2$ is made invertible, i.e.
that the morphism sets of $\Cobdl^3$ are tensored
by $\Z[1/2]$.
\subsection[Cup and cap]{Cup and cap.}\label{scupcap}
Let $(D,\fn)$ and $(D_0,\fn_0)$ be two colored
link diagrams which are related by a cup
cobordism. Assume that $D_0$ is the disjoint union
of $D$ with a trivial component $K=\bigcirc$, and
that $\fn_0$ restricts
to $\fn$ on $D$
and to an arbitrary color $n$ on $K$.
Given a pairing $\fs$ of $\fn$ and a pairing $\fs_0$ of $\fn_0$,
we define a morphism
$$
\iota_{\fs_0,\fs}:\Kh(D^{\fs})\longrightarrow\Kh(D_0^{\fs_0})
$$
as follows: $\iota_{\fs_0,\fs}$ is non--zero only if
the restriction of $\fs_0$ to $K$ is the
empty pairing (no pairs) and if
$\fs_0$ agrees with $\fs$ on all other components.
In this case, we put $\iota_{\fs,\fs_0}:=G\circ C$,
where $C:\Kh(D^{\fs})\rightarrow\Kh(D^{\fs_0})$ is the morphism
induced by a union of $n$ cups whose
boundaries are the $n$ strands of the $n$--cable of $K$, and
$G$ is the endomorphism of $\Kh(D^{\fs_0})$ defined by
$$
G:=\sum_{j=1}^n Y_j\circ Z_j\; .
$$
Here,
$Y_j$ is the composition of all morphisms $(\al \Id-\be X(K,i))/2$ for
$1\leq i\leq j$, and $Z_j$ is the composition of
all morphisms $(\al \Id+\be X(K,i))/2$ for $j<i\leq n$.
$\al$ and $\be$ are the same integers as in
the definition of the modified differentials.

Now let $(D,\fn)$ and $(D_0,\fn_0)$ be two colored link diagrams related
by a cap cobordism.
Assume that $D$ is the disjoint union of $D_0$ with a trivial
component $K$, and that $\fn$ restricts to $\fn_0$ on $D_0$
and to an arbitrary color $n$ on $K$.
We define
$$
\epsilon_{\fs_0,\fs}:\Kh(D^{\fs})\longrightarrow\Kh(D_0^{\fs_0})
$$
as follows:
$\epsilon_{\fs_0,\fs}$ is non--zero only if the restriction of $\fs$ to $K$ is
the empty pairing and if $\fs$ agrees with $\fs_0$ on all other
components. In this case, we define $\epsilon_{\fs_0,\fs}:=\bar{C}\circ G$
where $G$ is the endomorphism of $\Kh(D^{\fs})$ defined as above,
and $\bar{C}$ is the morphism induced by $n$ caps whose boundaries are the $n$
strands of the $n$--cable of $K$.

\subsection[Merging saddle]{Merging saddle.}\label{smsaddle}
Let $(D,\fn)$ and $(D_0,\fn_0)$ be two
colored link diagrams which are related
by a saddle merging two components $K_1$ and $K_2$ of $D$ into
a single component $K$ of $D_0$.
Assume that $\fn$ and $\fn_0$ restrict to a color $n$
on the components $K_1$, $K_2$ and $K$, and that they
are identical on all other components.

Let $\fs$ be a pairing of the $\fn$--cable of $D$,
and let $s_1$ and $s_2$
denote its restrictions to $K_1$
and $K_2$, respectively.
Let $s_1\cup s_2$ denote the union of $s_1$ and
$s_2$, i.e. the pairing of $n$ which consists of
all pairs which are contained in either $s_1$ or in
$s_2$ or in both. Given $\gamma,\delta\in\Z$
and a pairing $\fs_0$ of
the $\fn_0$--cable of $D_0$,
we define a morphism
$$
m^{\gamma,\delta}_{\fs_0,\fs}:\Kh(D^{\fs})\longrightarrow \Kh(D^{\fs_0})
$$
as follows. $m^{\gamma,\delta}_{\fs_0,\fs}$
is zero unless the following is satisfied:
\begin{itemize}
\item
$s_1$ and $s_2$ have no common dot (meaning that there is no dot
which belongs to a pair both in $s_1$ and in $s_2$),
\item
 $\fs_0$ is the pairing
which restricts to $s_1\cup s_2$ on $K$ and which agrees with $\fs$ on all other
components.
\end{itemize}
If the above conditions are satisfied, we put
$m^{\gamma,\delta}_{\fs_0,\fs}:=
F_3\circ F_2\circ F_1,$
where $F_1,F_2$ and $F_3$
are defined as follows.
\begin{itemize}
\item
$F_1$ is the endomorphism of $\Kh(D^{\fs})$ defined by
$F_1:=X_1\circ X_2$, where $X_1$ is the composition of all
$(\g \Id+\delta X(K_1,i))/2$ such that the dots
numbered $i$ and $i+1$
form a pair in $s_2$, and $X_2$ is the composition
of all $(\g \Id+\delta X(K_2,i))/2$ such that
the dots numbered $i$ and $i+1$ form a pair in $s_1$.
\item
Let $\fs'$ be the pairing of $\fn$ which
restricts to $s_1\cup s_2$ on both
$K_1$ and $K_2$ and which agrees
with the pairing $\fs$ on all other
components of $D$.
$F_2:\Kh(D^\fs)\rightarrow\Kh(D^{\fs'})$
is the morphism
induced by attaching annuli to the
strands of $K_1^{s_1}$ and $K_2^{s_2}$ according
to the following rule. If the two dots numbered
$i$ and $i+1$ form a pair in $s_2$, we
attach an annulus to the strands numbered
$i$ and $i+1$ in $K_1^{s_1}$.
Similarly, if the two dots numbered
$i$ and $i+1$ form a pair in $s_1$, we
attach an annulus to the strands numbered
$i$ and $i+1$ in $K_2^{s_2}$.
\item
$F_3:\Kh(D^{\fs'})\rightarrow\Kh(D^{\fs_0})$ is the morphism
obtained by merging each strand of $K_1^{s_1\cup s_2}$
with the corresponding strand of $K_2^{s_1\cup s_2}$
by a saddle cobordism.
\end{itemize}

The above construction mimics
a construction of Khovanov \cite{kh:colored}.
Khovanov's map $\psi$ corresponds to our morphism
$m^{0,2}_{\fs_0,\fs}$.
Note that $m^{0,\delta}_{\fs_0,\fs}$ is graded
of Jones degree $\deg(m^{0,\delta}_{\fs_0,\fs})=-n$,
where $n$ is the color of the components
$K_1$,$K_2$ and $K$.
We denote by $m^{\g,\delta}$ the collection
of all morphisms $m^{\g,\delta}_{\fs_0,\fs}$.

\subsection[Splitting saddle]{Splitting saddle.}\label{sssaddle}
Suppose the diagrams $(D,\fn)$ and $(D_0,\fn_0)$
are related by a saddle which splits a component
$K$ of $D$ into two components $K_1$ and $K_2$ of $D_0$.
Assume that the colorings
$\fn$ and $\fn_0$ are consistent
with each other, in the obvious sense.
Consider a pairing $\fs$ of
the $\fn$--cable of $D$
which restricts to a $k$--pairing $s$ on $K$.
Given $\g,\delta\in\mathbb{Z}$
and a pairing $\fs_0$ of the $\fn_0$--cable of $D_0$,
we define a morphism
$$
\D^{\g,\delta}_{\fs_0,\fs}:\Kh(D^{\fs})
\longrightarrow \Kh(D_0^{\fs_0})
$$
as follows.
$\D^{\g,\delta}_{\fs_0,\fs}$
is zero unless $\fs_0$ has the following properties:
\begin{itemize}
\item
 the restrictions $s_1$ and $s_2$ of
$\fs_0$ to $K_1$ and $K_2$ have no common dot,
\item
 the union of $s_1$ and $s_2$ is equal to $s$,
\item
 $\fs_0$ agrees with $\fs$ on all components of $D_0$ other
than $K_1$ and $K_2$.
\end{itemize}
If $\fs_0$ satisfies the above
properties, we define
$\D^{\g,\delta}_{\fs_0,\fs}:=
2^k\bar{F}_1\circ\bar{F}_2\circ\bar{F}_3$,
where $\bar{F}_1,\bar{F}_2$ and $\bar{F}_3$ are
the morphisms obtained
by turning the morphisms $F_1,F_2$ and $F_3$ of
Subsection~\ref{smsaddle}
upside down (i.e. by reflecting the
link cobordisms appearing in the definition
of $F_1,F_2$ and $F_3$
along the hyperplane $\R^3\times\{1/2\}$).

\subsection[Criteria for chain transformations]%
{Criteria for chain transformations.}
In this subsection, we give criteria under
which the matrix elements $m^{\g,\delta}_{\fs_0,\fs}$
and $\D^{\g,\delta}_{\fs_0,\fs}$ induce chain transformations.
To simplify the notation, we will drop the superscripts
$\g,\delta$ in $m^{\g,\delta}$ and $\D^{\g,\delta}$
and just write $m$ and $\D$.

Let us first consider the case of merging saddles.
Let $(D,\fn)$ and $(D_0,\fn_0)$ be two
colored link diagrams which are related by a merging
saddle, and let $d$ and $d_0$ denote the differentials
of $\Kh(D,\fn)_{\al,\be}$ and
$\Kh(D_0,\fn_0)_{\al,\be}$, respectively.
Let $\fs$ be a pairing of $\fn$, and let $\fs_0$ be
the pairing of $\fn_0$ which restricts to $s_1\cup s_2$
on $K$ and which agrees with $\fs$ on all other
components (here, $s_1$, $s_2$ and $K$ are
defined as in Subsection~\ref{sssaddle}).
For a pairing $\fs_0'$ of $\fn_0$, we wish to
compare the matrix elements $(d_0\circ m)_{\fs_0',\fs}$
and $(m\circ d)_{\fs_0',\fs}$.
Assume that at least one of these matrix elements
is non--zero.
This is only possible
if $s_1$ and $s_2$ have no common dot.
Moreover,
$\fs_0'$ must contain a unique pair $p$ which
does not appear in $\fs_0$, and otherwise be
identical with $\fs_0$.
We assume that $p$ lies on $K$
(for otherwise
$(d_0\circ m)_{\fs_0',\fs}=\pm(m\circ d_0)_{\fs_0',\fs}$
is trivially satisfied). Then we are in the
situation of \eqref{emerge}, where we have left away
all dots corresponding to strands on which
$(d_0\circ m)_{\fs_0',\fs}$
and $(m\circ d_0)_{\fs_0',\fs}$
agree trivially, and where
$d':=d_{\fs',\fs}$,
$d'':=d_{\fs'',\fs}$ and
$d_0':=(d_0)_{\fs_0',\fs_0}$.
Note that $p$ is the pair in
the upper right corner.
\begin{equation}\label{emerge}
\begin{split}
\epsfysize=4.8cm \epsffile{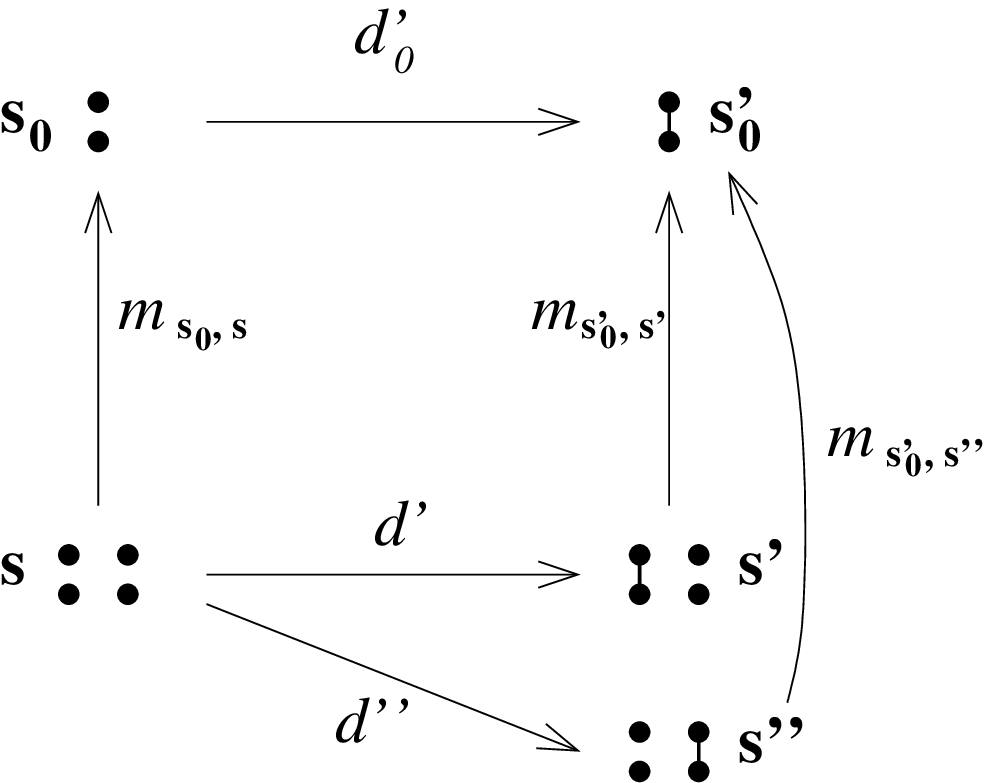}
\end{split}
\end{equation}

\begin{lemma}\label{lmsaddle}
Assume that
$d'_0\circ m_{\fs_0,\fs}=
\pm (m_{\fs_0',\fs'}\circ d'+m_{\fs_0',\fs''}\circ d'')$
for all diagrams as in \eqref{emerge}.
Then there is a $0$--cochain $\g\in
C^0(\G_{\fn},\mathbb{Z}/2\mathbb{Z})$ such that the morphisms
$F_{\fs_0,\fs}:=(-1)^{\g(\fs)}m_{\fs_0,\fs}$
determine a chain transformation between
$\Kh(D,\fn)_{\al,\be}$ and $\Kh(D_0,\fn_0)_{\al,\be}$, i.e.
such that $d_0\circ F=F\circ d$.
\end{lemma}
The proof of Lemma~\ref{lmsaddle}
is quite technical, so we skip it here
and instead refer to \cite{bw}.

Now assume that $(D,\fn)$ and $(D_0,\fn_0)$
are related by a splitting saddle.
Let $\fs$ be a pairing
of the $\fn$--cable of $D$ and let $\fs_0'$ be a pairing of the
$\fn_0$--cable of $D_0$, such that at least one of the morphisms
$(d_0\circ \D)_{\fs_0',\fs}$ and $(\D\circ d)_{\fs_0',\fs}$
is non--zero. Let $K$ denote the component of $D$ which is involved
in the saddle and let $s$ be the restriction of $\fs$ to $K$.
Similarly, let $K_1$ and $K_2$ be the components of $D_0$ which
are involved in the saddle and let $s_1'$ and $s_2'$ denote the
restrictions of $\fs_0'$ to $K_1$ and $K_2$.
Then every pair of $s$ must also appear in the
union $s_1'\cup s_2'$.
If $s_1'$
and $s_2'$ have a common pair, we are in the situation
of \eqref{esplit3}.

\begin{equation}\label{esplit3}
\begin{split}
\epsfysize=4.2cm \epsffile{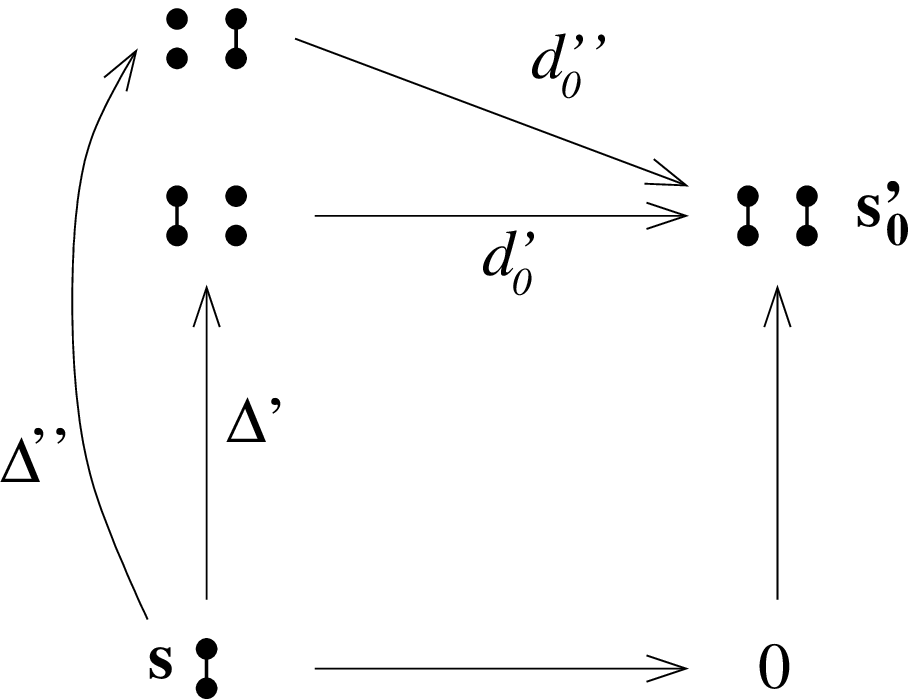}
\end{split}
\end{equation}

Now assume that $s_1'$ and $s_2'$ have no common
pair. Let $s_1$ and $s_2$ denote the intersections $s_1:=s\cap s_1'$
and $s_2:=s\cap s_2'$. Let $\fs_0$ denote the pairing of
the $\fn_0$--cable of $D_0$ which restricts to $s_1$ and $s_2$
on the components $K_1$ and $K_2$ and which agrees with $\fs$ on
all  other components of $D_0$. Then every pair of $\fs_0$ must also
be a pair of $\fs_0'$. Moreover, $\fs_0'$ has to contain a unique
pair $p$
which is not contained in $\fs_0$. We assume that $p$ belongs
to $K_1$ or $K_2$ (for otherwise
$(d_0\circ \D)_{\fs_0',\fs} =\pm(\D\circ d)_{\fs_0',\fs}$
is trivially satisfied).
If $p$ is disjoint from all pairs of $s_1\cup s_2$, we
are in the situation of \eqref{esplit1}, where $p$ is the pair
in the upper right corner.

\begin{equation}\label{esplit1}
\begin{split}
\epsfysize=3.5cm \epsffile{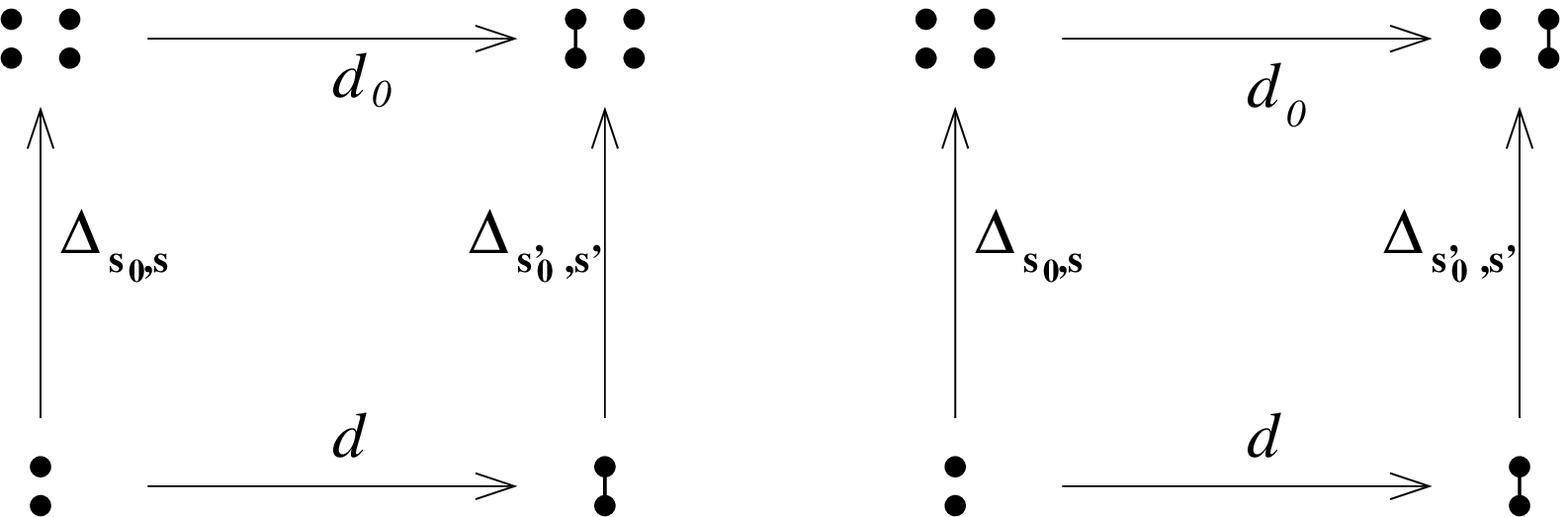}
\end{split}
\end{equation}

\noindent
It is also possible that $p$ has a common dot with
a pair of $s_1\cup s_2$. Examples of this case are
shown in \eqref{esplit2}.

\begin{equation}\label{esplit2}
\begin{split}
\epsfysize=3.9cm \epsffile{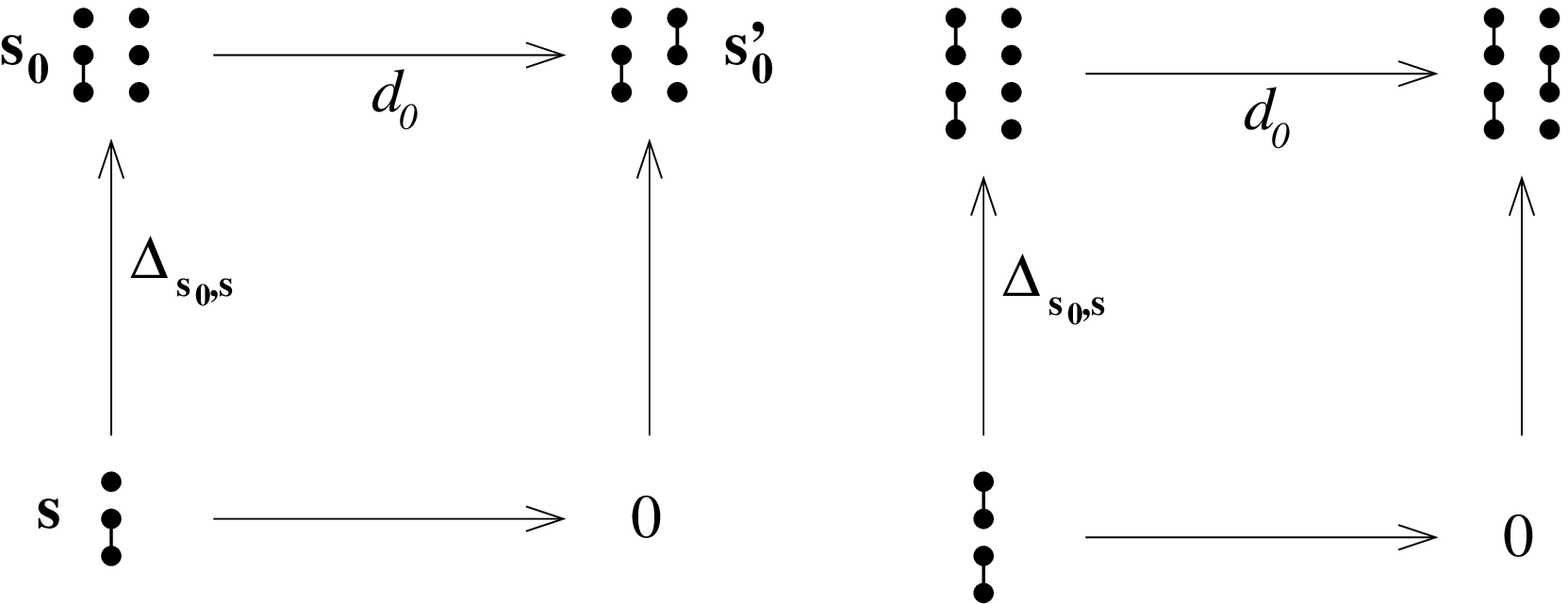}
\end{split}
\end{equation}

\begin{lemma}\label{lssaddle}
Assume that the squares of \eqref{esplit1} commute, up to sign,
and assume that $d_0\circ \D_{\fs_0,\fs}=0$
for all squares as in \eqref{esplit2}. Then there is a
$0$--cochain $\gamma\in C^0(\G_{\fn_0}, \Z/2\Z)$ such that the morphisms
$(-1)^{\g(\fs_0)}\D_{\fs_0,\fs}$
determine a chain transformation
between $\Kh(D,\fn)_{\al,\be}$ and $\Kh(D_0,\fn_0)_{\al,\be}$.
\end{lemma}
In Lemma~\ref{lssaddle}, no assumption has
to be made about the diagrams of \eqref{esplit3}.
Indeed,
if the squares of \eqref{esplit1} commute, then
the anticommutativity of the squares of $\G_{\fn_0}$
implies
$d_0'\circ\D'+d_0''\circ\D'=0$ for all diagrams
as in \eqref{esplit3}.

\subsection[Chain transformations]{Chain transformations.}

\begin{theorem}\label{tfunctorLeesaddle}
For $\al=\be=1$, the maps
$\F_{\rm Lee}(m_{\fs_0,\fs}^{1,1})$
and $\F_{\rm Lee}(\D_{\fs_0,\fs}^{1,1})$
induce chain transformations.
\end{theorem}

\begin{proofsk}
We have to show that for $\al=\be=1$,
the morphisms $\F_{\Lee}(m_{\fs_0,\fs}^{1,1})$
and $\F_{\Lee}(\D_{\fs_0,\fs}^{1,1})$
satisfy the conditions of Lemmas \ref{lmsaddle}
and \ref{lssaddle}.

We start with the proof of
$\F_{\Lee}(d_0\circ\D^{1,1}_{\fs_0,\fs})=0$
for the left square of \eqref{esplit2}.
Assume that the three dots in the
lower left corner of the square
are numbered from bottom to top from $i$ to $i+2$.
Moreover, assume that these dots
lie on a component $K$ of $D$,
and that the saddle cobordism
splits $K$ into components
$K_1$ and $K_2$. (In the left square of
\eqref{esplit2}, $K_1$ and $K_2$
correspond to the left and the right
column of dots in the upper left corner).
Let $s_2$ denote the restriction of $\fs_0$
to $K_2$, and
let $C_i$, $C_{i+1}$ and $C_{i+2}$ denote
the strands of $K_2^{s_2}$
corresponding
to the dots $i$, $i+1$ and $i+2$, respectively.
For $\al=\be=1$,
$d_0$ is given by
$$
A_{i+1}\circ(\Id+X(K_2,i+1))
$$
where $A_{i+1}$ is induced by
an annulus attached to the components
$C_{i+1}$ and $C_{i+2}$ of $K_2^{s_2}$.
Similarly, $\D^{1,1}_{\fs_0,\fs}$
is given by some saddle cobordisms,
composed with
$$
(\Id+X(K_2,i))\circ \bar{A}_i
$$
where $\bar{A}_i$ is induced by
an annulus attached to $C_i$ and $C_{i+1}$.
We can replace $X(K_2,i)$
by $-X(K_2,i+1)$ because we can
move the point $P$ used in the definition of
$X(K_2,i)=\sigma(P)X_P$ across the annulus.
The minus sign appears because
of the definition of $\sigma(P)$.
Summarizing, we see that
$d_0\circ\D^{1,1}_{\fs_0,\fs}$
factors through
$$
(\Id+X(K_2,i+1)))\circ
(\Id-X(K_2,i+1))
=\Id-X(K_2,i+1)^2\, .
$$
Now recall that
in Lee's Frobenius algebra, we have the relation
$X^2={\bf 1}$.
Hence $\F_{\Lee}(\Id-X(K_2,i+1)^2)=0$
and therefore
$\F_{\Lee}(d_0\circ\D^{1,1}_{\fs_0,\fs})=0$.
So the left square of \eqref{esplit2}
commutes. The proof for the right square is
analogous.

To show that the squares of \eqref{esplit1}
commute up to sign,
one has to apply isotopies, the (N) relation
and the defining relations for Lee's functor
to the cobordisms corresponding to
$(d_0\circ \D^{1,1})_{\fs_0',\fs}$,
$(\D^{1,1}\circ d)_{\fs_0',\fs}$.
Similarly, to prove that the assumption
of Lemma~\ref{lmsaddle} is satisfied,
on has to apply the same relations to
$(d_0\circ m^{1,1})_{\fs_0',\fs}$,
and $(m^{1,1}\circ d)_{\fs_0',\fs}$.
\end{proofsk}

\begin{theorem}\label{tfunctorLeecup}
For $\al=\be=1$, the maps
$\F_{\Lee}(\e_{\fs_0,\fs})$
and $\F_{\Lee}(\iota_{\fs_0,\fs})$
associated to caps and cups induce
chain transformations.
\end{theorem}

\begin{proof}
The case of caps is easy, so we only discuss the case of cups.
Let $D$ and $D_0$ be two link diagrams which are related by a
cup cobordism, i.e. $D_0=D\sqcup K$ for a trivial
component $K$.

Let $d$ and $d_0$ denote the differentials
of $\Kh(D,\fn)_{\al,\be}$ and
$\Kh(D_0,\fn_0)_{\al,\be}$, respectively.
We write $d_0$ as $d_0=d_0'+d_0''$, where $d_0'$
denotes the sum of all $A'_e$ which contract
a pair on $K$, and $d_0''$
denotes the sum of all $A'_e$ which contract
a pair on
one of the other components of $D_0$.
Then
$d_0''\circ\iota=\iota\circ d$,
so we must show that
$\F_{\Lee}(d_0'\circ\iota)=0$.

For $\al=\be=1$,
$d_0'$ is a sum of morphisms
$$
A'_i=A_i\circ (\Id +X(K,i))=A_i\circ (\Id-X(K,i+1))\, ,
$$
where $A_i$ is induced by
an annulus glued to the strands $i$ and $i+1$
of the cable of $K$.
$\iota$ is equal to
$G\circ C$, where $G=\sum_{j=0}^nY_j\circ Z_j$.
Using $X^2={\bf 1}$ and
the definitions of $Y_j$ and $Z_j$ with $\al=\be=1$,
we get
$$
\F_{\rm Lee}((\Id+X(K,i))\circ Y_j)=0
$$
for $i\leq j$, and
$$\F_{\rm Lee}((\Id-X(K,i+1)) \circ Z_j)=0
$$
for $i\geq j$.
Hence $\F_{\rm Lee}(A'_i\circ Y_j\circ Z_j)=0$
for all $i,j$, and therefore $\F_{\rm Lee}(d_0'\circ\iota)=0$.
\end{proof}

\noindent
Assume we can make the definition of the
chain transformations in Theorems \ref{tfunctorLeesaddle}
and \ref{tfunctorLeecup} canonical, i.e.
independent of any sign choices.
Then $\F_{\Lee}(\Kh_{1,1}(D,\fn))$ extends
to a well--defined functor
$\F_{\Lee}\circ\Kh_{1,1}:\Cob^4_f\rightarrow
\Kom(\Komh(\Z\mbox{-}\MOD))$.
Let $\Cobfi^4$ denote the quotient of
$\Cobf^4$ by framed movie moves.
We expect
\begin{conjecture}
The functor $\F_{\Lee}\circ\Kh_{1,1}$ descends to
a functor $\F_{\Lee}\circ\Kh_{1,1}:\Cobfi^4\rightarrow
\Kom_{/\pm h}(\Komh(\Z\mbox{-}\MOD))
$.
\end{conjecture}

In \cite{bw}, we also defined
chain transformations for
$\F_{\Kh}(\Kh(D,\fn)_{0,1})$:

\begin{theorem}\label{tfu}
For $\al=0$, $\be=1$, the maps
$\F_{\Kh}(m_{\fs_0,\fs}^{0,1})$,
$\F_{\Kh}(\D_{\fs_0,\fs}^{0,1})$,
$\F_{\Kh}(\e_{\fs_0,\fs})$
and $\F_{\Kh}(\iota_{\fs_0,\fs})$
induce chain transformations.
\end{theorem}
The proof of Theorem~\ref{tfu} is analogous
to the proofs of Theorems \ref{tfunctorLeesaddle}
and \ref{tfunctorLeecup}.

\begin{remark}
We do not know how to extend the original  colored
Khovanov bracket $\Kh(D,\fn)=\Kh(D,\fn)_{1,0}$
to a functor.
%from $\Cob_f^4$ to the category
%$\Kom(\Komh(\Z\mbox{-}\MOD))$.
For the original colored Khovanov bracket,
the morphisms $m_{\fs_0,\fs}^{0,2}$ induce
chain transformations
(cf. \cite{kh:colored}),
but there is no choice of $\g,\delta$ for which
the $\D_{\fs_0,\fs}^{\g,\delta}$
induce chain transformations.
\end{remark}

%%%%%%%%%%%%%%%%%%%%%%%%%% Bibliography %%%%%%%%%%%%%%%%%%%%%%%%%%%%

\end{document}